\documentclass[letterpaper]{article}

\usepackage{uai2019}
\usepackage[margin=1in]{geometry}

\usepackage{times}

\usepackage{natbib} 
\usepackage[colorlinks = true]{hyperref}       
\usepackage{url}            
\usepackage{booktabs}       
\usepackage{nicefrac}       
\usepackage{microtype}      

\usepackage{epsfig,amsmath,amssymb,amsfonts,amstext,amsthm,mathrsfs,dsfont}
\usepackage{latexsym,graphics,epsf,epsfig,psfrag}
\usepackage{color}
\usepackage{subfigure} 
\usepackage{caption}
\usepackage{graphicx}
\usepackage{epstopdf}
\usepackage{algorithm}
\usepackage{algorithmic} 
\usepackage{enumitem} 
\usepackage{cleveref} 
\hypersetup{linkcolor =red, citecolor = blue}

\usepackage{float}

\newcommand{\w}{\mathbf{w}}
\newcommand{\x}{\mathbf{x}}
\newcommand{\y}{\mathbf{y}}
\newcommand{\vb}{\mathbf{v}}
\newcommand{\s}{\mathbf{s}}
\newcommand{\RR}{\mathbb{R}} 

\newcommand{\norml}[1]{\| #1 \|}
\newcommand{\normlarge}[1]{\left\Vert #1\right\Vert}

\newcommand{\numleqslant}[1]{\overset{\text{(#1)}}{\leqslant}}
\newcommand{\numequ}[1]{\overset{\text{(#1)}}{=}}
\newcommand{\numgeq}[1]{\overset{\text{(#1)}}{\geqslant}}

\newtheorem{Theorem}{Theorem}
\newtheorem{Proposition}{Proposition} 
\newtheorem{lemma}[Theorem]{Lemma} 
\newtheorem{Assumption}{Assumption}

\newcommand\numberthis{\addtocounter{equation}{1}\tag{\theequation}}

\DeclareMathOperator*{\argmin}{argmin}

\allowdisplaybreaks

\title{Cubic Regularization with Momentum for Nonconvex Optimization} 

\author{} 

\author{ {\bf Zhe Wang} \\
EECS Dept. \\
Ohio State University \\ 
wang.10982@osu.edu \\
\And
{\bf Yi Zhou}  \\
EECS Dept.         \\
Duke University \\
yi.zhou610@duke.edu \\
\And
{\bf Yingbin Liang}   \\
EECS Dept. \\
Ohio State University \\
liang.889@osu.edu\\ 
\And
{\bf Guanghui Lan}   \\
ISyE  Dept. \\
Georgia Institute of Technology \\
george.lan@isye.gatech.edu\\ 
}

\begin{document}  
\maketitle
\begin{abstract}
  Momentum is a popular technique to accelerate the convergence in practical training, and its impact on convergence guarantee has been well-studied for first-order algorithms. However, such a successful acceleration technique has not yet been proposed for second-order algorithms in nonconvex optimization.
  In this paper, we apply the momentum scheme to cubic regularized (CR) Newton's method and explore the potential for acceleration. Our numerical experiments on various nonconvex optimization problems demonstrate that the momentum scheme can substantially facilitate the convergence of cubic regularization, and perform even better than the Nesterov's acceleration scheme for CR. Theoretically, we prove that CR under momentum achieves the best possible convergence rate to a second-order stationary point for nonconvex optimization. Moreover, we study the proposed algorithm for solving problems satisfying an error bound condition and
  establish a local quadratic convergence rate. Then, particularly for finite-sum problems, we show that the proposed algorithm can allow computational inexactness that reduces the overall sample complexity without degrading the convergence rate.
  
\end{abstract}

\section{INTRODUCTION} 

In the era of machine learning, deep models such as neural networks have achieved great success in solving a variety of challenging tasks. However, training deep models is in general a difficult task and traditional first-order algorithms can easily get stuck at sub-optimal points such as saddle points, which have been shown to bottleneck the performance of practical training \citep{Dauphin_2014}. Motivated by this, there is a rising interest in designing algorithms that can escape saddle points in general nonconvex optimization, and the cubic regularization (CR) Newton's method is such a type of popular optimization algorithm.

More specifically, consider the following generic nonconvex optimization problem.
\begin{align}
\min_{\x \in \mathbb{R}^d} f(\x), \label{Object_function}
\end{align}
where $f: \RR^d \to \RR$ is a twice-differentiable and nonconvex function. The CR algorithm \citep{Nesterov2006} takes an initialization $\x_0 \in \mathbb{R}^d$, a proper parameter $M>0$, and generates a sequence $\{\x_k\}_k$ for solving \cref{Object_function} via the following update rule.
\begin{align*} 
\s_{k+1} &= \argmin_{\s\in \RR^d} \nabla f(\x_k)^\top\s     + \frac{1}{2} \s^\top \nabla^2 f(\x_{k}) \s + \frac{M}{6} \norml{\s}^3, \\
\x_{k+1} &= \x_{k} + \s_{k+1}. 
\end{align*}
Intuitively, the main step of CR solves a cubic minimization subproblem that is formulated by the second-order Taylor expansion at the current iterate with a cubic regularizer. Such a cubic subproblem can be efficiently solved by many dedicated solvers \citep{Cartis2011a,Carmon2016b,Agarwal2017} that induce a low overall computation complexity (see \Cref{sub_problem_solver} for further elaboration). By exploiting second order information (i.e., gradient and Hessian) of the objective function, the CR algorithm has been shown to produce a solution $\x$ that satisfies the $\epsilon$-second-order stationary condition, i.e., 
\begin{align}
\norml{\nabla f(\x) } \leqslant \epsilon \quad \text{and} \quad \lambda_{\min} \big(\nabla^2 f(\x)\big) \geqslant - \sqrt{\epsilon},
\end{align} 
where $\lambda_{\min} \big(\nabla^2 f(\x)\big)$ denotes the minimum eigenvalue of the Hessian $\nabla^2 f(\x)$. Unlike the first-order stationary condition (i.e., $\|\nabla f(\x)\| \leqslant \epsilon$) which does not rule out the possibility of converging to a saddle point, the second-order stationary condition requires the corresponding Hessian to be almost positive semidefinite and hence can avoid convergence to strict saddle points (i.e., at which Hessian has negative eigenvalue). In particular, a variety of nonconvex machine learning problems such as phase retrieval \citep{Sun2017}, dictionary learning \citep{Sun2015} and tensor decomposition \citep{Ge2015} have been shown to have only strict saddle points. Therefore, CR is guaranteed to escape all the saddle points and converge to a local minimum in solving these problems.

While most existing studies on the CR algorithm focus on reducing the computation complexity by various sampling schemes, e.g., mini-batch sampling \citep{Xu2017}, sub-sampling \citep{kohler2017}, variance-reduced sampling \citep{Wang2018,Zhou2018}, less attention has been paid to the design of new schemes for accelerating CR. The only exception is \cite{Nesterov2008}, where an acceleration   scheme was proposed for CR, but has been shown to achieve a faster convergence rate than CR only for {\em convex} problems. Such an accelerated scheme consists of hyperparameters that are fine-tuned in the context of convex optimization, and hence may not guarantee to produce a second-order stationary solution in {\em nonconvex} optimization. There does not exist any accelerated CR algorithm that has provable convergence for nonconvex optimization.
Therefore, the aim of this paper is to design a momentum-based scheme for CR with provable second-order stationary convergence guarantee for nonconvex optimization as well as yielding faster convergence in practical scenarios.

\subsection{OUR CONTRIBUTIONS}
Our major contribution lies in proposing the first CR algorithm that incorporates momentum technique, which has provable convergence guarantee to a second-order stationary point in nonconvex optimization. We also performed a comprehensive study of this algorithm from various aspects both in theory and experiments to demonstrate the appealing attributes of the proposed algorithm. Our specific contribution are listed as follows.
\begin{itemize}
\item We propose a CR type algorithm with momentum acceleration (referred to as CRm), which includes a cubic regularization step, a momentum step for acceleration and a monotone step. The momentum step introduces negligible computation complexity compared to that of the cubic regularization step in original CR, but can provide substantial advantage of acceleration. 

\item We establish the global convergence of CRm to a second-order stationary point in nonconvex optimization. The corresponding convergence rate is as fast as that of CR in the order-level, which is the best one can expect for nonconvex optimization. Our experiments demonstrate that CRm  substantially outperforms CR as well as Nesterov's accelerated CR (which does not have guaranteed performance for nonconvex optimization). 

\item We also show that CRm enjoys the local quadratic convergence property under a local error bound condition, which establishes the advantage of the second-order algorithms than the first-order algorithms in nonconvex optimization. 

\item We further show that the inexact variant of CRm significantly improves the computational complexity without losing the convergence rate. We also study the finite-sum problem, where we implement the inexact CRm via a subsampling approach, and established the total Hessian sample complexity to guarantee the convergence with high probability.
\end{itemize}

On the core of our proof technique, we rely on the delicate  design of the adaptive momentum parameter in \cref{li_2}, and the monotone step in the algorithm, which makes it possible to establish the convergence result under nonconvex optimization but with momentum acceleration. 
To the best of our knowledge, there is no  result on accelerated CR type algorithms that have such good convergence property, or even the convergence property under nonconvex optimization.

\subsection{RELATED WORKS}
\textbf{Escaping saddle points:} A number of algorithms have been proposed to escape saddle points in order to find local minima. In general, There are three lines of research. It has been shown that with random perturbation, gradient descent algorithm \citep{Jin2017a}, the stochastic gradient descent \citep{Ge2015},   the zero-th order method \citep{Jin2018},  and the accelerated gradient descent \citep{Jin2017}  can escape saddle points. The gradient descent  has also been incorporated with the negative curvature descent in \cite{Carmon2016a,Liu2017,YangT2017a} in order to converge  to the second-order stationary points. Furthermore, the cubic regularized (CR) algorithm, which first appeared in \cite{Griewank1982}, has been shown by \cite{Nesterov2006} to converge  to the second-order stationary points. \cite{Cartis2011a,Cartis2011b} then proposed an adaptive CR method with an approximate sub-problem solver. \cite{Agarwal2017} established an efficient sub-problem solver for CR by using the Hessian-vector product technique, and \cite{Carmon2016b} showed that gradient descent can efficiently solve the sub-problem in CR. \cite{Yi2018} studied the CR algorithm in nonconvex optimization. This paper further accelerates the CR algorithm with momentum and establishes its convergence rate to a second-order stationary point.

\textbf{Algorithms with momentum for nonconvex optimization:} 
\cite{Ghadimi2016,Li2015} proposed   accelerated gradient descent type of algorithms for nonconvex optimization, which are guaranteed to converge as fast as gradient descent for nonconvex problems.  
\cite{yao2017} proposed an efficient  accelerated proximal gradient descent algorithm for nonconvex problems, which requires only one proximal step in each iteration as compared to the requirement of two proximal steps in each iteration in the algorithm proposed in \cite{Li2015}. Then \cite{li2017} analyzed the algorithm in \cite{yao2017}  under the KL condition. While the existing studies analyzed only convergence to first-order stationary points, this paper proposes the CR algorithms with momentum that converge  to a second-order stationary point.

\textbf{Inexact CR algorithms:}
 To reduce the computational complexity for the CR type of algorithms, various inexact Hessian and gradient approaches were proposed. In particular, \cite{Saeed2017} studied the inexact Hessian CR and accelerated CR for convex optimization, where the inexact level is fixed during iterations. \cite{JinChi2017cubic} studied an inexact CR for nonconvex optimization, which allows both the gradient and Hessian to be inexact.  Alternatively, \cite{Cartis2011a,Cartis2011b} studied the  inexact Hessian CR for nonconvex optimization, where the inexact condition is adaptive during iterations.  \cite{Jiang2017} studied a unified  scheme of inexact accelerated adaptive CR  and gradient descent for convex optimization. 
Furthermore, \cite{kohler2017} proposed a subsampling CR (SCR) that adaptively changes the sample batch size to guarantee the inexactness condition in \cite{Cartis2011a,Cartis2011b}, \cite{Wang2018inexact} relaxed the inexact condition in \cite{kohler2017,Cartis2011a,Cartis2011b}, and \cite{Xu2017} proposed uniform and non-uniform sampling algorithms with fixed inexactness for nonconvex optimization. \cite{Wang2018,Zhou2018} proposed stochastic variance reduced subsampling CR algorithms. This paper establishes the convergence rate for the inexact scenarios of the proposed CR algorithm  with momentum.

\textbf{Local quadratic convergence:} 
The   Newton's method and cubic regularized algorithm have been shown  to converge quadratically to the global minimum under the strongly convex condition in \cite{Nesterov2006,Nesterov2008}, respectively. Furthermore, various Newton-type algorithms, i.e., the Levenberg-Marquardt method  \citep{Yamashita2001,Fan2005}, the regularized Newton method \citep{Li2004}, the regularized proximal Newton's method \citep{Yue2016}, and the CR algorithm \citep{Yue2018}, have been shown to have the local quadratic convergence under the more relaxed local error bound condition. This paper further establishes such a property for the proposed CR with momentum algorithm.

\section{CRm: CUBIC REGULARIZATION WITH MOMENTUM}  
In this section, we propose a CR-type algorithm that adopts a momentum scheme (referred to as CRm). The algorithm steps of CRm are summarized in \Cref{Adaptive_version}.   

 At each iteration, the proposed CRm conducts a cubic step (\cref{I_1}),  a momentum step (\cref{li_2,li_1}), and a monotone step (\cref{li_0}). In particular, the cubic step solves a subproblem of the second-order Taylor expansion with a cubic regularizer at the current iterate $\x_k$.  The cubic step can be implemented efficiently by adopting the solver based on the Hessian-vector product approach (see \Cref{sub_problem_solver} for details). The momentum step is an extrapolation step that aims to accelerate the algorithm. We note that the momentum step requires very little additional computation  compared to the cubic step, but offers substantial advantage for accelerating the algorithm.  The monotone step  chooses the next iteration point between the cubic step and the momentum step to achieve the minimum function value. This guarantees that the algorithm outputs a desirable monotonically decreasing function value sequence, and helps to establish the convergence guarantee under nonconvex optimization.  

\begin{algorithm}  
	\caption{CRm}
	\label{Adaptive_version}
	\begin{algorithmic}[1] 
		\STATE {\bfseries Input:} Initialization $\x_0 =  {\y}_0\in \mathbb{R}^d,\rho < 1, M > L_2$ 
		\FOR{$k=0,1,\dots $}
		\STATE {\bfseries Cubic step:}   
		{ 
		\begin{align*}  { \s}_{k+1} &=\argmin_{\s} \nabla f(\x_k)^\top\s   + \frac{1}{2} \s^\top \nabla^2 f(\x_{k}) \s \\
		&\qquad \qquad \quad + \frac{M}{6} \norml{\s}^3   \\
		\y_{k+1} &= \x_{k} + \s_{k+1} \numberthis \label{I_1}
			\end{align*} 
		}
		\vspace{-.3cm}
		\STATE {\bfseries Momentum step:}   
		\begin{align} 
			\beta_{k+1} &=    \min\{\rho , \norml{\nabla f( {\y}_{k+1})}, \norml{ {\y}_{k+1} - \x_k}  \} \label{li_2}   \\
			 {\vb}_{k+1} &=  {\y}_{k+1} + \beta_{k+1} ( {\y}_{k+1} -  {\y}_{k}) \label{li_1} 
		\end{align}
		\vspace{-.3cm} 
		\STATE \textbf{Monotone Step:}
		\vspace{-.3cm}
		\begin{align}
		\x_{k+1}= \argmin_{\x \in \{ {\y}_{k+1},  {\vb}_{k+1}\}} f(\x) \label{li_0}
		\end{align}  
		\ENDFOR
	\end{algorithmic} 	 
\end{algorithm}
We further highlight the ideas in the design of  CRm. First, we choose the momentum in the direction of $\y_{k+1} - \y_{k}$, which  has been used for the first-order methods with momentum for nonconvex problems \citep{li2017,yao2017}.  Second, the momentum parameter $\beta_{k+1}$ in \cref{li_2} is set to be adaptive  (in fact proportional) to the norm of the progress made in the cubic regularization step and the norm of gradient, i.e., $\norml{ {\y}_{k+1} - \x_k} $ and $\norml{\nabla f( {\y}_{k+1})}$. In this way, if the iterate is far away from a second-order stationary point, $\norml{ {\y}_{k+1} - \x_k}$ and $\norml{\nabla f( {\y}_{k+1})}$ are large so that the momentum takes a large stepsize to make good progress. On the other hand, as the iterate is close to the stationary point, $\norml{ {\y}_{k+1} - \x_k}$ and $\norml{\nabla f( {\y}_{k+1})}$ are small so that the momentum takes a small momentum stepsize in order not to miss the stationary point. It turns out that such a choice of the momentum parameter is critical to guarantee the convergence of CRm (as can be seen in the proof) as well as achieving acceleration. Our experiments (see \Cref{sec: exp}) show that such a momentum scheme can substantially accelerate the convergence of CR in various nonconvex optimization problems. Therefore, the requirement of the adaptive step size $\beta_k$ in \cref{li_2} is not only intuitively reasonable but also theoretically sound. 

In the monotone step, the algorithm  compares the function values between the cubic regularization step and the momentum step, and choose the better one to perform the next step. In this way, the proposed accelerated CR algorithm is guaranteed to be monotone, i.e., the generated function value sequences are monotonically decreasing. This monotone step is not required in convex optimization, but it seems crucial in nonconvex optimization due to the landscape of nonconvex function does not have strong structure as convex function.  We further note that although the momentum step may not play a role in every iteration due to the monotone step, our experiments show that the momentum step does participate for most iterations during the course of convergence, validating its importance to accelerate the algorithm.


\section{CONVERGENCE ANALYSIS OF CRm}

In this section, we establish both the global and the local convergence rates of CRm to a second-order stationary point. 

\subsection{GLOBAL CONVERGENCE OF CRm}	\label{I_s_2}

First recall that our goal is to minimize a twice-differentiable nonconvex function $f(\x)$ (c.f. \cref{Object_function}). We adopt the following standard assumptions on the objective function.
\begin{Assumption} \label{Assumption_1} 
	The objective function in \cref{Object_function} satisfies: 
	\begin{enumerate}[leftmargin=*,topsep=0pt,noitemsep]
		\item   $f$ is twice-continuously differentiable and bounded below, i.e., $f^{\star} \triangleq \inf\limits_{\x \in \mathds{R}^d} f(\x) > -\infty$;
		\item  For all $\alpha\in \RR$, the sublevel set $\{\x: f(\x) \leqslant \alpha \}$ of   $f$ is bounded; 
		\item  The gradient $\nabla f(\cdot)$ and  Hessian $\nabla^2 f(\cdot)$ are $L_1$ and $L_2$-Lipschitz continuous, respectively.
	\end{enumerate} 
\end{Assumption}
\Cref{Assumption_1} imposes standard conditions on the nonconvex objective function $f$. In particular, the bounded sublevel set condition in item 2 is satisfied whenever $f$ is coercive, i.e., $f(\x) \to +\infty$ as $\|\x\| \to +\infty$. This is true for many non-negative loss functions under mild conditions.

Based on \Cref{Assumption_1}, we characterize the global convergence rate of CRm to a second-order stationary point in the following result.
We refer the readers to the supplementary materials for the proof.
\begin{Theorem}[Global convergence rate] \label{Adaptive_algorithm_convergence}
	Let \Cref{Assumption_1} hold and fix any $\epsilon \leqslant 1$. Then, the sequence $\{\x_{k}\}_{k\geqslant 0}$ generated by CRm   contains an $\epsilon$-second-order stationary point provided that the total number of iterations $k$ satisfies that
	\vspace{-0.3cm}
	\begin{align}
	k \geqslant \frac{C}{\epsilon^{3/2}},
	\end{align}
	where $C$ is a universal positive constant and is specified in the proof.
\end{Theorem}

 
\Cref{Adaptive_algorithm_convergence} establishes the global convergence rate to an $\epsilon$-second-order stationary point for CRm. Although the obtained convergence rate of CRm  achieves the same order as that of the original CR algorithm in \cite{Nesterov2006}, which is in fact the best that one can expect for general nonconvex optimization, the technical proof critically exploits the design of the momentum scheme, and requires substantial machinery to handle the momentum step. Further in \Cref{sec: exp}, we demonstrate via various experiments that CRm  do enjoy the momentum acceleration and converge much faster than the original CR algorithm.


\subsection{LOCAL CONVERGENCE OF CRm} \label{I_s_3}
It is well known that Newton-type second-order algorithms enjoy a local quadratic convergence rate for minimizing strongly convex functions. While strong convexity is a restrictive condition in nonconvex optimization, many nonconvex  problems such as phase retrieval \cite{Zhang2017} and low-rank matrix recovery \cite{Tu_2016} have been shown to satisfy the following more relaxed local error bound condition \citep{Yue2018}.

\begin{Assumption}[Local error bound]\label{ass_eb}
	Denote $\mathcal{X}$ as the set of second-order stationary points of $f$. There exists $\kappa,r > 0$ such that for all $\x\in \{\x: \mathrm{dist}(\x, \mathcal{X}) \leqslant r\}$, it holds that
	\begin{equation} 
	\textrm{dist}(\x, \mathcal{X}) \leqslant \kappa\|\nabla f(\x)\|, 
	\end{equation} 
	where dist$(\x, \mathcal{X})$ denotes the point-to-set distance between $\x$ and $\mathcal{X}$.
\end{Assumption}
One can easily check that all strongly convex functions satisfy the above local error bound condition. Therefore, the local error bound condition is a more general geometry than strong convexity.  

Next, we explore the local convergence property for CRm  under the local error bound condition. Typically, such a property is due to the usage of the Hessian information in the algorithm. In CRm, the momentum step does not directly exploit the Hessian information. Hence, it is not clear {\em a priori} by including the momentum step whether CRm still enjoys the local quadratic convergence property. The following theorem provides an affirmative answer.
\begin{Theorem} \label{Local_adaptive_momentum_Thm}
	Let Assumptions \ref{Assumption_1} and \ref{ass_eb} hold. Then, the sequence $\{\x_k\}_{k\geqslant0}$ generated by CRm  with $M > L_2$ converges quadratically to a point $\x^\star \in\mathcal{X}$, where $\mathcal{X}$ is the set of second-order stationary points of $f$. That is, there exists an integer $k_1$ such that for all $k \geqslant k_1$,
	\begin{align}
	\norml{\x_{k+1} - \x^\star } \leqslant C  \norml{\x_{k} - \x^\star }^2, \label{eq: q}
	\end{align}
		where $C$ is a universal positive constant  and is specified in the proof.
\end{Theorem} 
Under the local error bound condition,  \Cref{Local_adaptive_momentum_Thm} shows that CRm enjoys a quadratic convergence rate as shown in \cref{eq: q}. To elaborate, note that \Cref{Adaptive_algorithm_convergence} guarantees the convergence of CRm  to a second-order stationary point, i.e., $\norml{\x_{k} - \x^\star } \to 0$ as $k\to \infty$. Thus, the recursion in \cref{eq: q} implies that $C \norml{\x_{k} - \x^\star } \leqslant (C \norml{\x_{k_1} - \x^\star })^{2^{k-k_1}}$, which is at a quadratic converge rate. In particular, the region of quadratic convergence is defined by $\norml{\x_{k} - \x^\star } \leqslant  1/C$. Such quadratic convergence achieves an $\epsilon$-accuracy second-order stationary point within $k = O(\log\log(1/\epsilon))$ number of iterations, which is much faster than the linear converge rate of fisrt-order methods in local region. 

 Local quadratic convergence has also been established for the original CR algorithm under the local error bound condition \cite{Yue2018}. As a comparison, our proof of \Cref{Local_adaptive_momentum_Thm} for CRm exploits the proposed momentum scheme, which results in additional terms that requires extra effort to handle. 



\section{INEXACT VARIANTS OF CRm} \label{I_s_4}

The major computational load of CRm lies in the cubic step, which requires to solve a computationally costly optimization problem. In this section, we explore three implementation schemes that can efficiently perform the cubic step without sacrificing the acceleration performance.



\subsection{CUBIC STEP WITH INEXACT HESSIAN} 
The cubic step requires the full Hessian information, which can be too costly in practice. Instead,
we consider performing the following cubic step with an inexact approximation of the Hessian.
\begin{align} 
\hat{ \x}_{k+1} = \argmin_{\s \triangleq \x - \x_k} \nabla f(\x_k)^\top\s + \frac{1}{2} \s^\top \mathbf{H}_k \s + \frac{M}{6} \norml{\s}^3,
\label{eq: H}
\end{align}  
where $\mathbf{H}_k$ denotes the inexact estimation of the full Hessian $\nabla^2 f(\x_k)$, and their difference is assumed to satisfy the following criterion. \Cref{sec:subsampling} proposes a subsampling scheme to achieve \Cref{assump_2} for the finite-sum problem.
\begin{Assumption}\label{assump_2}
	The inexact Hessian $\mathbf{H}_k$ in \cref{eq: H} satisfies, for all $k \geqslant 0$,
	\begin{align*}
	\norml{\mathbf{H}_k - \nabla^2 f(\x_{k})} \leqslant \epsilon_1.
	\end{align*}
\end{Assumption} 
\Cref{assump_2} assumes that the inexact Hessian is close to the exact one in terms of a small operator norm gap. Such inexact criterion has been considered in \cite{JinChi2017cubic,Xu2017} to study the convergence property of the inexact CR algorithm. 

Next, we study the convergence of the inexact variant  of CRm  by replacing the cubic step in \cref{I_1} with the inexact cubic step in \cref{eq: H}.  Our main result is summarized as follows, and the proof is provided in the supplemental materials. 

\begin{Theorem} \label{tyep1_im_thm}
	Let Assumptions \ref{Assumption_1} and \ref{assump_2} hold and fix any $\epsilon \leqslant 1$. Then, the sequence $\{\x_{k} \}_{k\geqslant 0}$ generated by the inexact CRm with $M >  {2 L_2}/{3}  + 2$ and $ \epsilon_1 = \theta  \sqrt{\epsilon}$   contains an $\epsilon$-second-order stationary point provided that the total number of iterations $k$ satisfies that
	\begin{align}
	k \geqslant \frac{C}{\epsilon^{3/2}},
	\end{align}
	where  $C, \theta$ are universal constants, and  are specified in the proof.
\end{Theorem}
\Cref{tyep1_im_thm} shows that, under a proper inexact criterion, the iteration complexity of  inexact CRm  is on the same order as that of exact CRm for achieving an $\epsilon$-second-order stationary point. 
Since the inexact Hessian saves the computation in each iteration comparing to the full Hessian, it is clear that the overall computation complexity of the inexact CRm  is less than that of the exact cases. In \Cref{ex_inexact_exact}, we verify through experiments that the inexact algorithm do perform much better than the corresponding exact version. 

We note that the proof of   \Cref{tyep1_im_thm} suggests that the condition that $\norml{\y_{k+1} - \x_k} \leqslant \epsilon_1$ implies the point $\x_{k+1}$ is an $\epsilon$-second-order stationary point, where $ \epsilon_1 = \theta  \sqrt{\epsilon}$. Thus, the implementation of the inexact CRm can terminate by checking the satisfaction of the condition $\norml{\y_{k+1}- \x_k} \leqslant \epsilon_1$.

\subsection{INEXACT CRm VIA SUBSAMPLING}\label{sec:subsampling}

In this subsection, we consider a general finite-sum optimization problem, where inexact CRm can be implemented via subsampling.  More specifically, consider to solve the following optimization problem:
\begin{align}
	 f(x) \triangleq \sum_{i=1}^{n} f_i(x),
\end{align}
where $f_i(\cdot)$ is possibly nonconvex. Furthermore, we assume that \Cref{Assumption_1} holds for each $f_i(\cdot)$. 
For finite-sum problems, the full Hessian can be approximated by the Hessian of a mini-batch of data samples each uniformly randomly drawn from the dataset,  i.e.,
\begin{align}
	 \mathbf{H}_k = \frac{1}{|S_1|} \sum_{i \in S_1} \nabla^2 f_i(\x_k). \label{batch}
\end{align}


We use the subsampling technique introduced in \cite{kohler2017} to satisfy the inexact condition in \Cref{assump_2}. The following theorem provides our characterization of the overall Hessian sample complexity in order to guarantee the convergence of the subsampling algorithm with high probability over the entire iteration process.
\begin{Theorem}[Total Hessian sample complexity] \label{total_compelxity}  Assuming that \Cref{Assumption_1} holds for each $f_i(\cdot)$,  and let the sub-sampled mini-batch of Hessians $\mathbf{H}_k, k = 0, 1, \ldots$   satisfies  
	\begin{align*}
		|S_1| = \left( \frac{8L_1^2}{\theta^2 \epsilon} + \frac{4L_1 }{3 \theta \sqrt{\epsilon} }\right) \log \left(\frac{4d}{ \epsilon^{3/2}\delta}\right),
	\end{align*} 
	 then the sequence $\{\x_{k} \}_{k\geqslant 0}$ generated by the inexact CRm  with $M >  L_2  + 2$  outputs an $\epsilon$-second-order stationary point with probability at least $1 - \delta$ by taking at most the following number of Hessian samples in total:
\begin{align*}
 S \leqslant C\left( \frac{8L_1^2}{\theta^2 \epsilon^{5/2}} + \frac{4L_1 }{3 \theta {\epsilon}^2 }\right) \log \left(\frac{4d}{\epsilon\delta}\right) . 
\end{align*}
\end{Theorem}  

 \Cref{total_compelxity} characterizes the total Hessian sample complexity to guarantee the convergence of CRm with high probability. This is the first such a type result for stochastic CR algorithms. Note that previous studies \cite{kohler2017,Xu2017} on subsampling CR provide only the Hessian sample complexity per iteration to guarantee inexactness condition with high probability. Our result indicates that even over the entire iteration process, the convergence is still guaranteed with high probability. In fact, if we let $N$ denote the total sample complexity, \Cref{total_compelxity} implies that the failure probability $\delta$ decays exponentially fast as the total sample complexity $N$ becomes asymptotically large. Such a result by nature is   stronger than those that characterize  the convergence only in expectation, not in (high) probability, in existing literature. 
\begin{figure*}[h]  
	\vspace{-0.3cm}
	\centering 
	\subfigure{\includegraphics[width=0.32\linewidth]{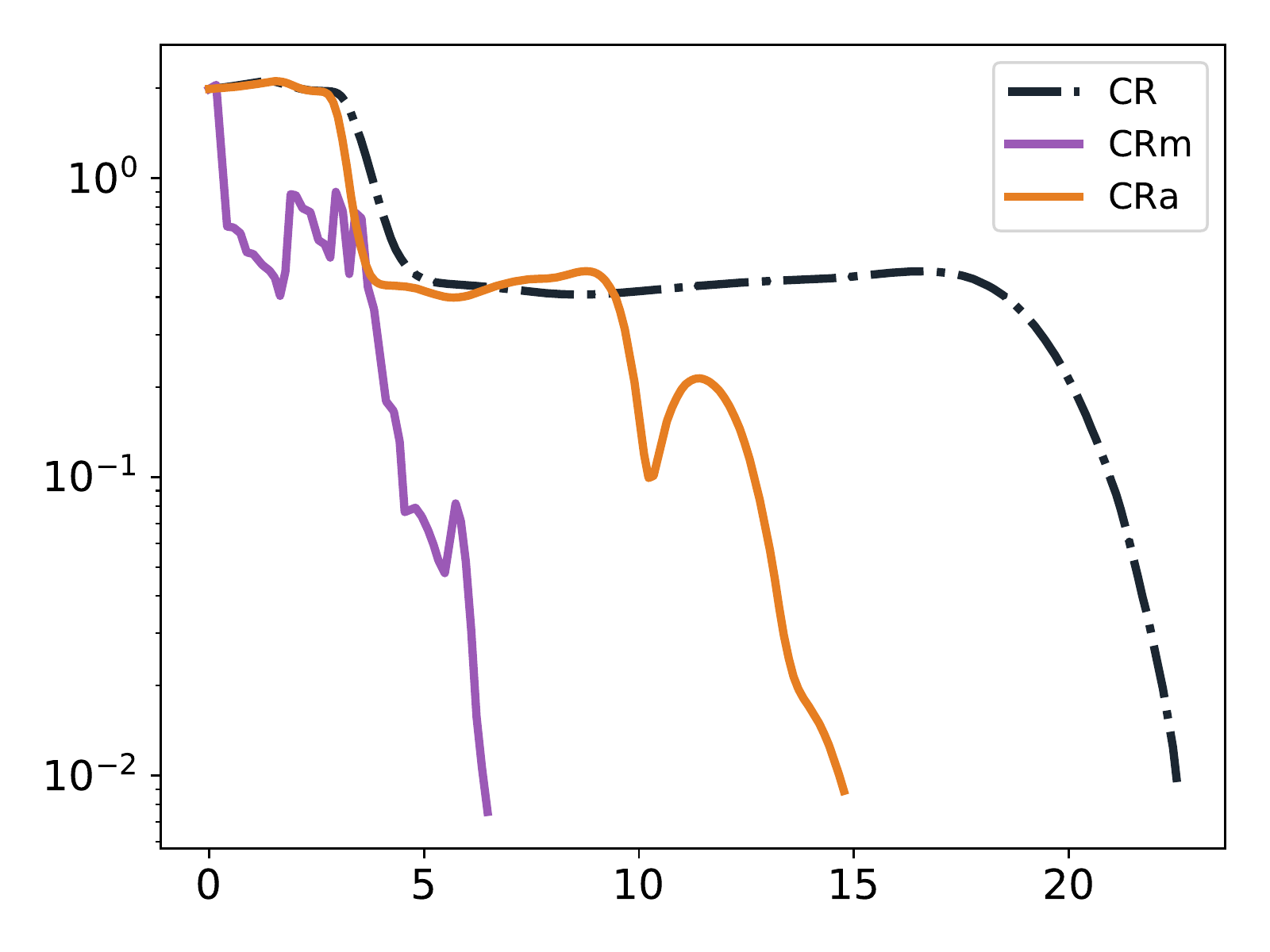}} 
	\subfigure{\includegraphics[width=0.32\linewidth]{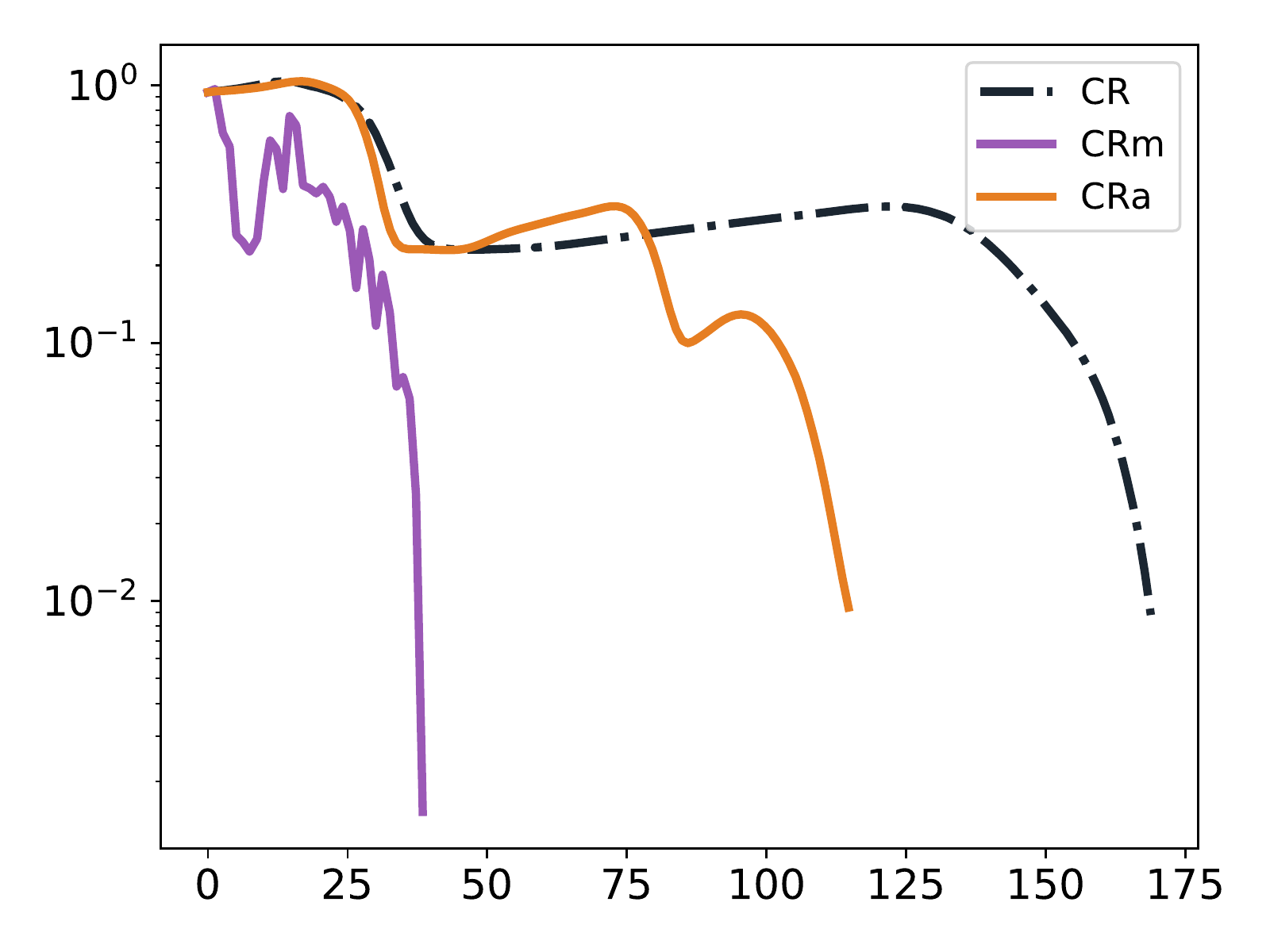}}
	\subfigure{\includegraphics[width=0.32\linewidth]{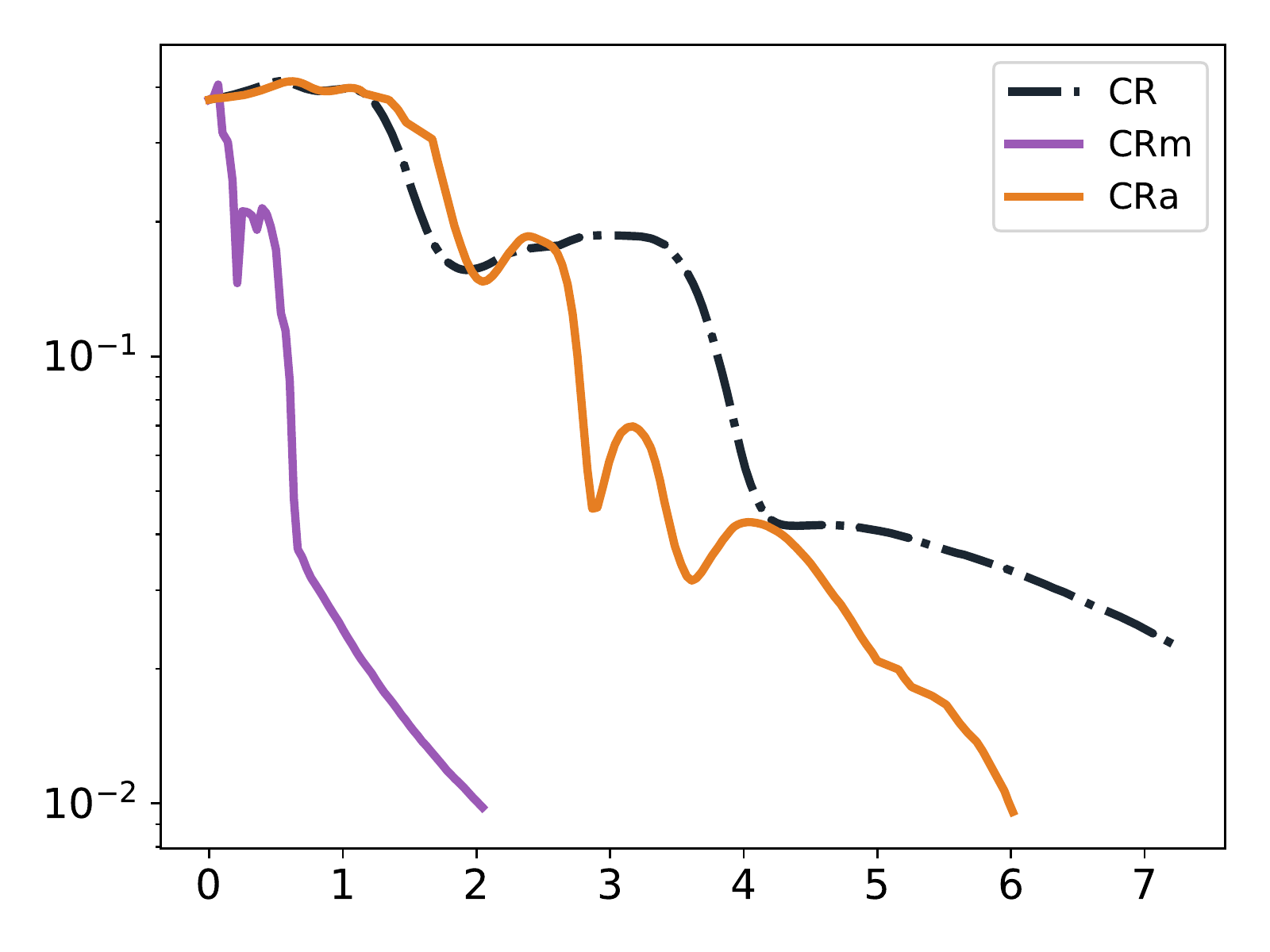}}\\
	\addtocounter{subfigure}{-3}
	\subfigure[a9a ]{\includegraphics[width=0.32\linewidth]{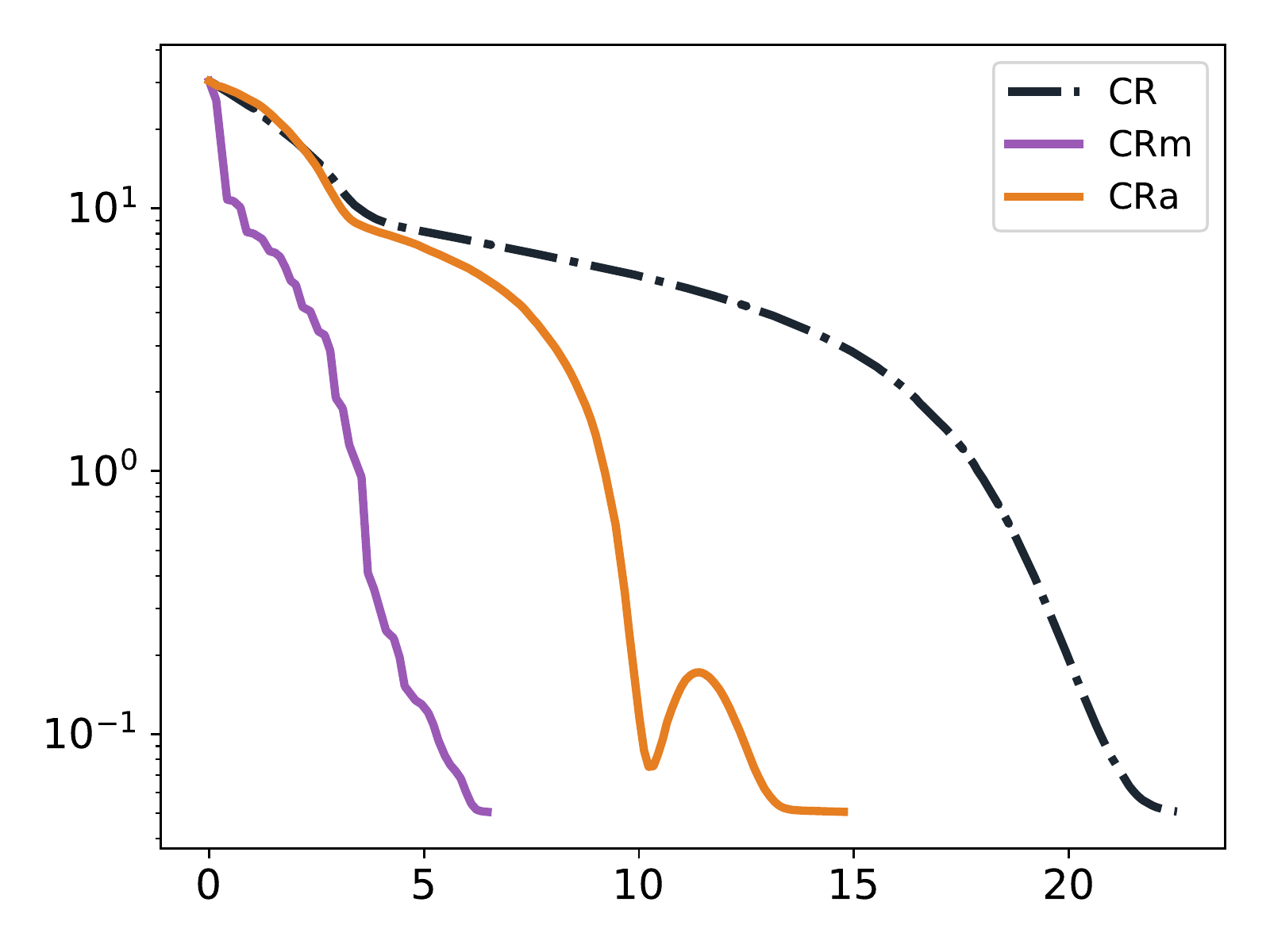}} 
	\subfigure[covtype ]{\includegraphics[width=0.32\linewidth]{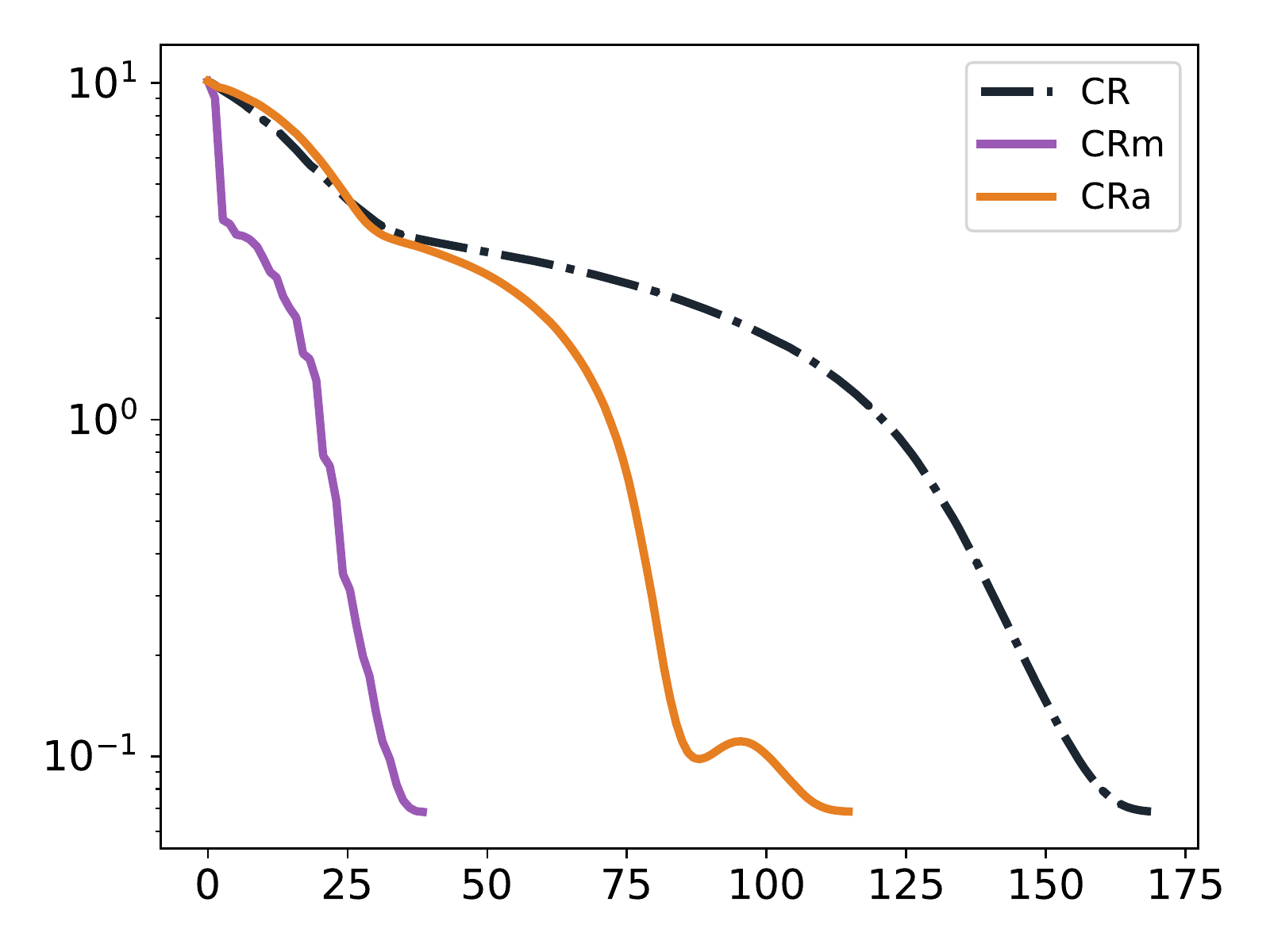}}  
	\subfigure[ijcnn1 ]{\includegraphics[width=0.32\linewidth]{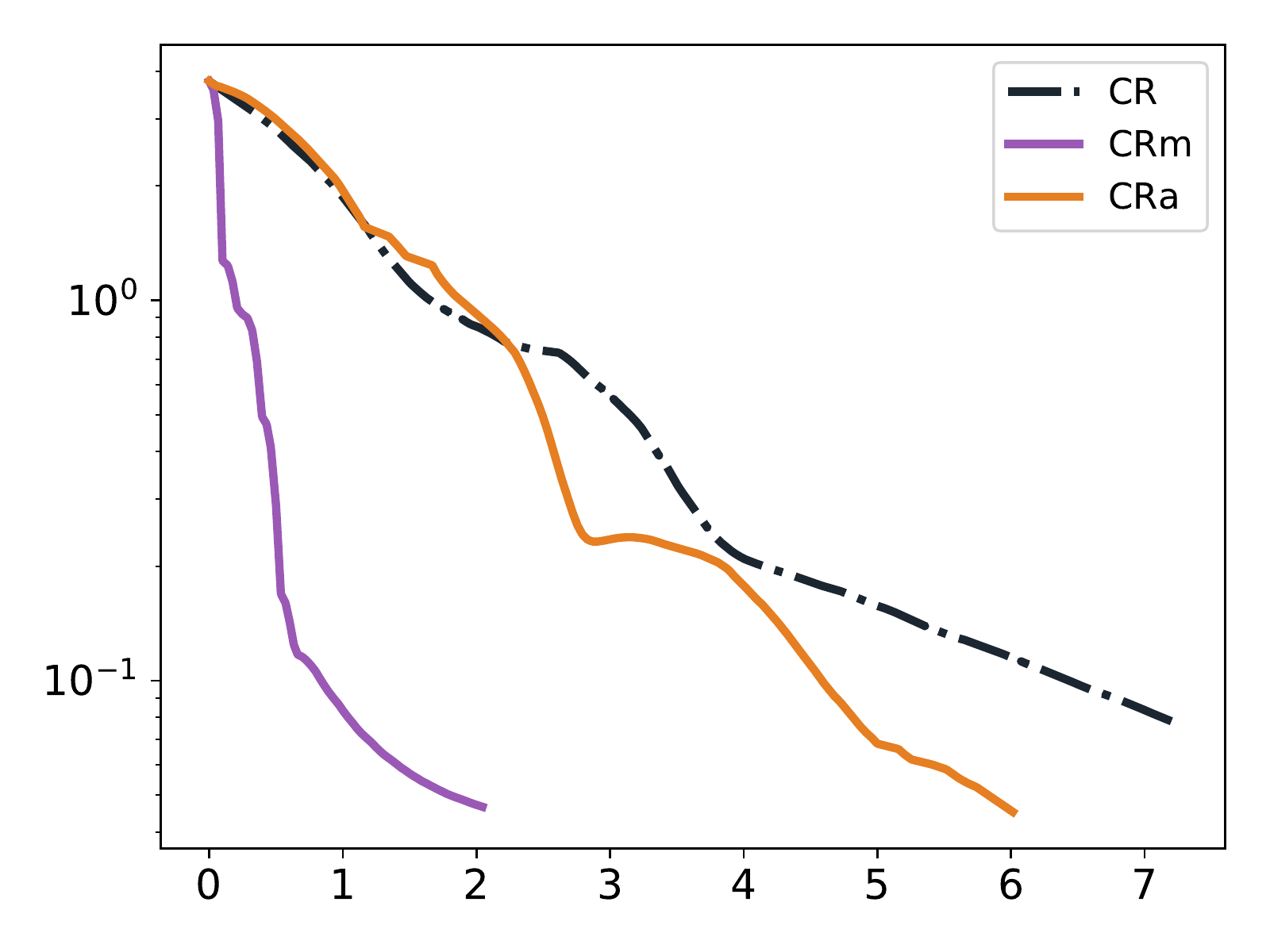}}  
	\caption{Nonconvex logistic regression. Top: gradient norm v.s. time. Bottom: function value gap v.s. time. }   \label{Experment_1}
\end{figure*}
\begin{figure*}[h]   
	\vspace{-0.3cm}
	\centering 
	\subfigure{\includegraphics[width=0.32\linewidth]{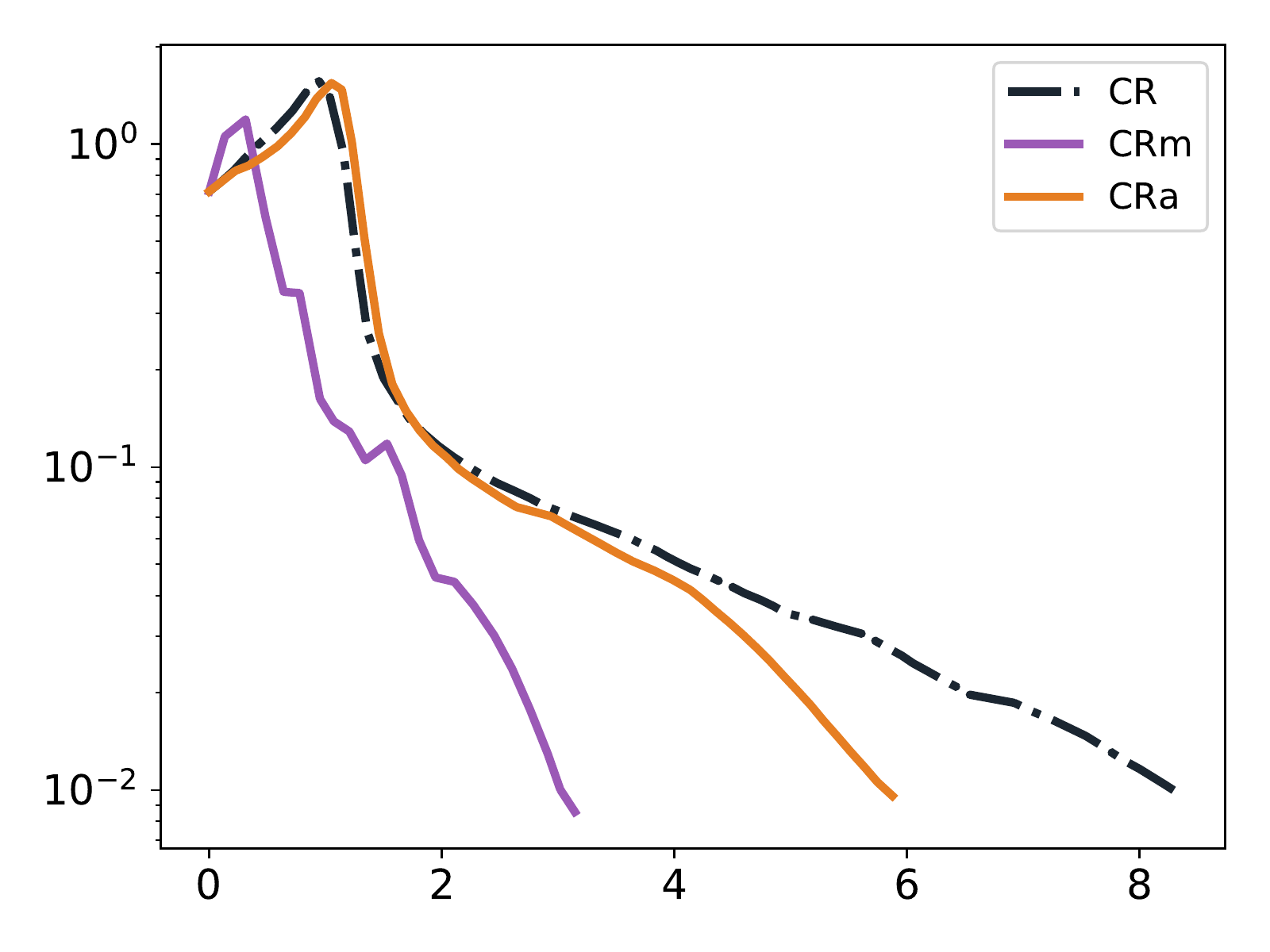}} 
	\subfigure{\includegraphics[width=0.32\linewidth]{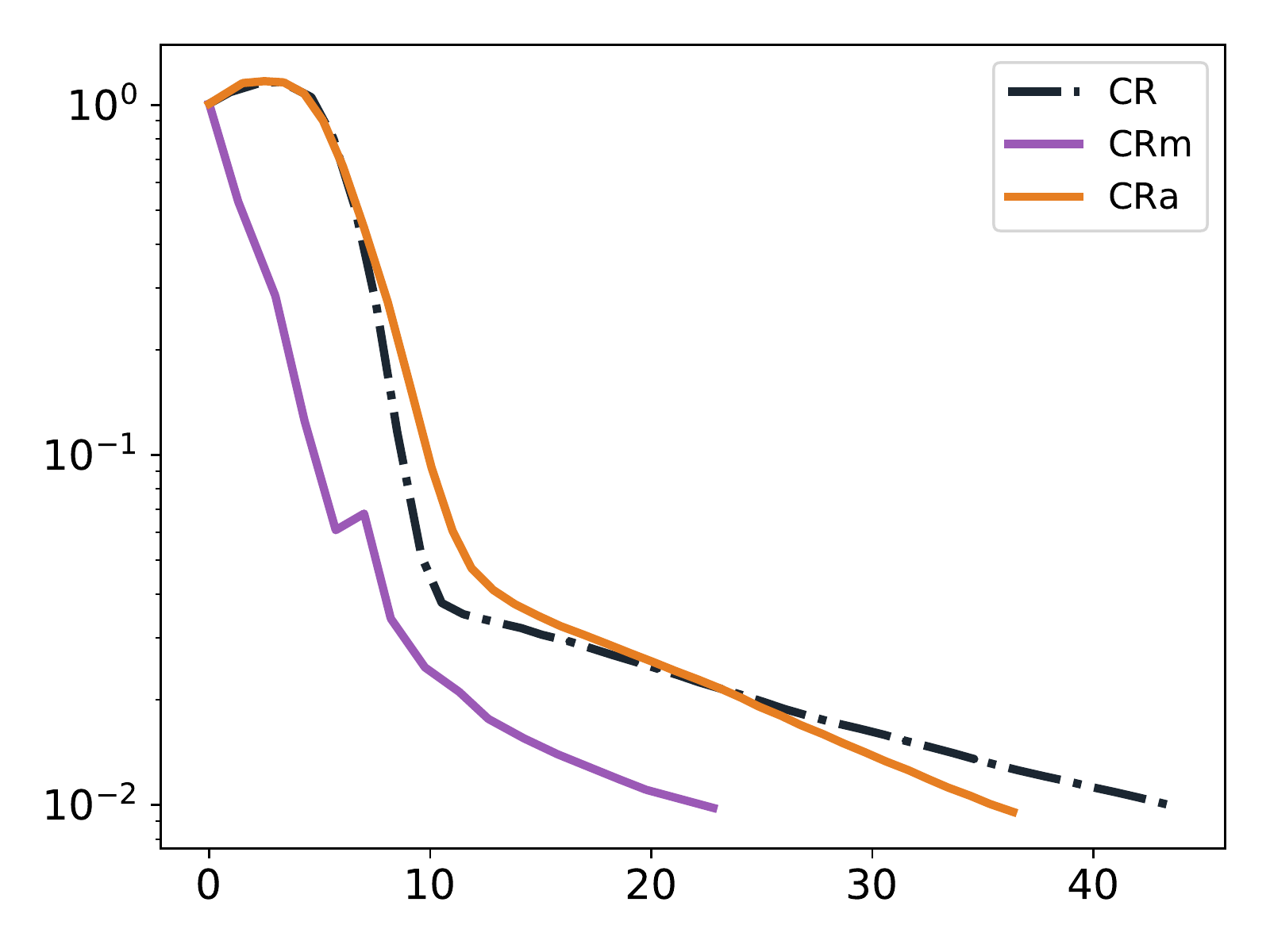}}
	\subfigure{\includegraphics[width=0.32\linewidth]{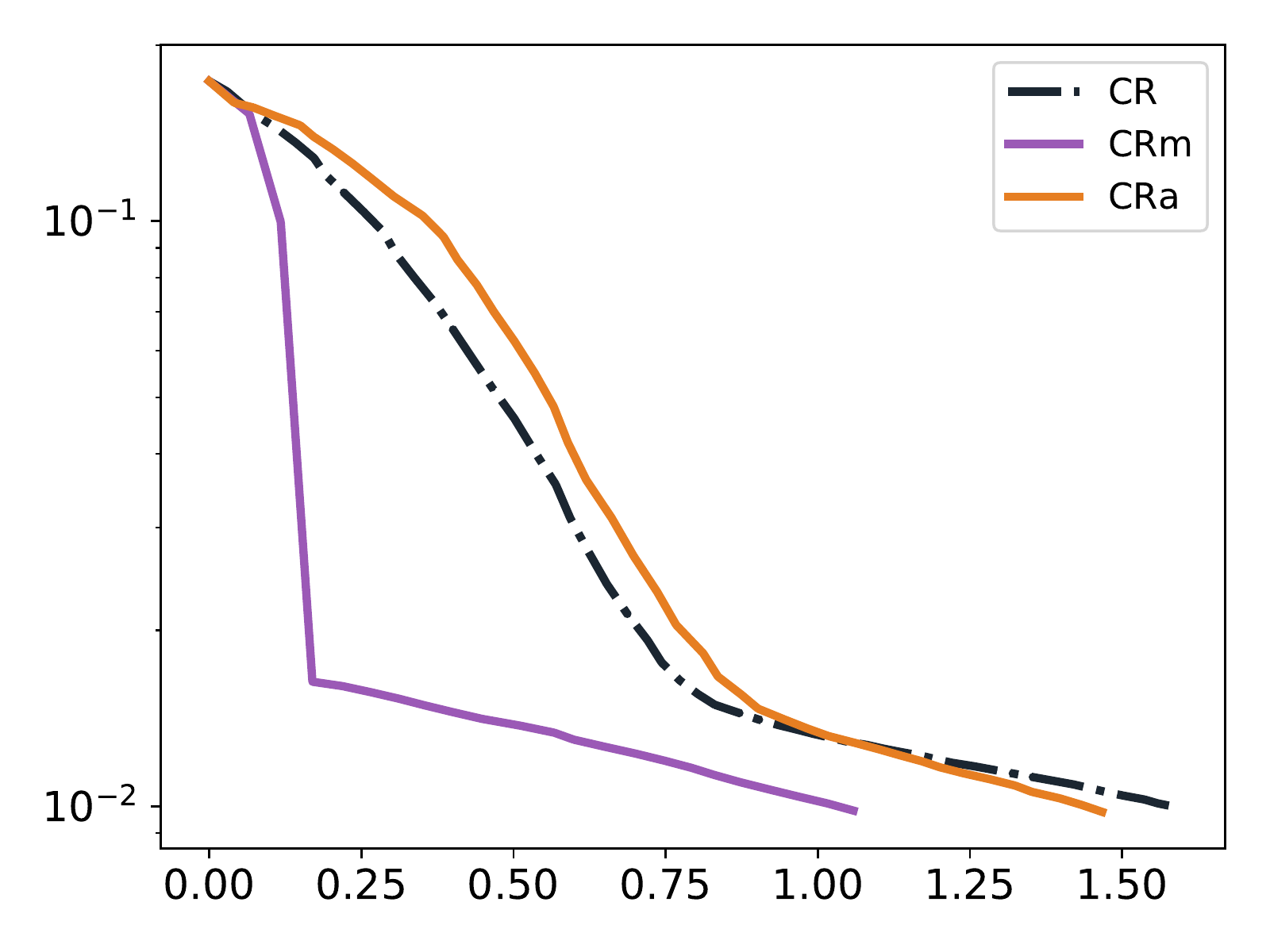}} \\
	\addtocounter{subfigure}{-3}
	\subfigure[a9a ]{\includegraphics[width=0.32\linewidth]{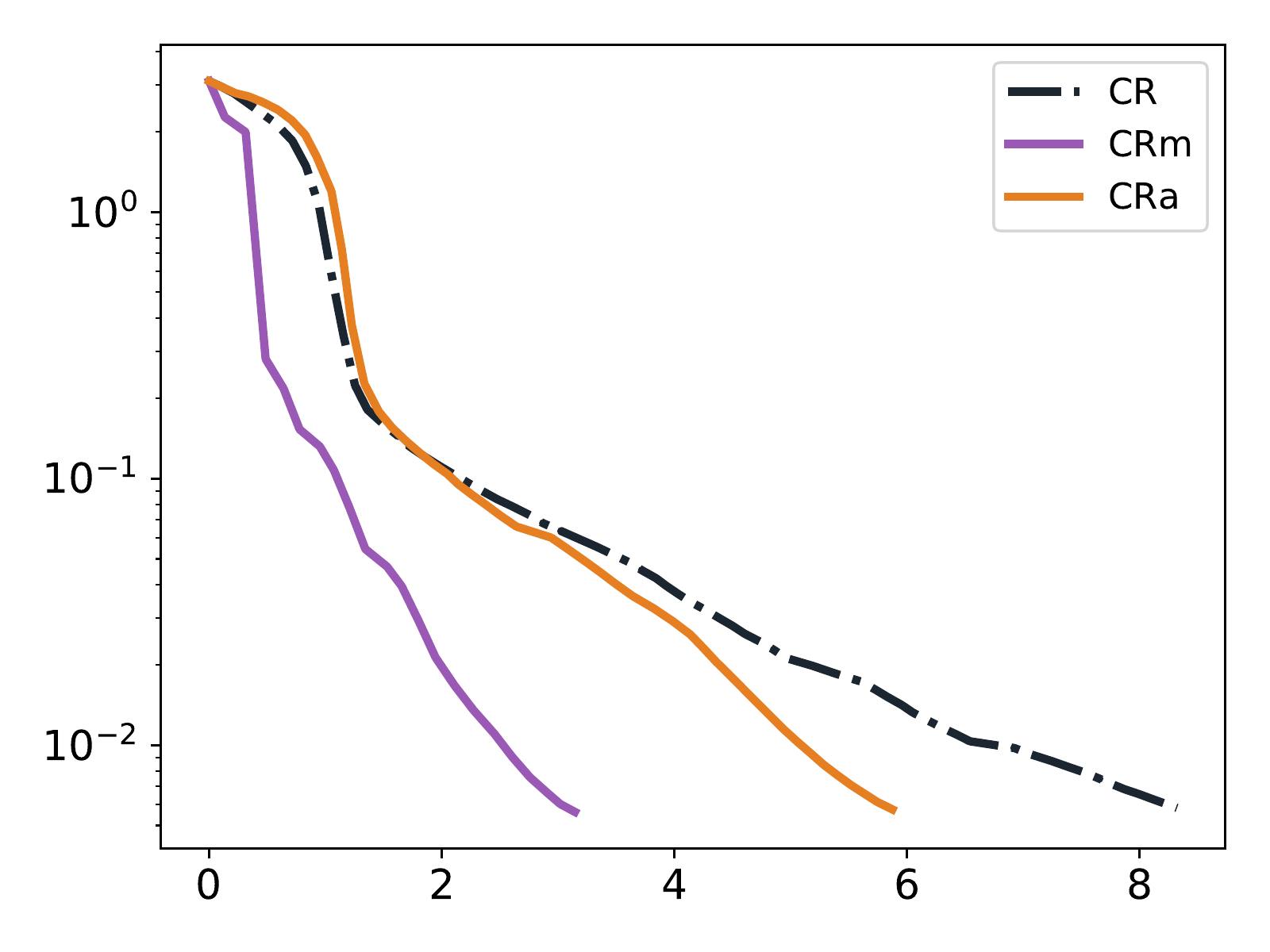}} 
	\subfigure[covtype ]{\includegraphics[width=0.32\linewidth]{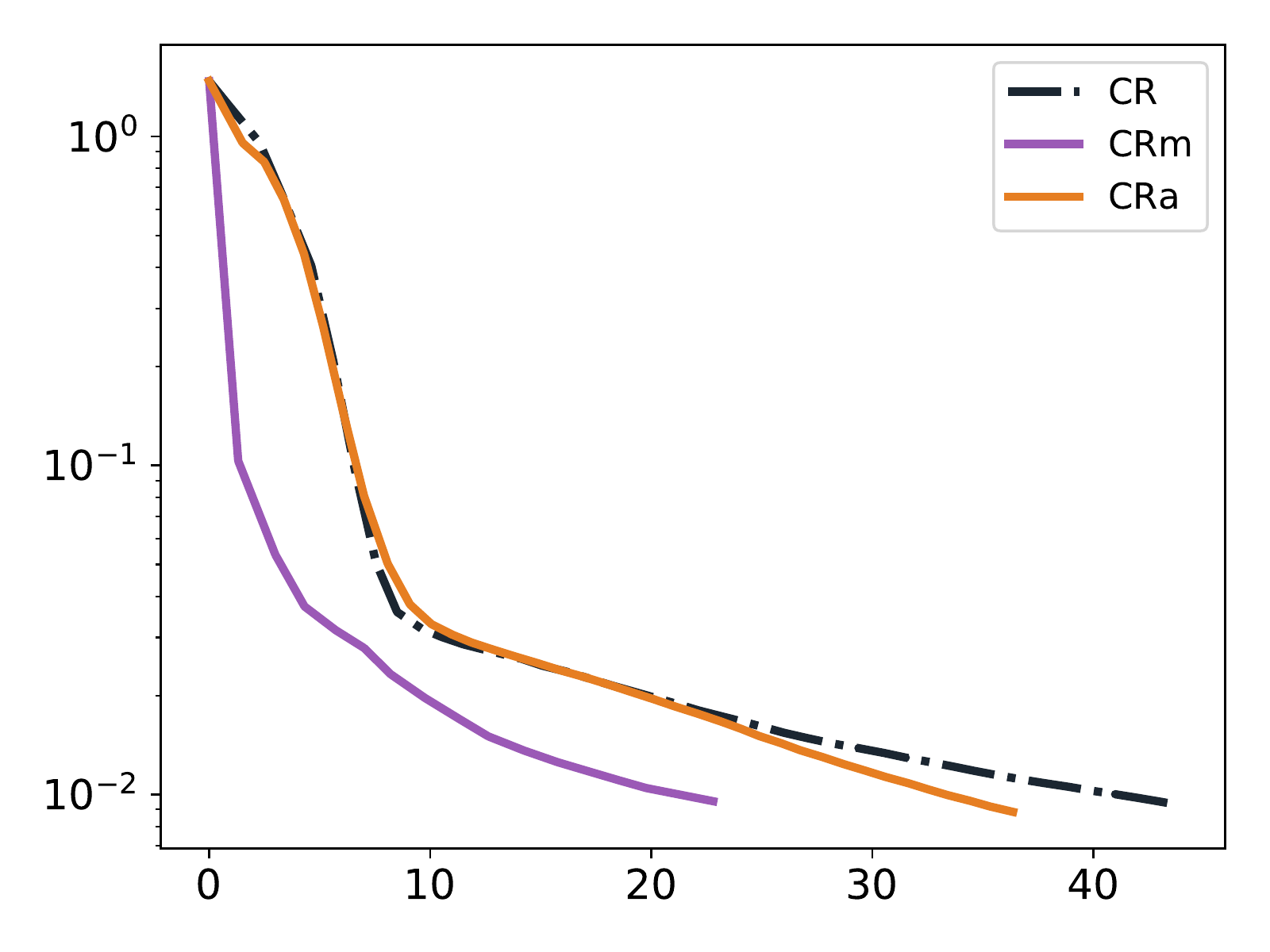}}  
	\subfigure[ijcnn1 ]{\includegraphics[width=0.32\linewidth]{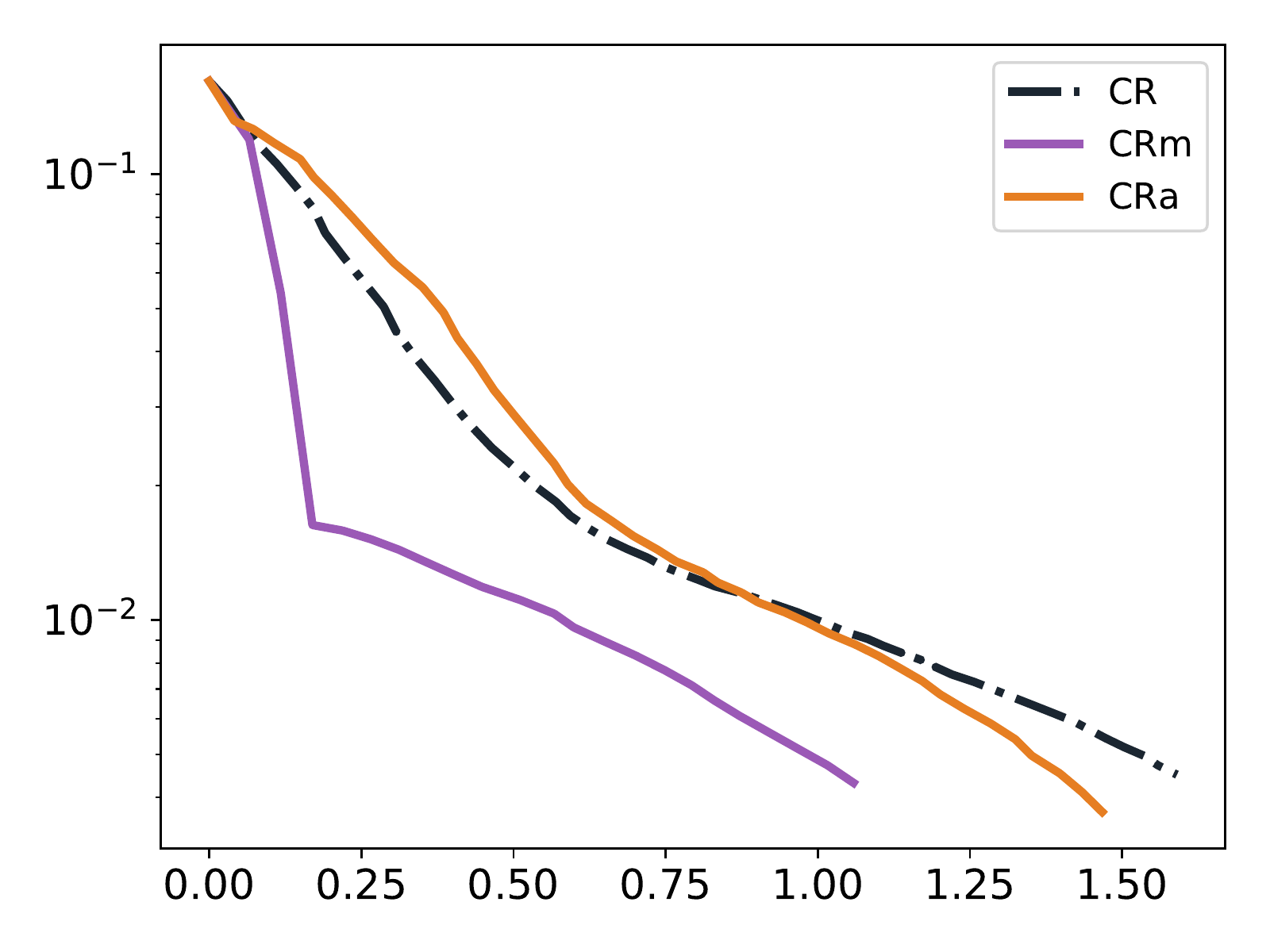}}  
	\caption{Robust linear regression. Top: gradient norm v.s. time. Bottom: function value gap v.s. time.}   \label{Experment_2}	
	\vspace{-0.3cm}
\end{figure*}  
\subsection{EFFICIENT SOLVERS FOR CUBIC STEP} \label{sub_problem_solver} 
Since the cubic step does not have a closed form solution, an inexact solver is typically used for solving the cubic step. Various solvers have been proposed to approximately solve such a subproblem. The first type of solver is based on the Lanczos method \citep{Cartis2011a,Cartis2011b}, which solves the cubic subproblem in a Krylov subspace $\mathcal{K} = \text{span} \{\nabla f(\x_k),\nabla^2 f(\x_k)\nabla f(\x_k), \cdots \}$ instead of in the entire space. Each step of the solver can be implemented efficiently with a computation cost of $O(d)$ \citep{kohler2017}. Moreover, building the subspace requires a Hessian-vector product, which introduces a cost of $O(nd)$ per additional subspace dimension for finite-sum problem with $n$ data samples. The second type of solver is proposed by \cite{Agarwal2017}, which is based on the techniques of Hessian-vector product and binary search. The proposed solver 
can find an approximate solution of the cubic subproblem with a total cost of $O(nd/\epsilon^{1/4})$ for finite-sum problems, where $\epsilon$ is the desired accuracy. \cite{Carmon2016b} proposed another solver   based on the gradient descent method. The solver finds an approximate solution of the cubic subproblem within $O(\epsilon^{-1} \log (1/\epsilon))$ iterations for large $\epsilon$ and $O(\log(1/\epsilon))$ iterations for small  $\epsilon$. 

All of these solvers can be applied to solve the cubic step in CRm. Note that the momentum and monotone steps in CRm introduce order-level less computation complexity compared to these solvers for solving the cubic step. Thus, CRm  have the same per-iteration computational complexity as CR when implementing the same solver, and have at least the same overall computational complexity as CR (in fact, much less overall computational complexity in practice as demonstrated by our experiments). As the solvers solve the cubic subproblem up to certain accuracy in practice, the total computation complexity of CRm to achieve a second-order stationary point can still be established.

\section{EXPERIMENTS}\label{sec: exp}   

\subsection{SETUP}

We compare the performance among the following six algorithms: cubic regularization algorithm (\textbf{CR}) in \cite{Nesterov2006}, accelerated cubic regularization algorithm (\textbf{CRa}) in \cite{Nesterov2008} (whose convergence guarantee has not been established), cubic regularization algorithm (\textbf{CRm}) with  momentum, cubic regularization algorithm with inexact Hessian (\textbf{CR\_I}), accelerated cubic regularization with inexact Hessian (\textbf{CRa\_I}), CRm with  inexact Hessian (\textbf{CRm\_I}).  In this section, we present the comparison among the  three exact algorithms. The comparisons  among the  inexact variants are presented in \Cref{Add_Experiment} due to space limitation. The details of the experiment settings can also be found  in \Cref{Add_Experiment}.

We conduct two experiments. The first experiment solves the following logistic regression problem with a nonconvex regularizer
\begin{align*}
\min_{\w \in \mathbb{R}^d} &-\left(  \frac{1}{n} \sum_{i=1}^{n} y_i  \log \left(\frac{1}{1 + e^{-\w^T\x}}\right) \right.\\
&\left.+( 1 - y_i) \log \left(\frac{e^{\w^T\x}}{1 + e^{-\w^T\x}} \right) \right) + \alpha \sum_{i=1}^{d} \frac{ w_i^2}{1  + w_i^2 },
\end{align*}
where we set $\alpha = 0.1$ in our experiment. The second experiment solves the following nonconvex robust linear regression problem
\begin{align}
\min_{\w \in \mathbb{R}^d} \frac{1}{n} \sum_{i=1}^{n}  \eta(y_i - \w^T\x_i), 
\end{align}
where $\eta (x) =\log (\frac{x^2}{2} + 1)$. Each experiment is performed over three datasets, i.e., a$9$a, covtype, and ijcnn \citep{Chang_2011}.

\subsection{RESULTS} 
\Cref{Experment_1,Experment_2} show the results of the two experiments for comparing the three exact algorithms, respectively. From both figures, it can be seen that CRm outperforms the vanilla CR, which demonstrates that the momentum step in CRm significantly accelerates the CR algorithm for nonconvex problems.  Also, CRm  outperforms CRa in the experiments with datasets a9a and covtype, while its performance is comparable to CRa in the experiments with dataset ijcnn1. Thus, our proposed momentum step achieves a faster convergence than the Nesterov's acceleration scheme for CR.

We note that similar comparisons are observed in the comparison of the corresponding three inexact variants of the algorithms (see \Cref{inexact_1,inexact_2} in \Cref{Add_Experiment}), i.e., our momentum scheme  with inexact Hessian outperforms other inexact CR algorithms.  We also note that all inexact variants of the algorithms outperform their exact counterparts (see \Cref{ie_1,ie_2} in \Cref{Add_Experiment}). Hence, the inexact implementation plays an important role in reducing the computation complexity of these CR type of algorithms in practice.


\section{CONCLUSION}
In this paper, we proposed a momentum scheme  to accelerate the cubic regularization algorithm. We showed that the order of the global convergence rate of the proposed algorithm CRm is at least as fast as its vanilla version. We also established the local quadratic convergence property for the proposed algorithm, and extended our analysis for the proposed algorithm  to the inexact Hessian case and established the total Hessian sample complexity to guarantee the convergence with high probability. We further conducted various experiments to demonstrate  the advantage of applying momentum for accelerating the cubic regularized algorithm.  

\section*{Acknowledgement}

The work of Z. Wang, Y. Zhou and Y. Liang was supported in part by the U.S. National Science Foundation under the grant CCF-1761506,  and the work of G. Lan was supported in part by Army Research Office under the grant  W911NF-18-1-0223 and U.S. National Science Foundation under the grant  CMMI-1254446.

\bibliographystyle{apalike} 
\bibliography{ref}

\onecolumn 
\newpage
\appendix
\noindent {\Large \textbf{Supplementary Materials}}

We first note that in order for a clear illustration, we use a slightly difference notation such that $\vb_{k+1}$ is replaced by $\tilde{ \x}_{k+1}$ and $\y_{k+1}$ is replaced by $\hat{ \x}_{k+1}$. 
\section{Experiment Setting and Additional Result} \label{Add_Experiment}
\textbf{Parameters Setting:} The experiment specifications are as follows. For all algorithms, the subproblem in each iteration of the cubic step is solved by the Lanczos-type method as suggested by \cite{Cartis2011a}. We set the parameter $M = 10$ through all the experiments. The momentum parameter $\beta_{k+1}$ in CRm  is set to be $8\times \norml{\hat{\x}_{k+1} - \x_k} $.  Although in theory of CRm, we require $\beta_{k+1} < \min\{\rho , \norml{\nabla f(\hat{\x}_{k+1})}, \norml{\hat{\x}_{k+1} - \x_k}  \} $, in practice, a more relaxed value works well in our experiments. The initial point is set to be an all-two-vector for all datasets for the logistic regression problem and an all $0.5$ vector for all datasets for the robust linear regression problem.

\subsection{Comparison Among Inexact Algorithms}
In this subsection, we present the comparison among the three  algorithms with inexact Hessian. Namely,  vanilla cubic regularized algorithm with inexact Hessian (\textbf{CR\_I}), Nesterov accelerated cubic regularization algorithm with inexact Hessian (\textbf{CRa\_I}), and the proposed CRm with  inexact Hessian (\textbf{CRm\_I}). For the  implementation, since we solve  a finite-sum problem $f(\x) =  \sum_{i=1}^{n} f_i(\x)$, we draw a mini-batch of data samples as  the estimated Hessian, given by $\mathbf{H_k} =  \sum_{i \in S_k} \nabla^2 f_i(\x_k)/|S_k|
$. We take $n/20$ as the batch size  in the logistic regression problem and $n/5$ in the robust linear regression problem.  The results are shown in \Cref{inexact_1,inexact_2}. The performance comparison among the four algorithms with inexact Hessian  is similar to that of their exact cases. However, we should note that the   time used by inexact algorithms is much less than that of their  corresponding exact cases.
\begin{figure} [h]
	\vspace{-0.3cm}
	\centering 
	\subfigure{\includegraphics[width=0.32\linewidth]{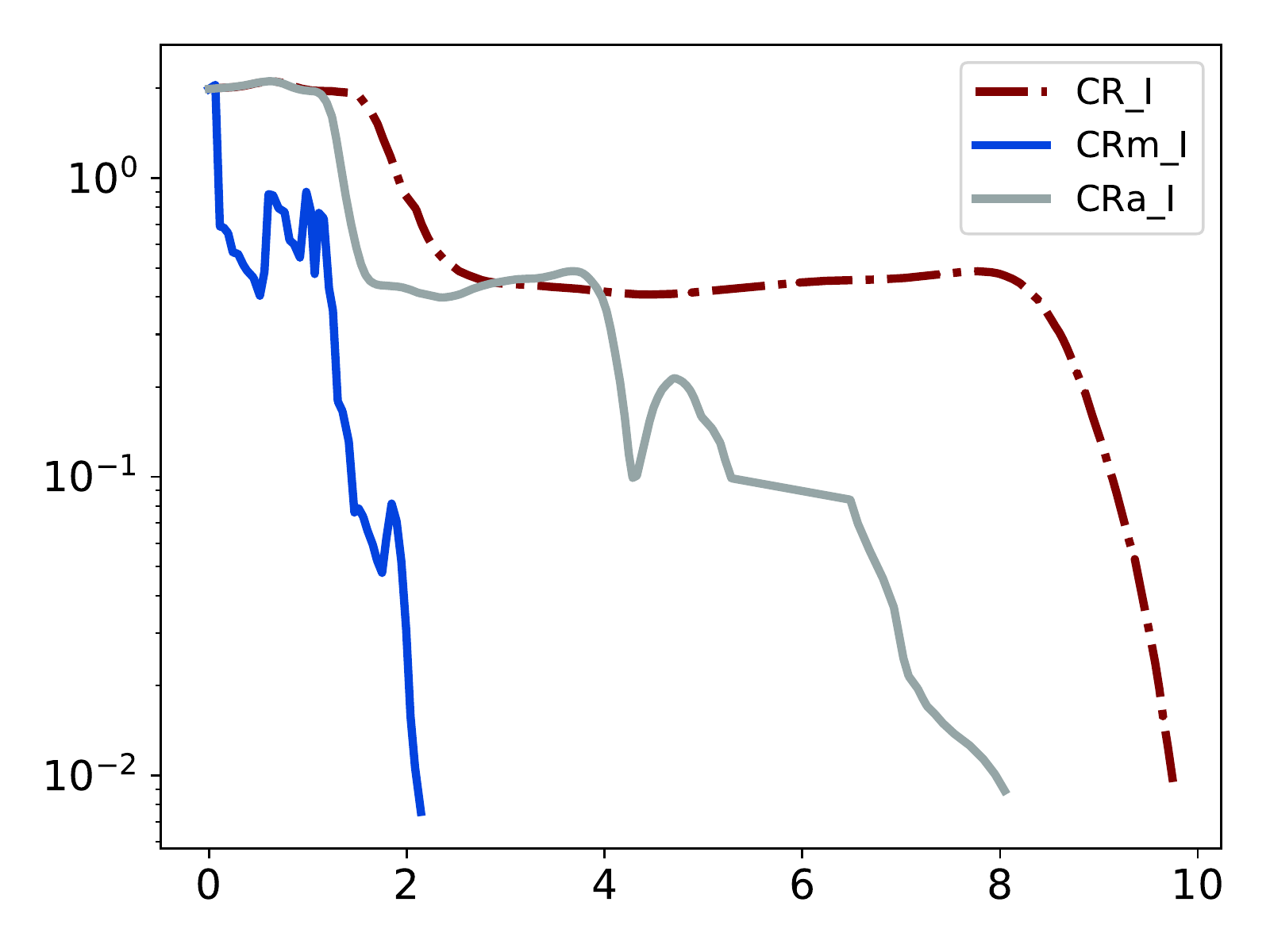}} 
	\subfigure{\includegraphics[width=0.32\linewidth]{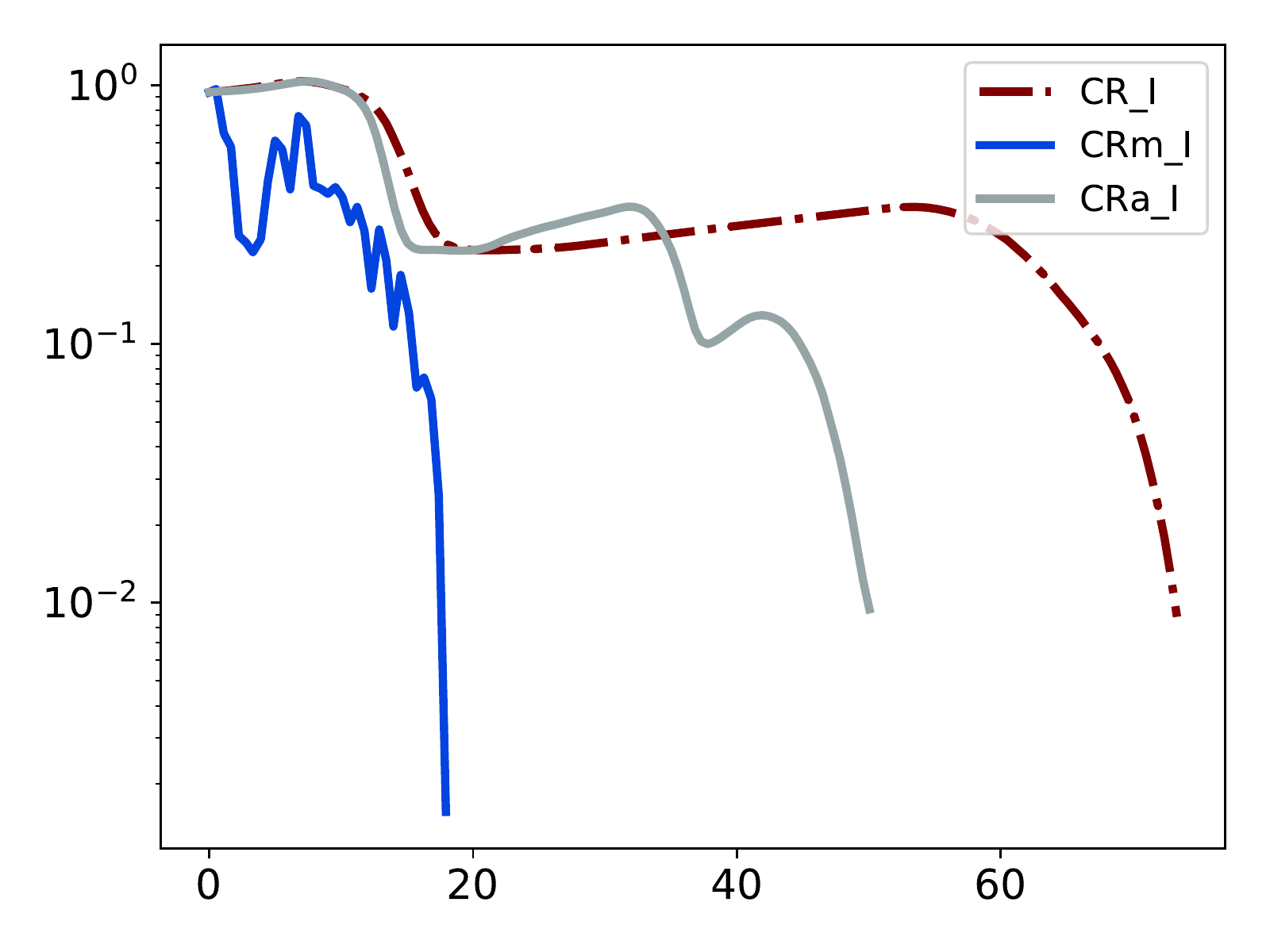}}
	\subfigure{\includegraphics[width=0.32\linewidth]{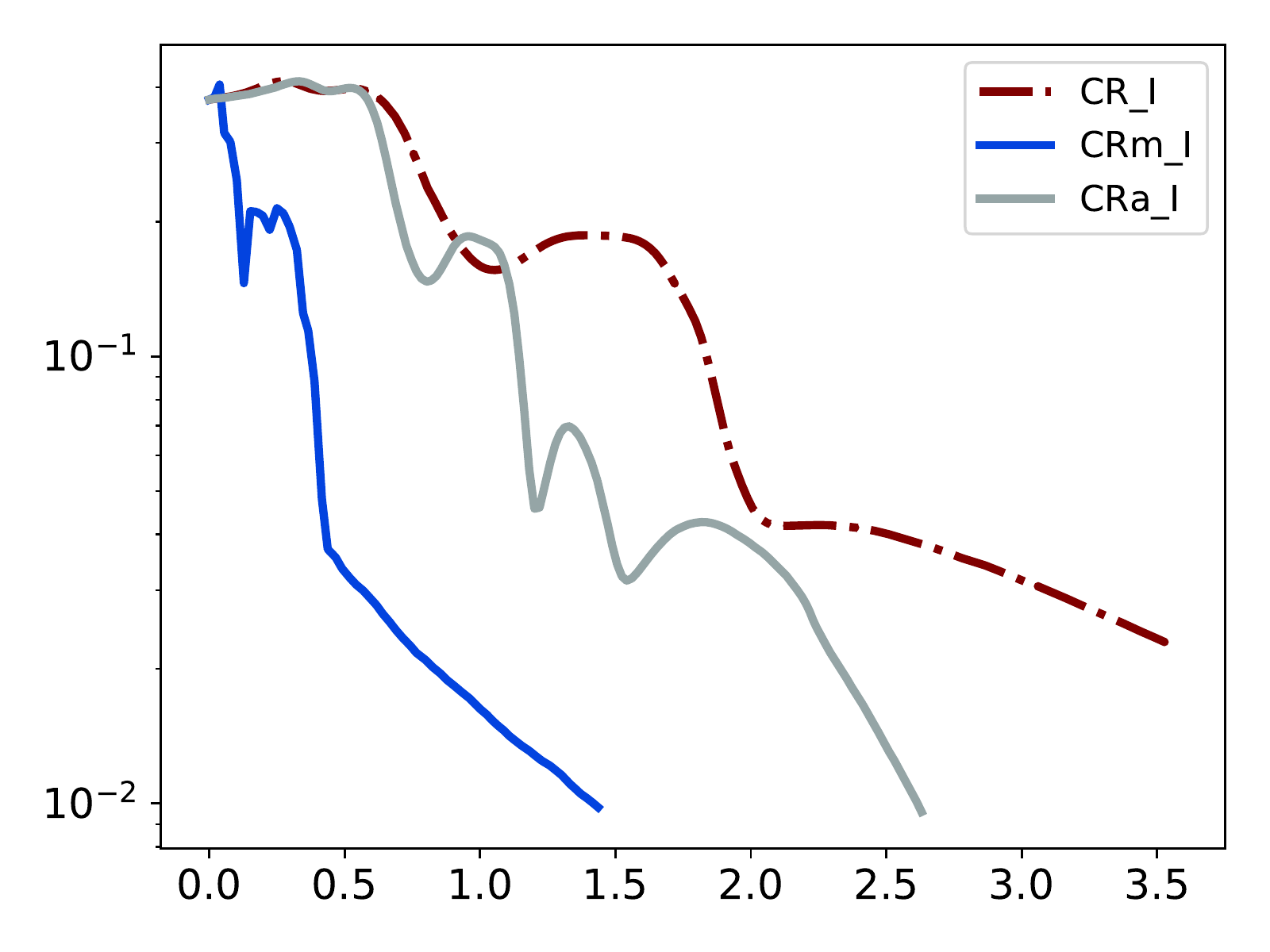}}
	\addtocounter{subfigure}{-3}
	\subfigure[a9a ($n = 32651, d =123$)]{\includegraphics[width=0.32\linewidth]{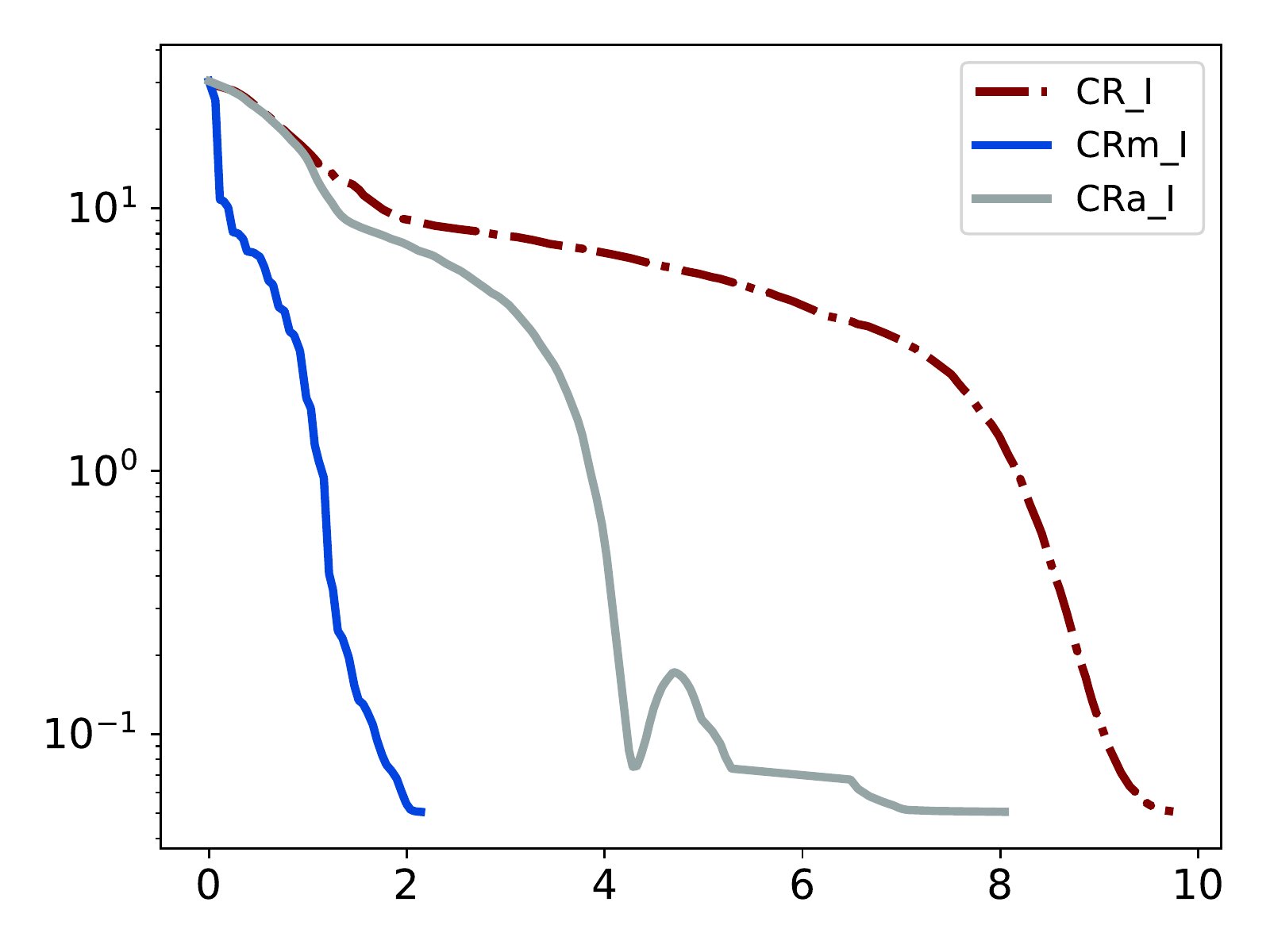}} 
	\subfigure[covtype ($n = 581012, d=54$)]{\includegraphics[width=0.32\linewidth]{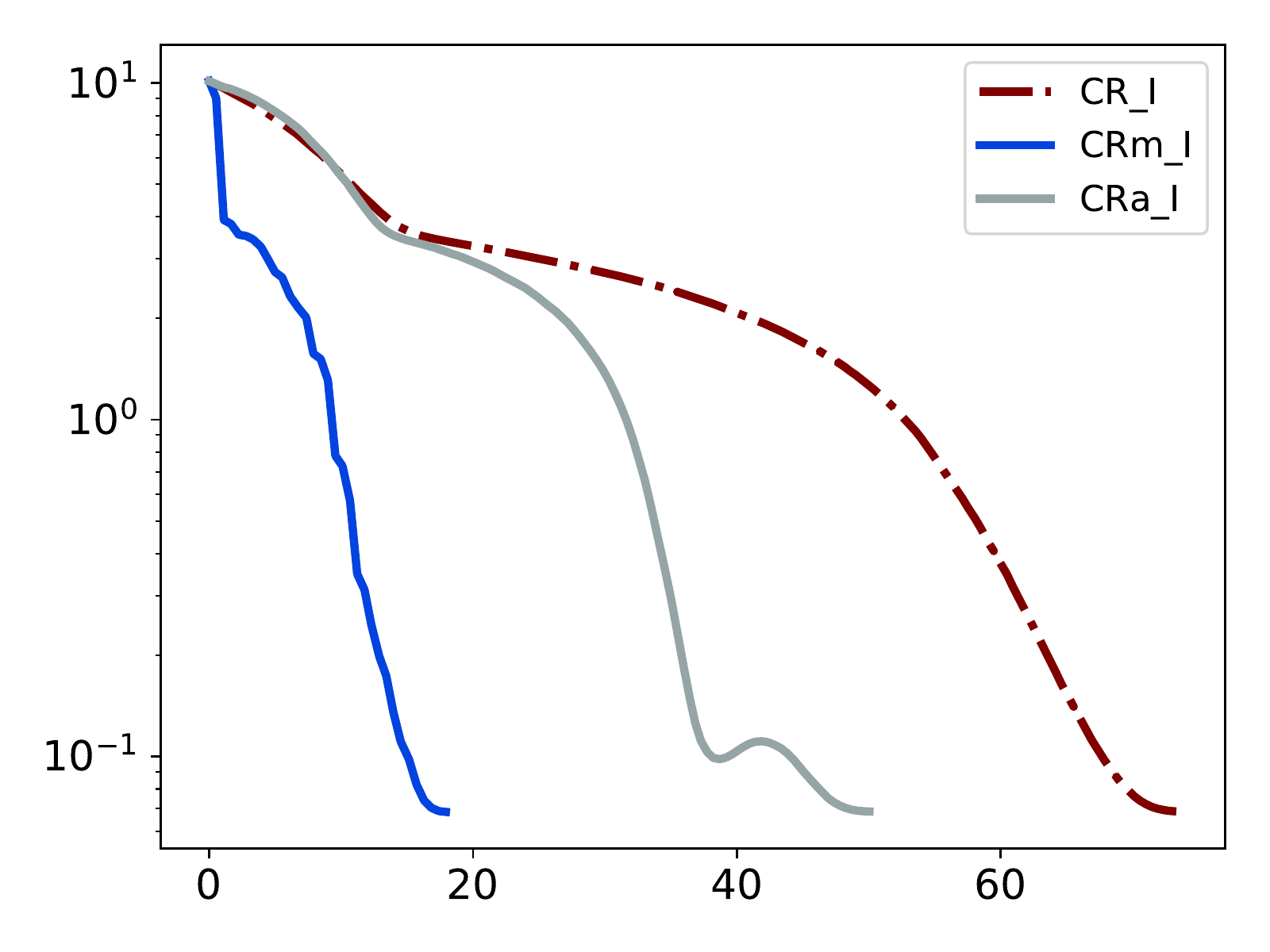}}  
	\subfigure[ijcnn1($n = 35000, d =22$)]{\includegraphics[width=0.32\linewidth]{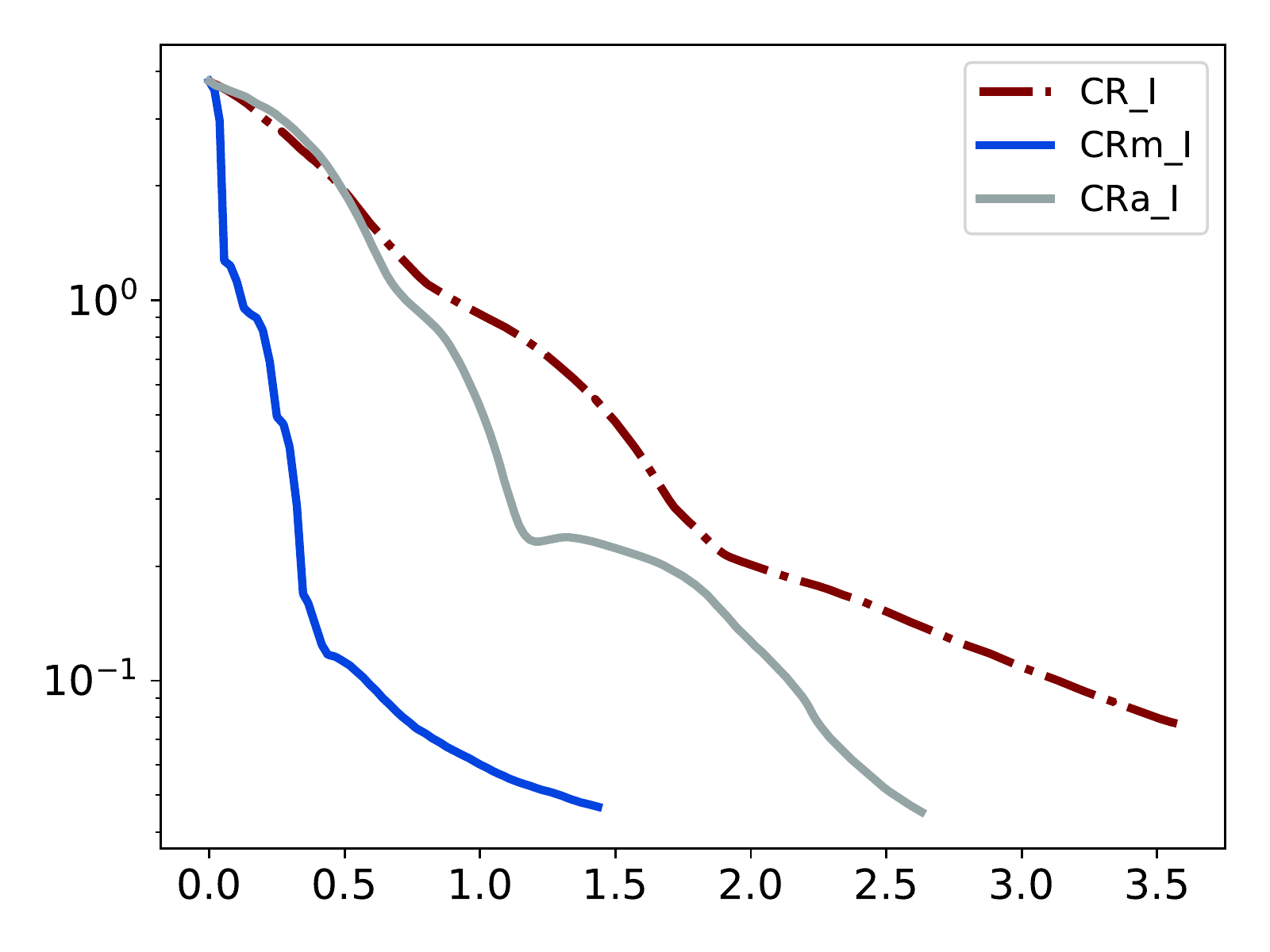}}  
	\caption{Logistic regression loss with a  nonconvex regularizer: The top row presents the gradient norm versus runtime. The bottom row presents the function value gap versus runtime. }    \label{inexact_1}
\end{figure}
\begin{figure}   
	\vspace{-0.3cm}
	\centering 
	\subfigure{\includegraphics[width=0.32\linewidth]{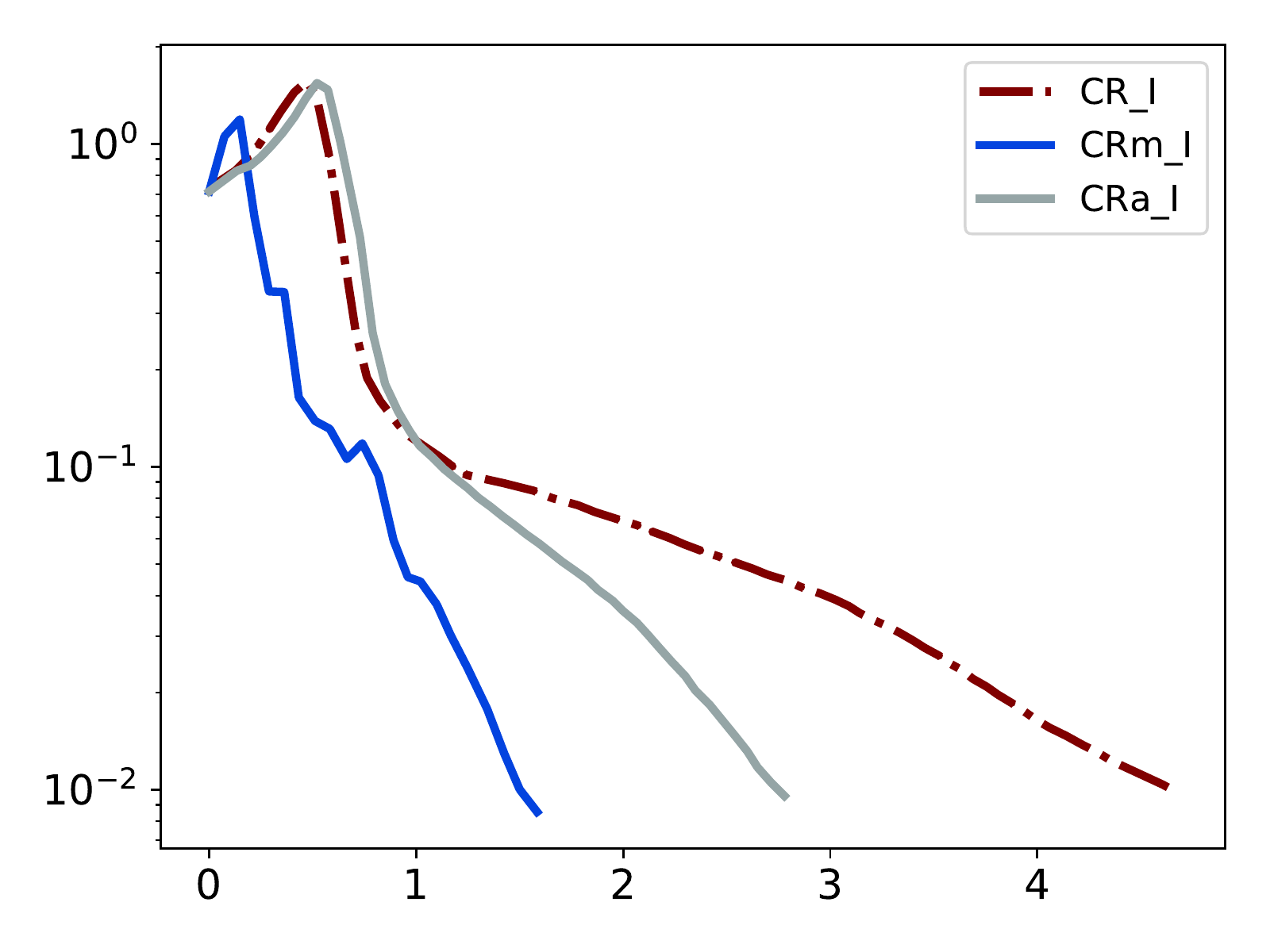}} 
	\subfigure{\includegraphics[width=0.32\linewidth]{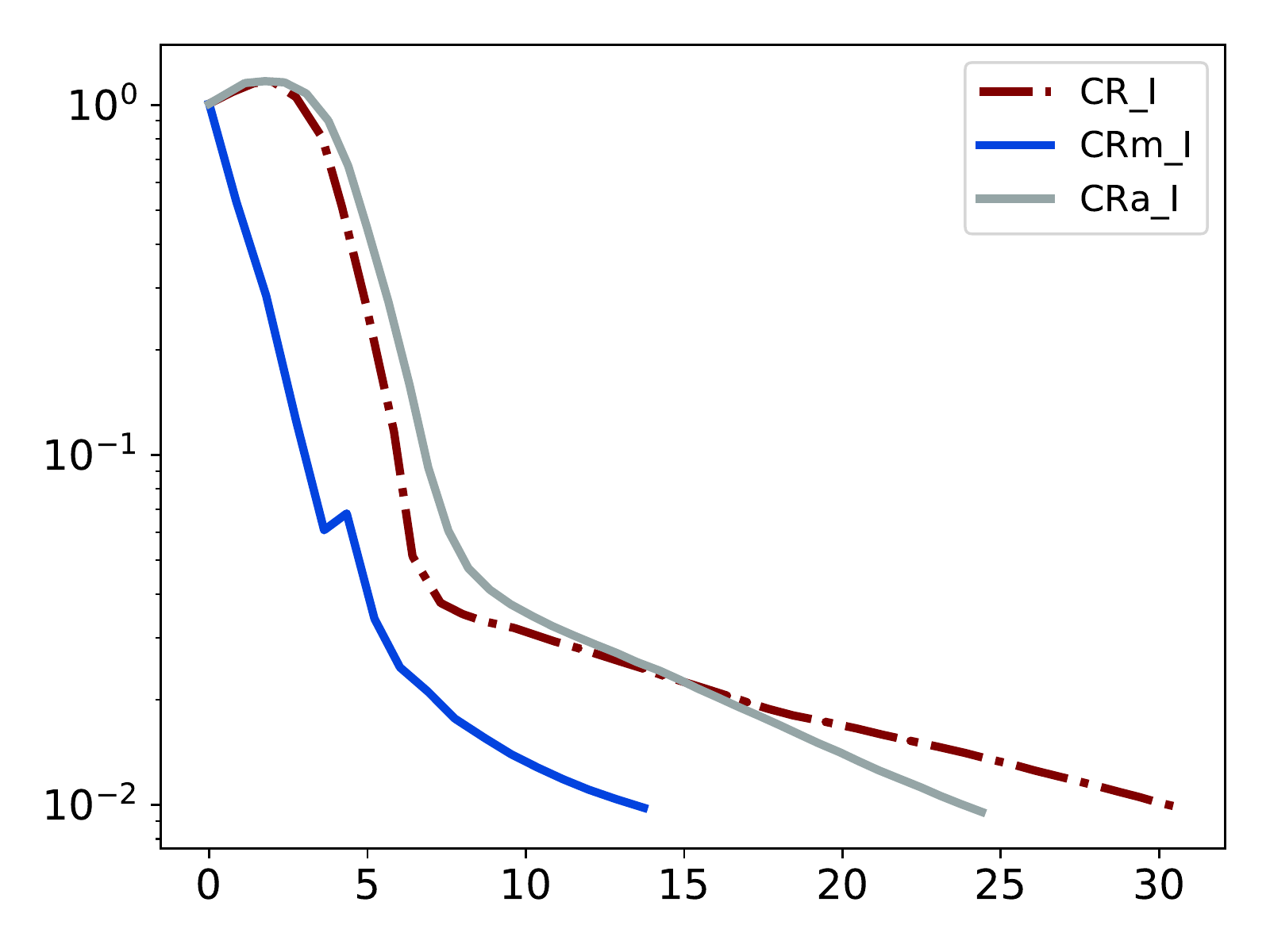}}
	\subfigure{\includegraphics[width=0.32\linewidth]{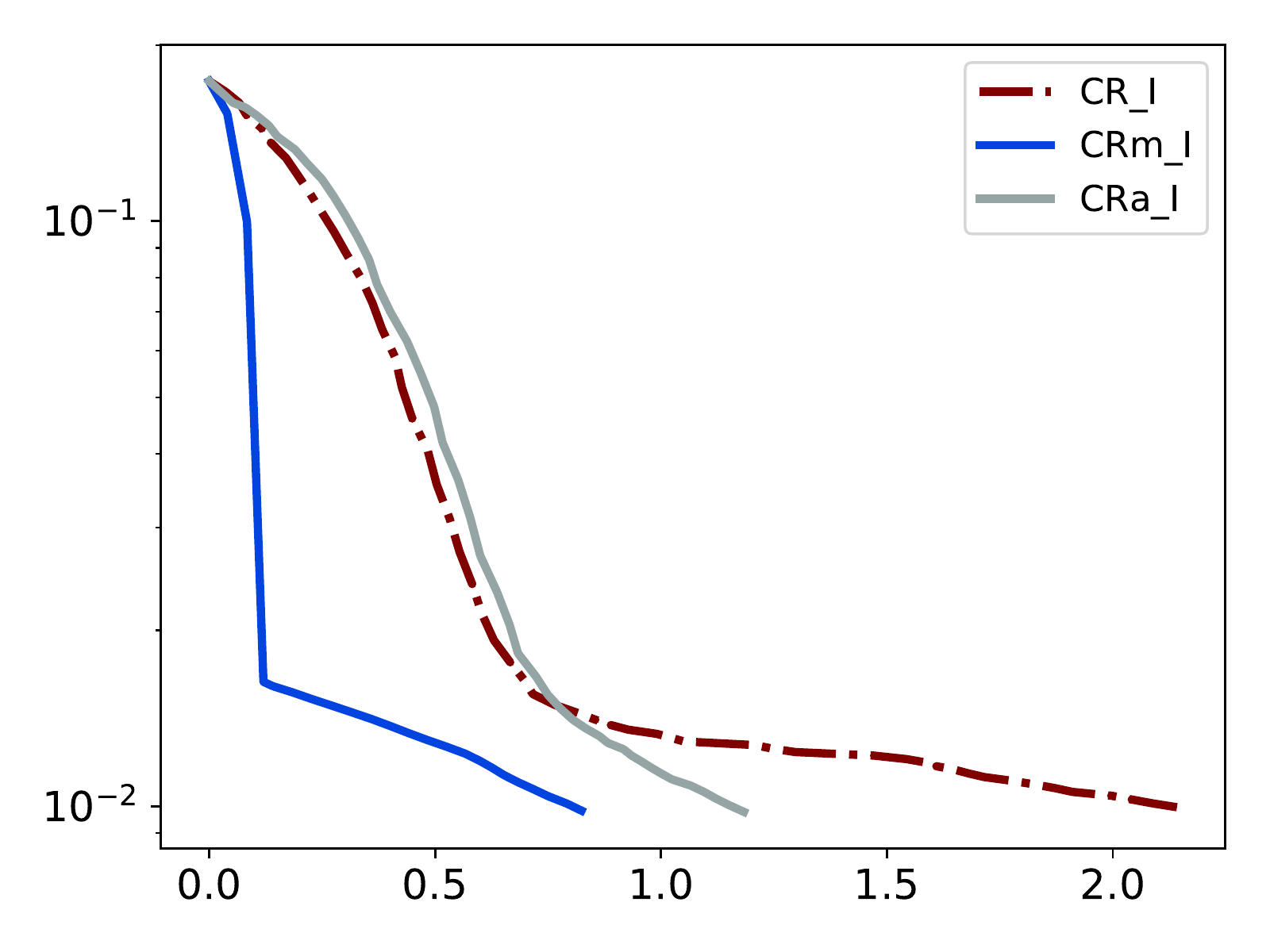}} 
	\addtocounter{subfigure}{-3}
	\subfigure[a9a($n = 32651, d =123$)]{\includegraphics[width=0.32\linewidth]{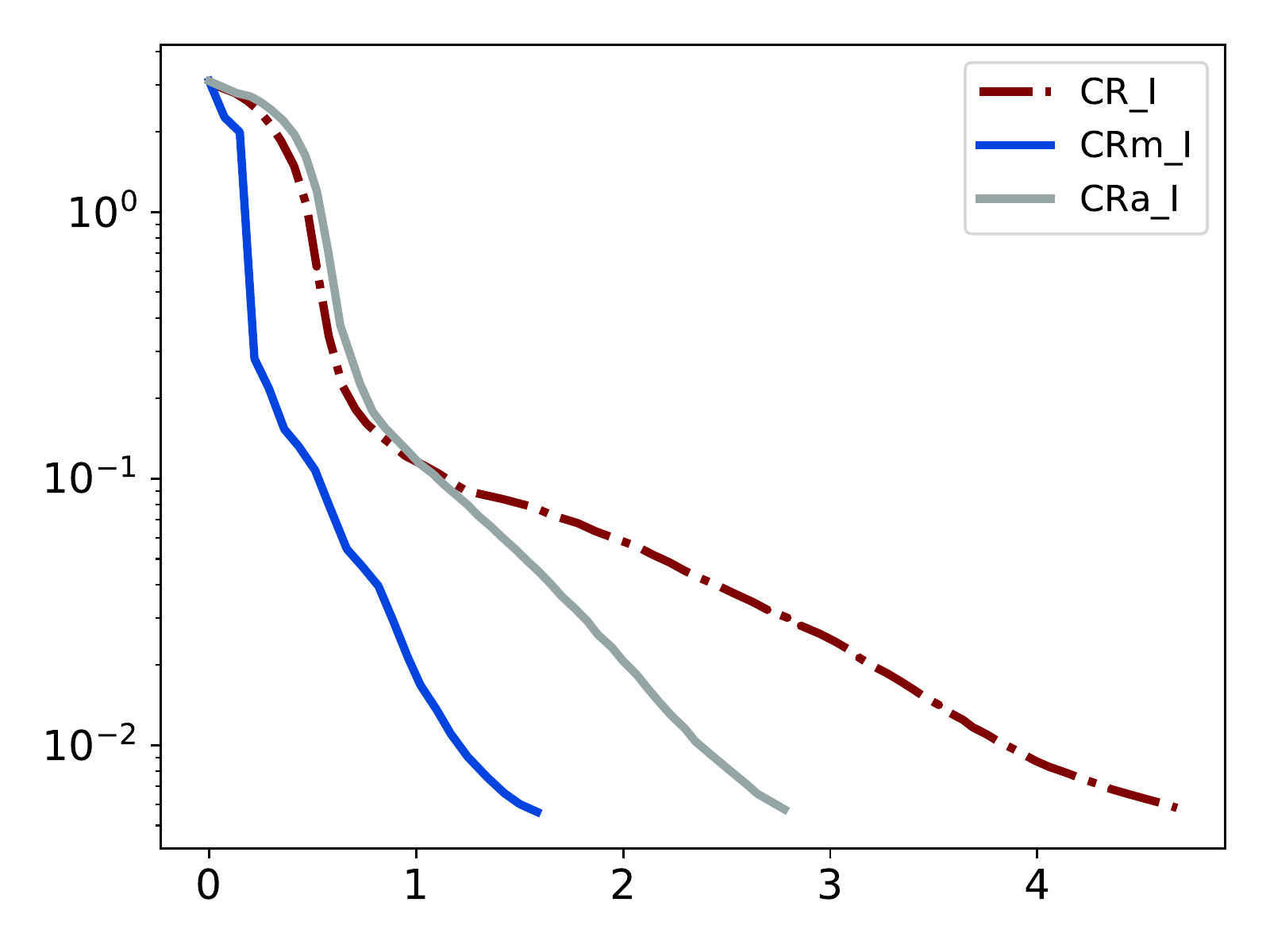}} 
	\subfigure[covtype($n = 581012, d =54$)]{\includegraphics[width=0.32\linewidth]{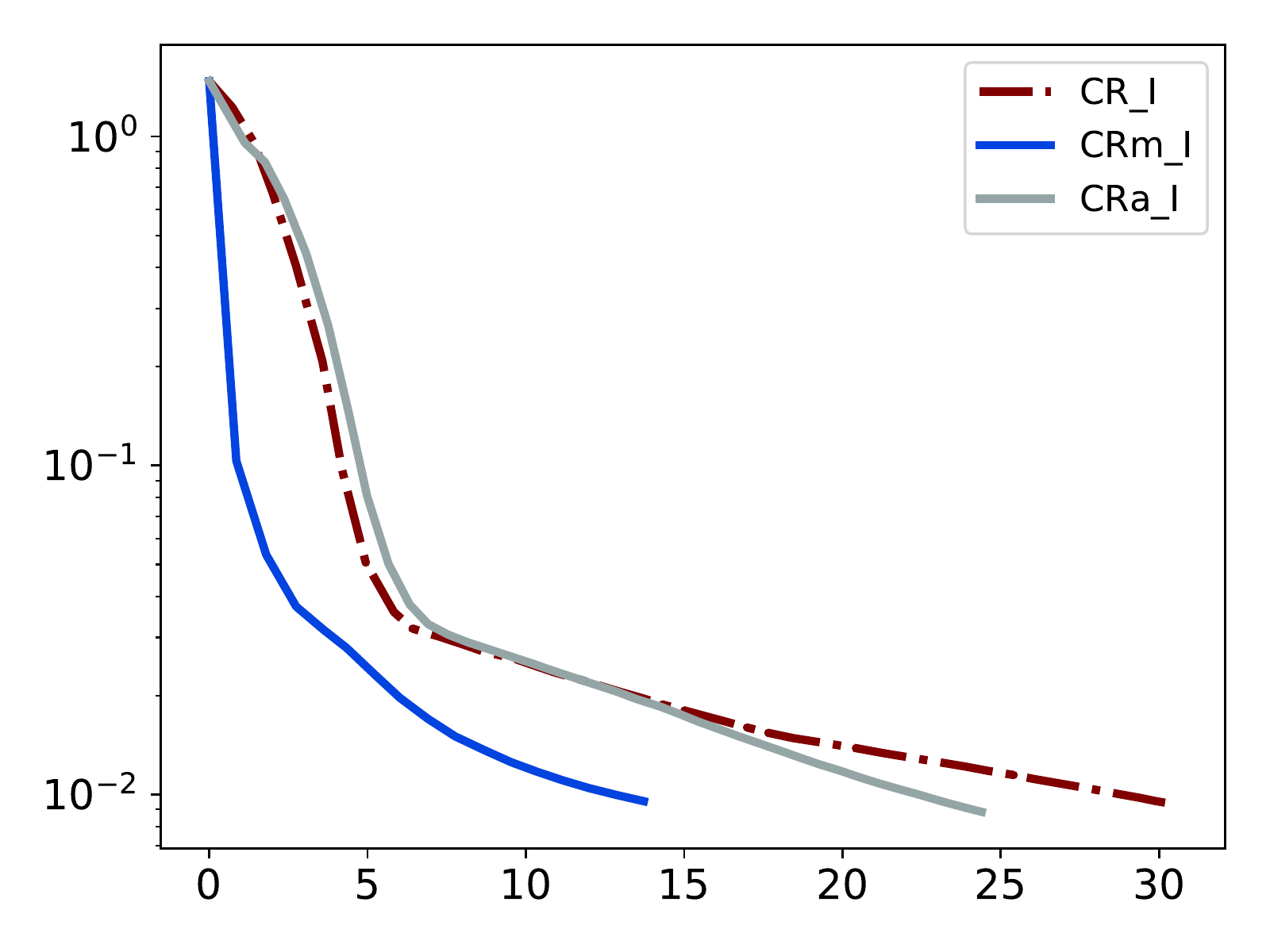}}  
	\subfigure[ijcnn1($n = 35000, d =22$)]{\includegraphics[width=0.32\linewidth]{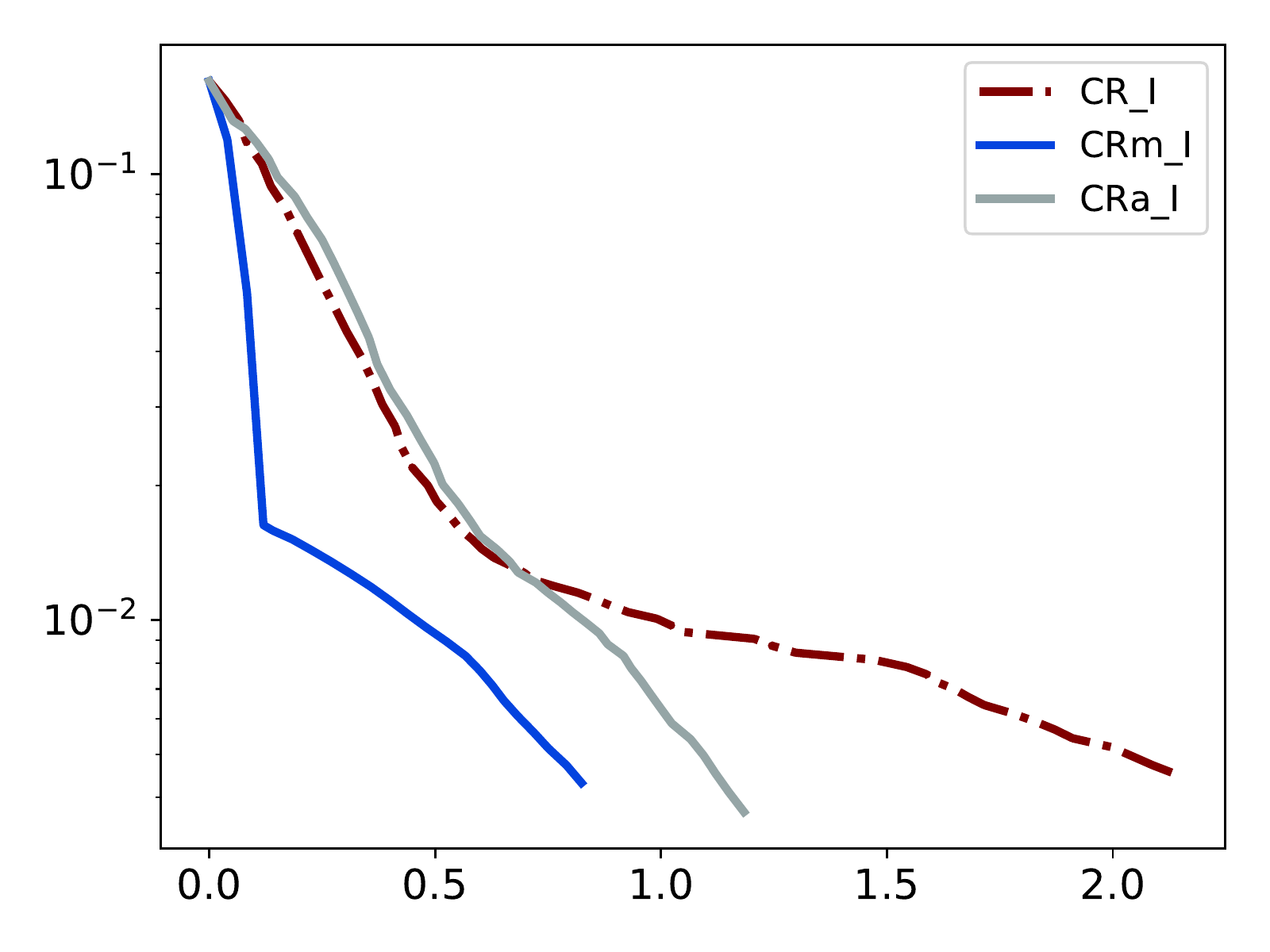}}  
	\caption{Robust linear regression loss: The top row presents the gradient norm versus runtime. The bottom row presents the function value gap versus runtime.}   
	\vspace{-0.3cm}  \label{inexact_2}
\end{figure} 

\subsection{Comparison between  Exact and Inexact Algorithms} \label{ex_inexact_exact}
In this subsection, we present the comparison between the algorithms with exact Hessian  and the algorithms with inexact Hessian.  The results are shown in \Cref{ie_1,ie_2}. It is clear that all inexact algorithms significantly outperform their corresponding exact algorithms. This demonstrates the efficiency of the inexact Hessian technique for second-order algorithms. 
\begin{figure}[H]
	\vspace{-0.3cm}
	\centering 
	\subfigure{\includegraphics[width=0.32\linewidth]{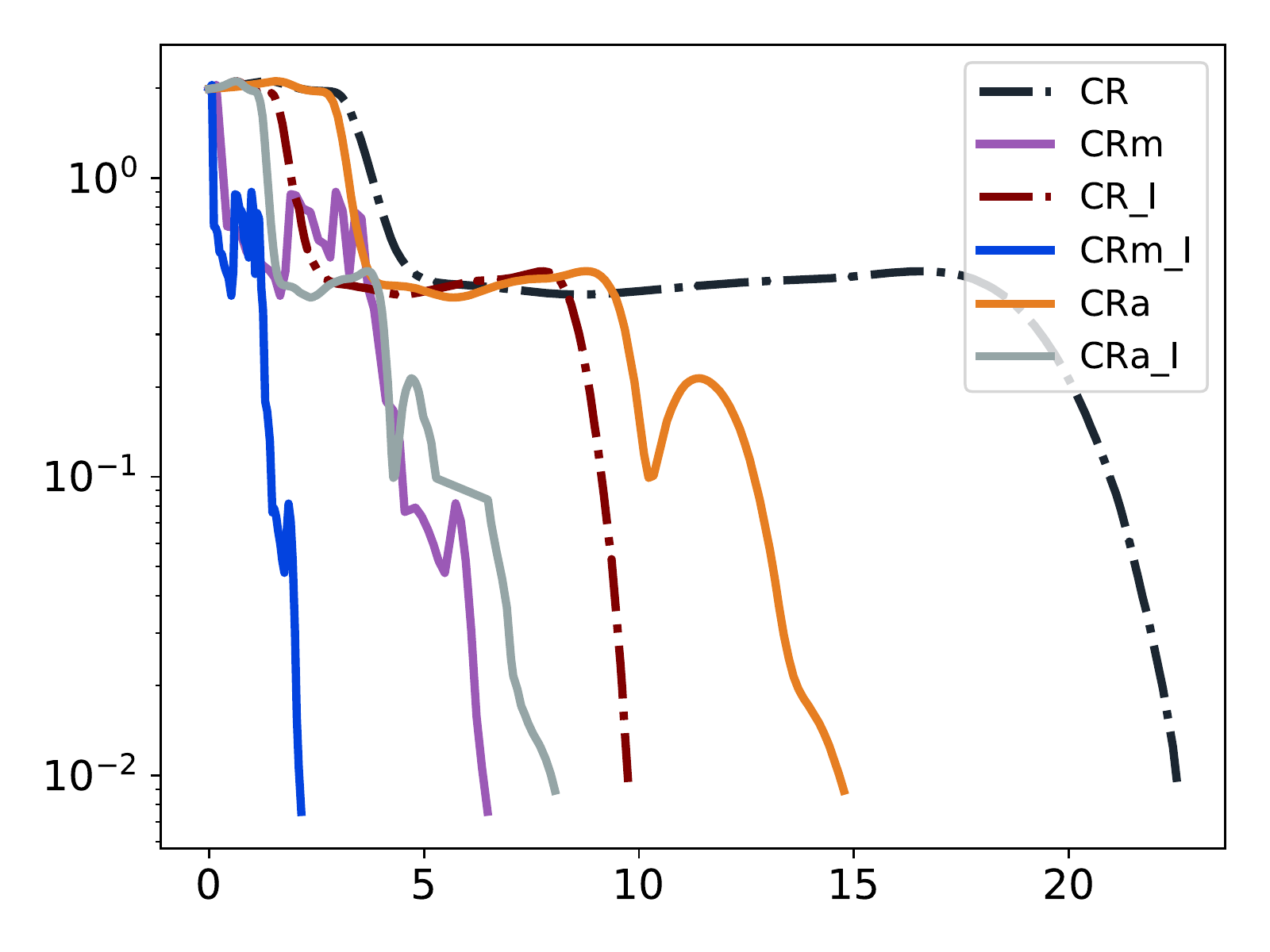}} 
	\subfigure{\includegraphics[width=0.32\linewidth]{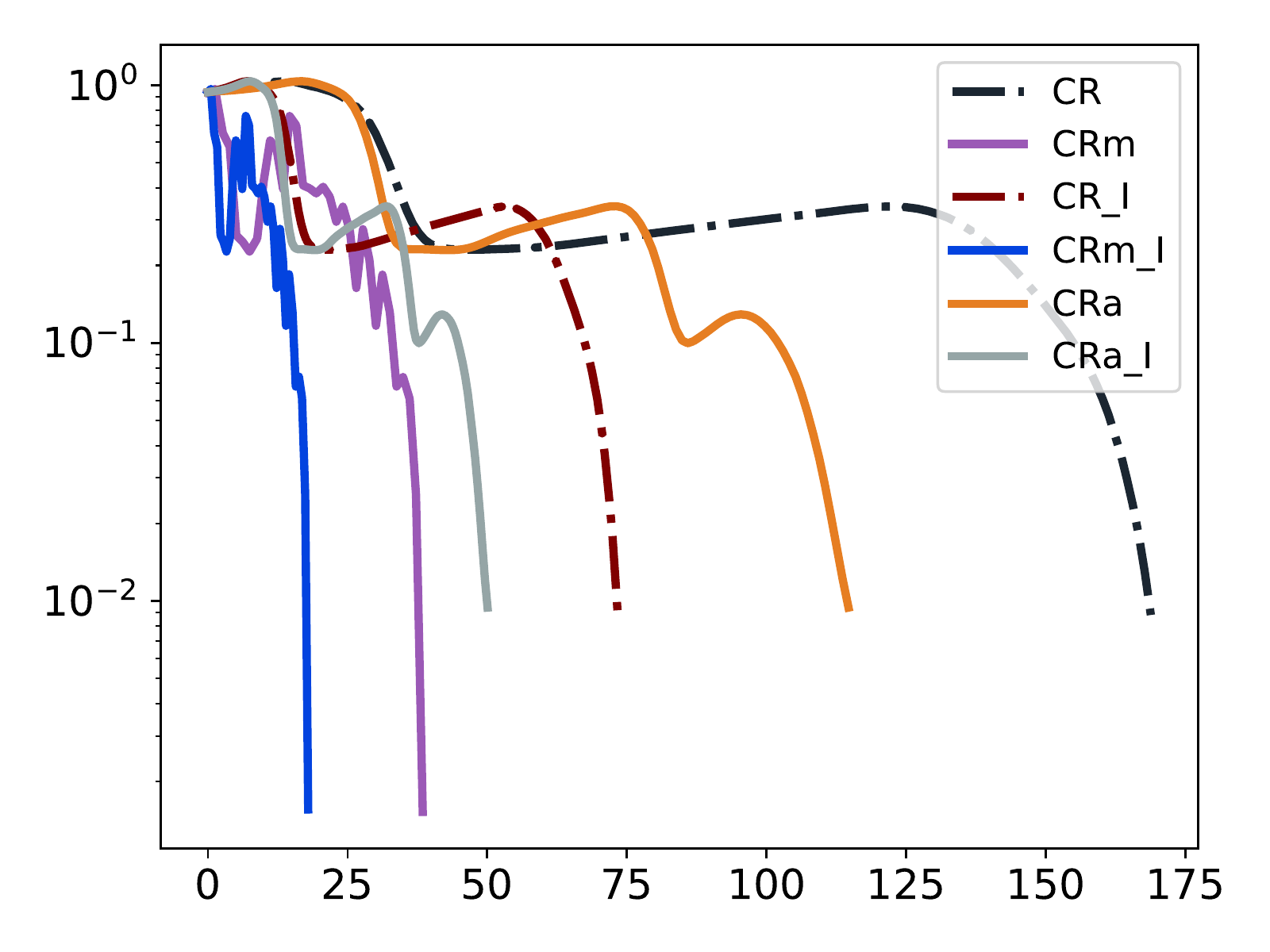}}
	\subfigure{\includegraphics[width=0.32\linewidth]{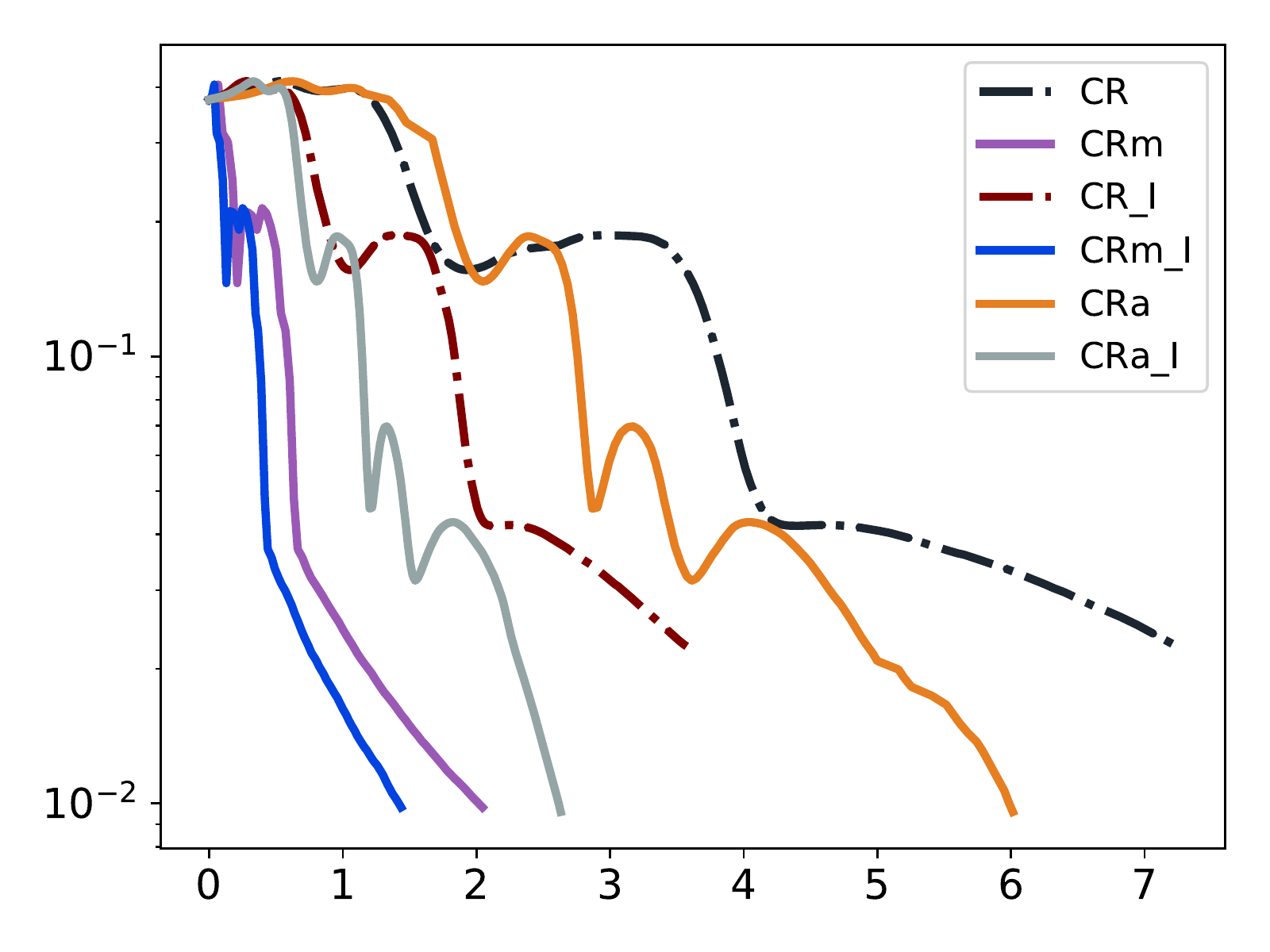}}\\
	\addtocounter{subfigure}{-3}
	\subfigure[a9a ($n = 32651, d =123$)]{\includegraphics[width=0.32\linewidth]{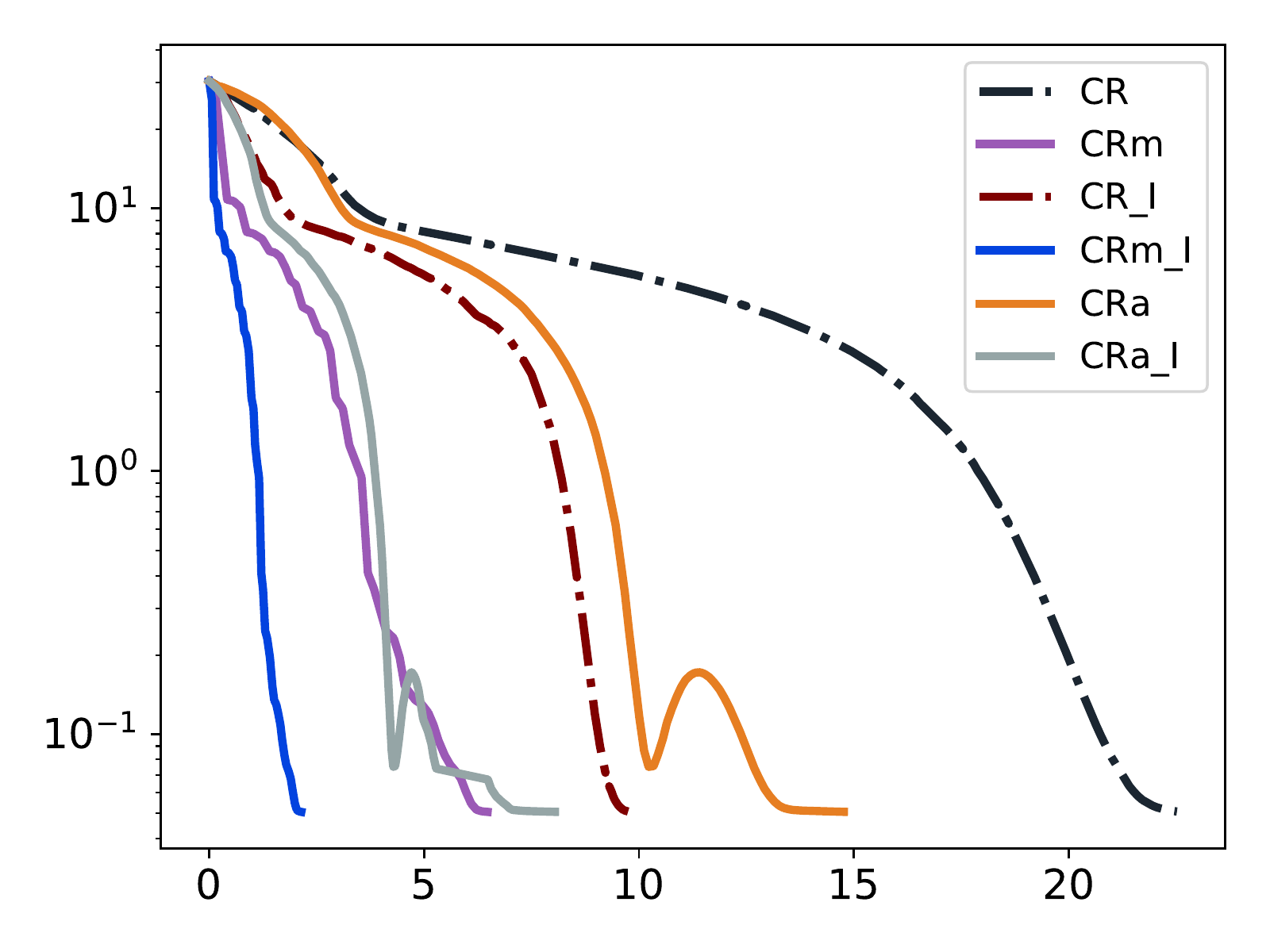}} 
	\subfigure[covtype ($n = 581012, d=54$)]{\includegraphics[width=0.32\linewidth]{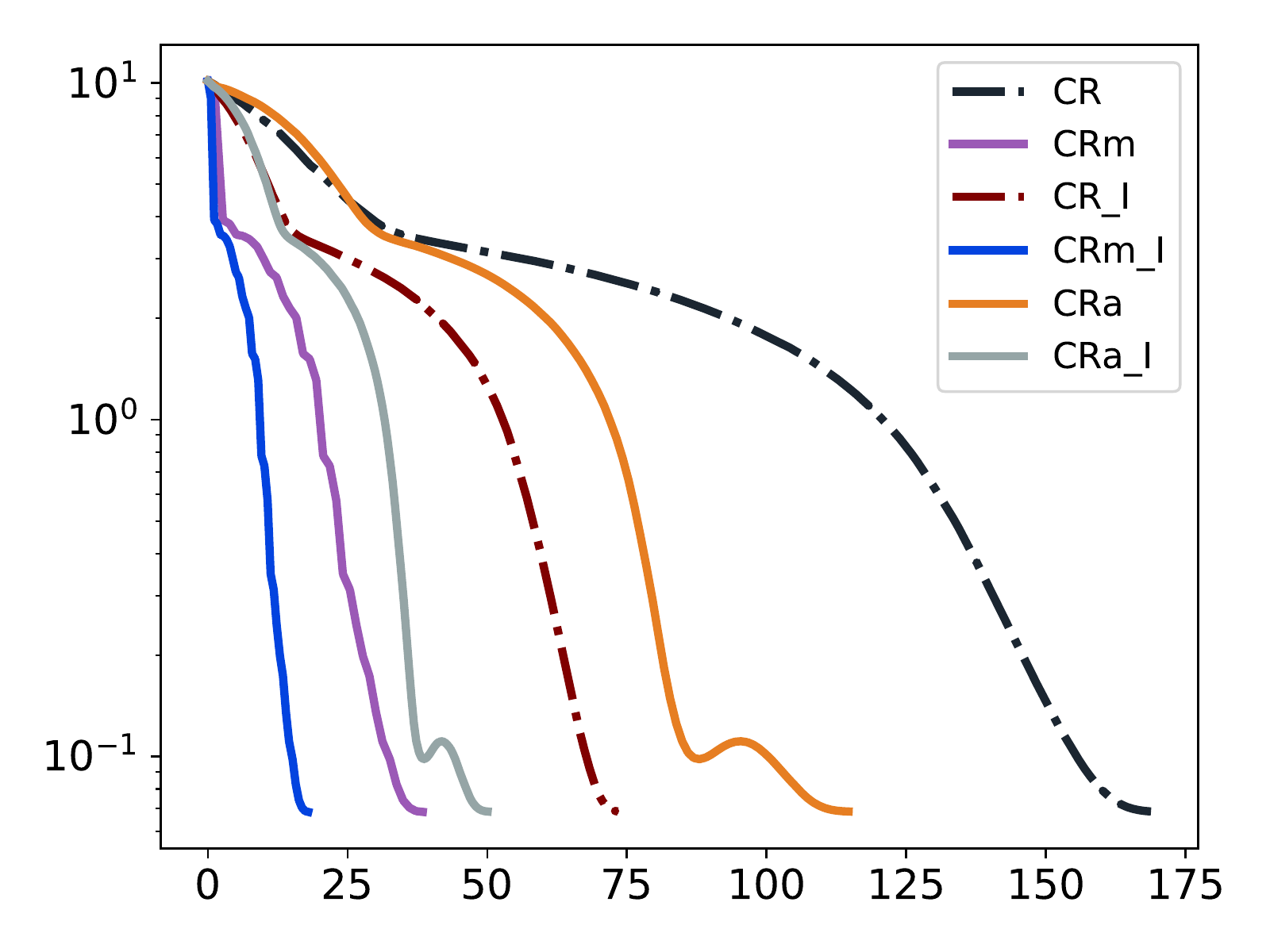}}  
	\subfigure[ijcnn1($n = 35000, d =22$)]{\includegraphics[width=0.32\linewidth]{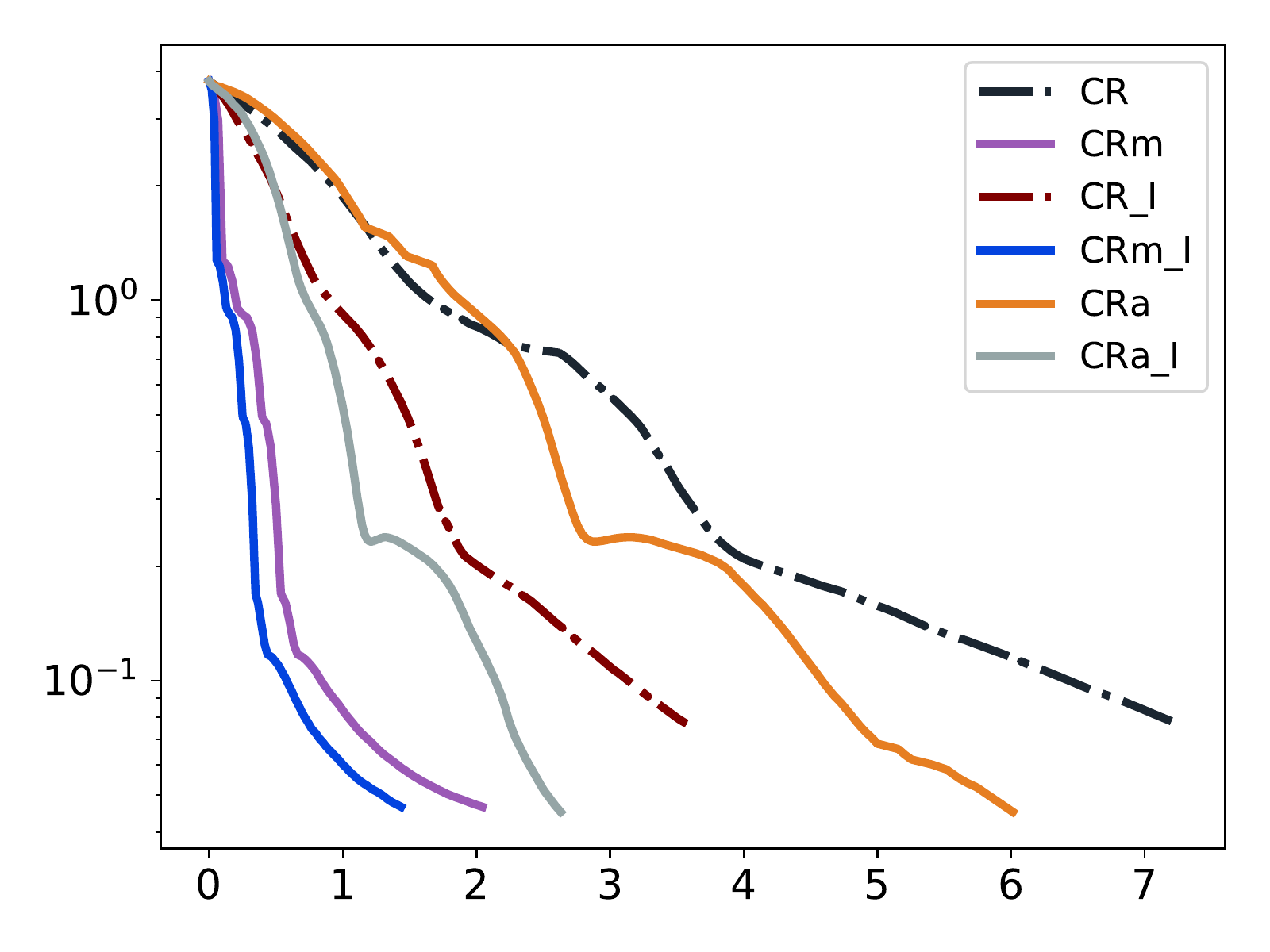}}  
	\caption{Nonconvex logistic regression. Top: gradient norm v.s. time. Bottom: function value gap v.s. time. }    \label{ie_1}
\end{figure}
\begin{figure} [h]  
	\vspace{-0.3cm}
	\centering 
	\subfigure{\includegraphics[width=0.32\linewidth]{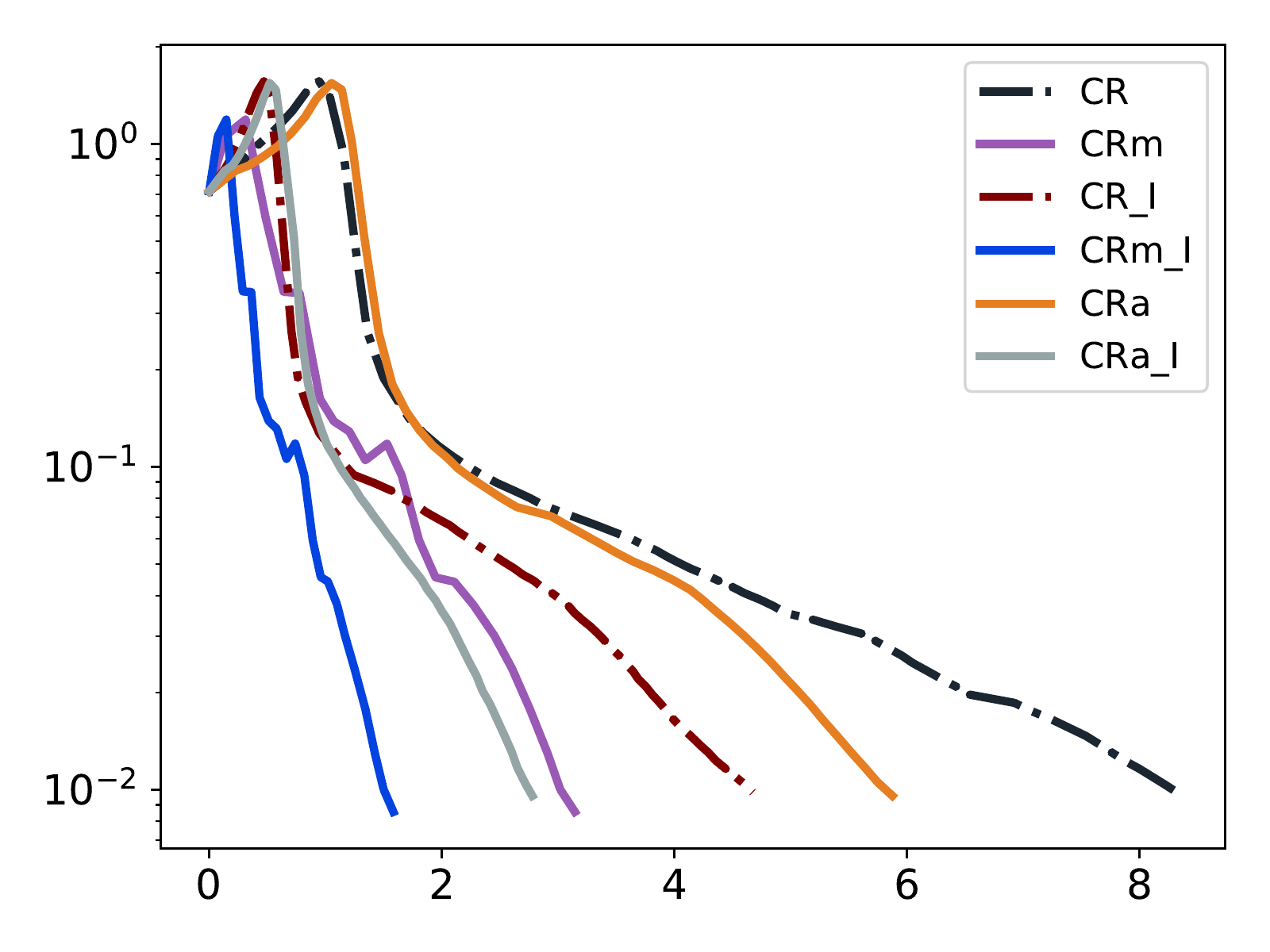}} 
	\subfigure{\includegraphics[width=0.32\linewidth]{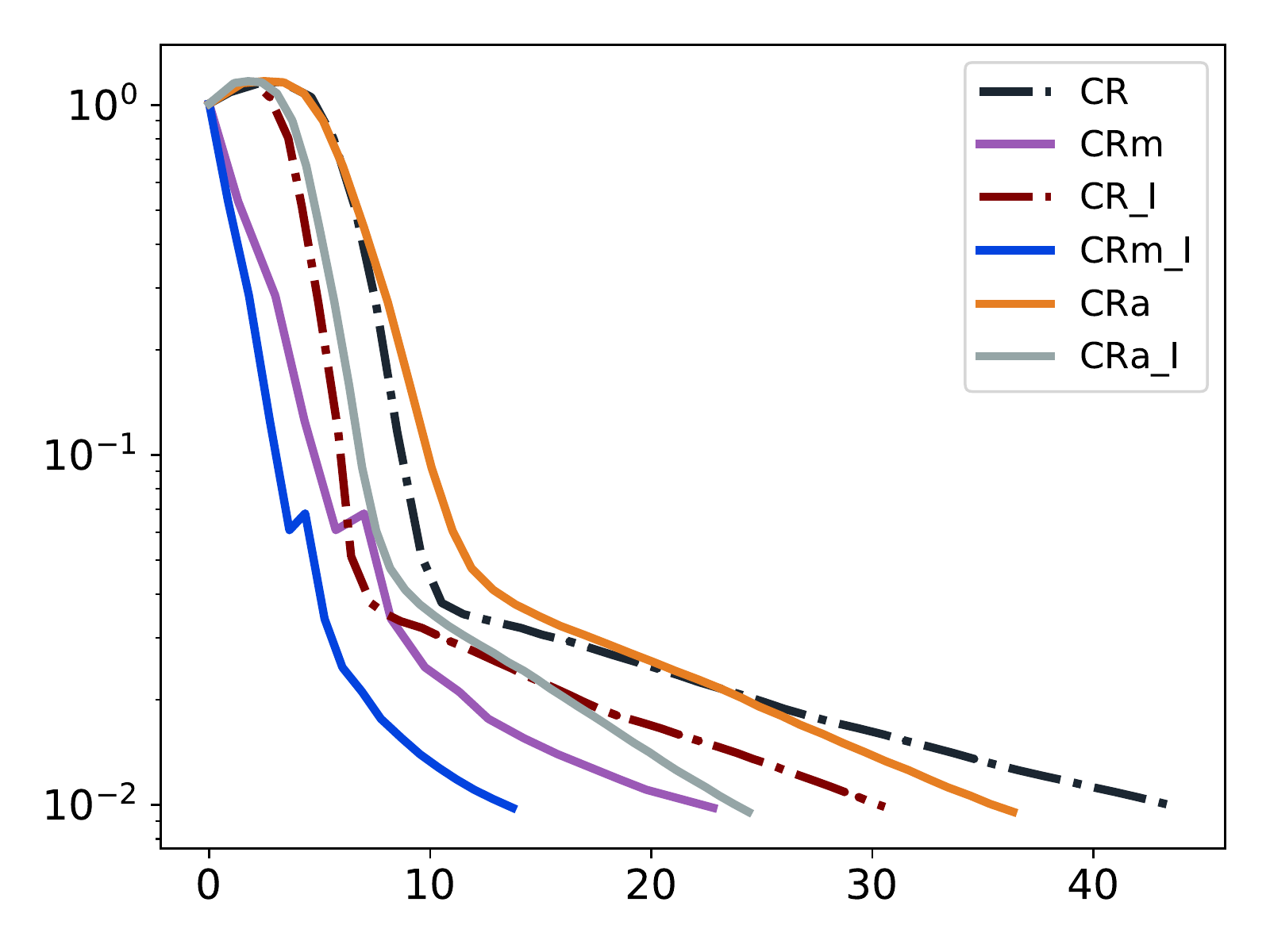}}
	\subfigure{\includegraphics[width=0.32\linewidth]{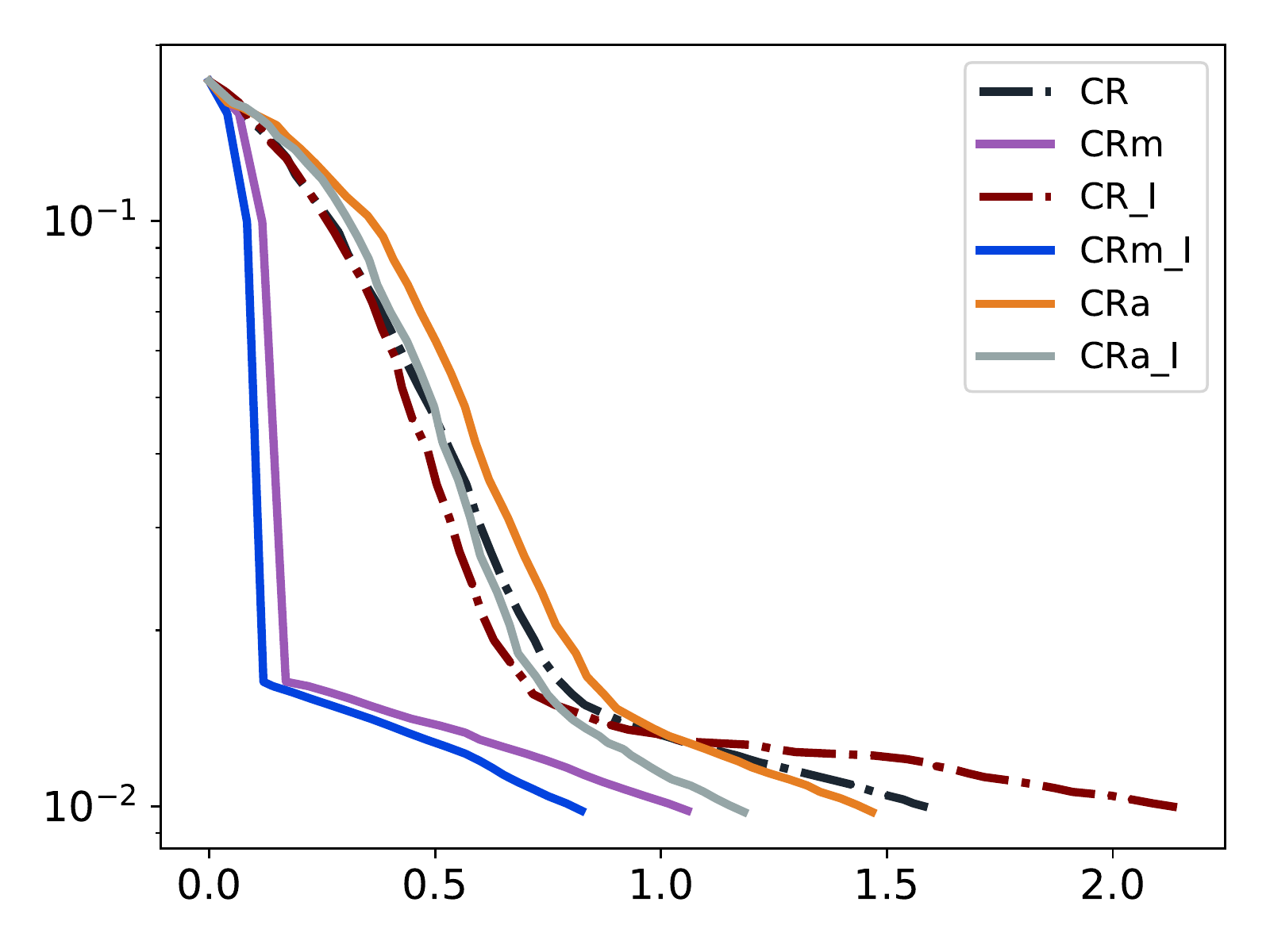}} 
	\addtocounter{subfigure}{-3}
	\subfigure[a9a($n = 32651, d =123$)]{\includegraphics[width=0.32\linewidth]{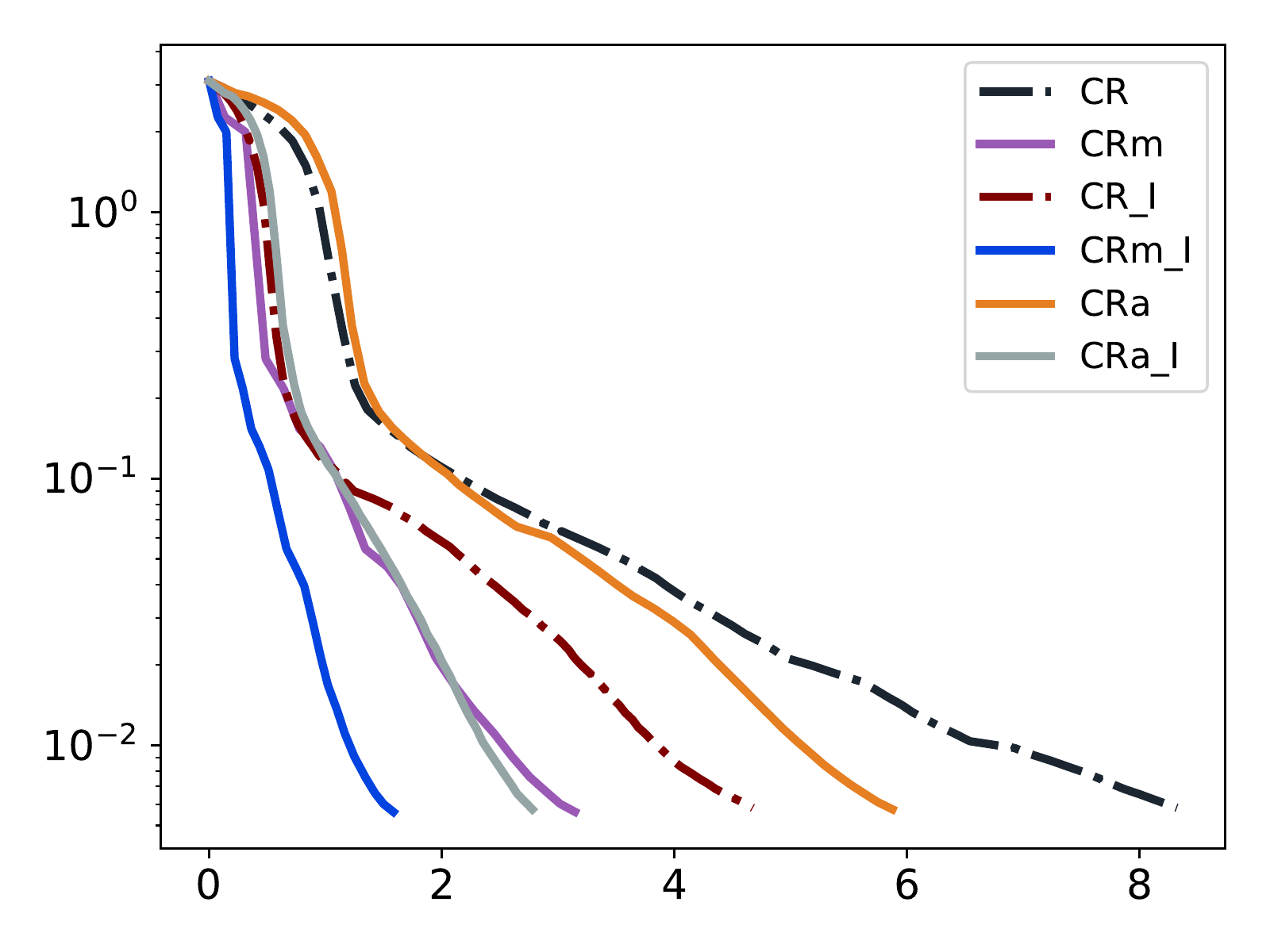}} 
	\subfigure[covtype($n = 581012, d =54$)]{\includegraphics[width=0.32\linewidth]{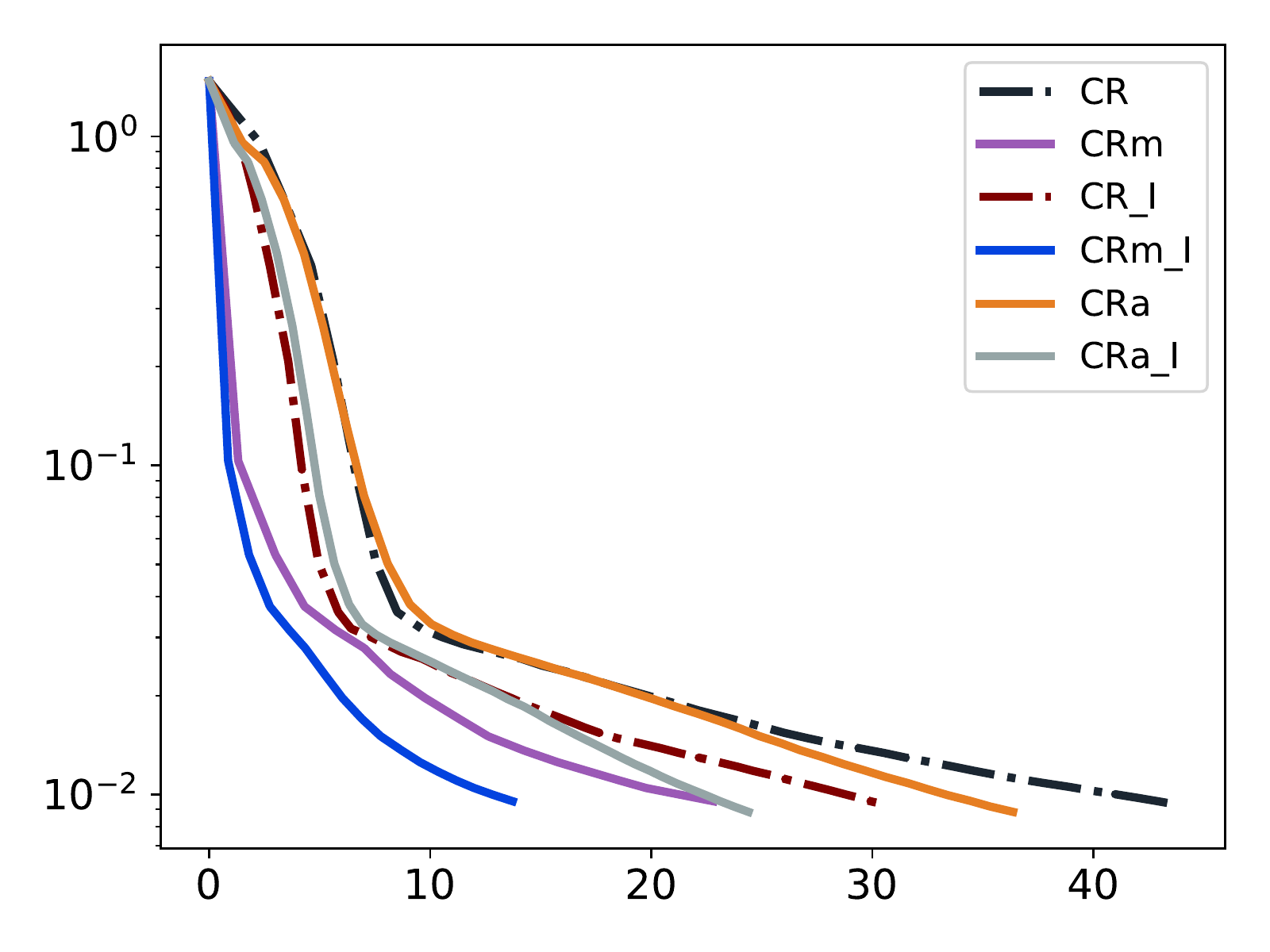}}  
	\subfigure[ijcnn1($n = 35000, d =22$)]{\includegraphics[width=0.32\linewidth]{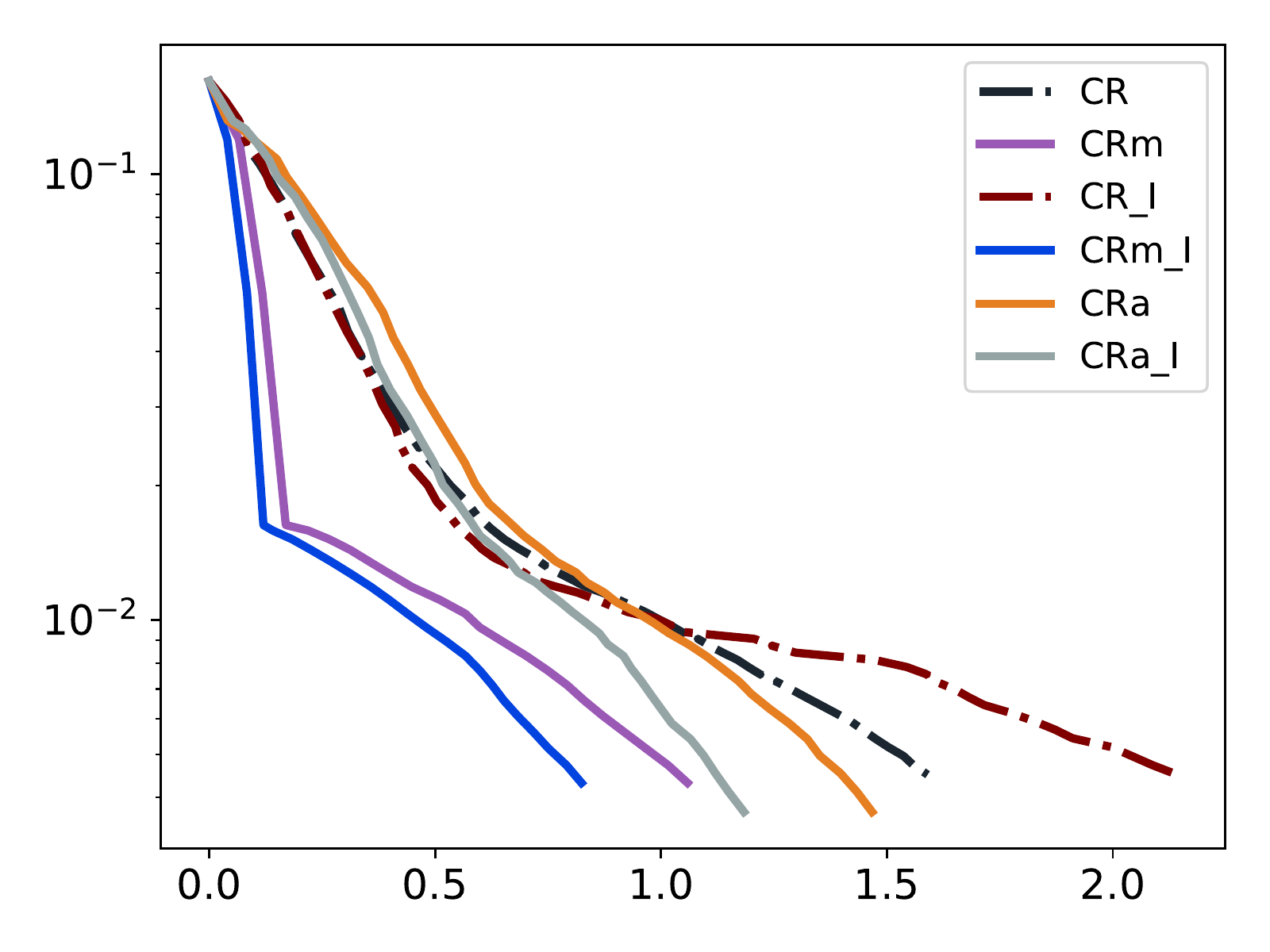}}  
	\caption{Robust linear regression loss: The top row presents the gradient norm versus runtime. The bottom row presents the function value gap versus runtime.}   
	\vspace{-0.3cm} \label{ie_2}
\end{figure}

\section{Technical Lemmas}
In this section, we introduce several technical lemmas that are useful for proving our main results.  


\begin{lemma}[\cite{Nesterov2006}, Lemma 1] \label{Hessian_square_bound}
	Let the Hessian $\nabla^2 f$ of the function  $f$ be $L_2$-Lipschitz continuous with $L_2 > 0$. Then, for any $\x, \mathbf{y} \in \mathbb{R}^d$, we have
	\begin{align*}
	\norml{\nabla f( \mathbf{y} ) - \nabla f(\x) - \nabla^2 f(\x )(\mathbf{y} - \x ) } &\leqslant \frac{L_2}{2} \norml{\mathbf{y}  - \x}^2. \\
	|f(\mathbf{y}) - f(\x) - \nabla f(\x)^T(\mathbf{y}  - \x) - \frac{1}{2} (\mathbf{y}  - \x)^T &\nabla^2 f(\x )(\mathbf{y} - \x )|  \leqslant \frac{L_2}{6}  \norml{\mathbf{y}  - \x}^3.
	\end{align*}
\end{lemma}

The following lemma provides bounds on $\norml{\hat{\x}_{k+1} - \x_k}$ as shown in \cite[Lemma 1]{Yue2018}. 

\begin{lemma}[\cite{Yue2018}, Lemma 1] \label{ada_lemma_1}
  	Let \Cref{Assumption_1} hold. Then, the sequences $\{ \x_k\}_{k \geqslant 0}$ and $\{\hat{\x}_{k} \}_{k\geqslant 1}$ generated by Algorithm \ref{Adaptive_version} satisfies, for all $k \geqslant 0$,
  	\begin{align}
  	\norml{\hat{\x}_{k+1} - \x_k} \leqslant c_1 \cdot \textrm{dist}(\x_k, \mathcal{X}), 
  	\end{align}
  	where $c_1 \triangleq \left( 1 + \frac{L_2}{M} + \sqrt{\left(1+ \frac{L_2}{M}\right)^2 + \frac{L_2}{M}}\right)$.
\end{lemma}


Next, we develop a number of useful bounds regarding the sequences $\{ \x_k\}_{k \geqslant 0}$ and $\{\hat{\x}_{k} \}_{k\geqslant 1}$ that are generated by  \Cref{Adaptive_version}. We refer to \Cref{proof_of_lemmas} for the details of the proofs.
 \begin{lemma} \label{after_cubic}
 	Let \Cref{Assumption_1} hold. Then, the sequences $\{\x_k\}_{k\geqslant0}$ and  $\{\hat{\x}_k\}_{k\geqslant0}$ generated by Algorithms \ref{Adaptive_version}  with $M > \frac{2L_2}{3}$ satisfy, for all $k \geqslant 0$,
 	\begin{align}
 	f(\hat{\x}_{k+1}) - f(\x_k) &\leqslant -\gamma \norml{\hat{\x}_{k+1}- \x_k}^3, \label{momentum_12}\\
 	\sum_{i=0}^{k}    \norml{\hat{\x}_{k+1}- \x_k}^3  &\leqslant\frac{ f(\x_0) - f^\star }{\gamma}, \label{momentum_4}\\
 	\norml{\nabla f(\hat{\x}_{k+1})} &\leqslant  \frac{L_2 + M}{2} \norml{\hat{\x}_{k+1} - \x_k}^2, \label{momentum_7}\\
 	\lambda_{\min} ( \nabla^2 f(\hat{\x}_{k+1}) ) &\geqslant - \frac{M+2L_2}{2}  \norml{\hat{\x}_{k+1} - \x_k}, \label{momentum_14}
 	\end{align}
 	where $\gamma = \frac{3M-2L_2 }{12}$. Furthermore, there exists a $k_0 \in \{0, \cdots k\}$ such that
 	\begin{align}
 	\norml{\hat{\x}_{k_0+1}- \x_{k_0}}  \leqslant  \frac{1}{k^{1/3}} \left(  \frac{ f(\x_0) - f^\star }{  \gamma}\right)^{1/3}. \label{momentum_8}
 	\end{align}
 \end{lemma}
 The following lemma establishes the corresponding bounds regarding $\{ \x_k\}_{k \geqslant 0}$ and $\{\hat{\x}_{k} \}_{k\geqslant 1}$ that are generated by the {\em inexact} variants of Algorithms \ref{Adaptive_version}. We refer to \Cref{proof_of_lemmas} for the details of the proof.
\begin{lemma} \label{im_important_cubic_lemma}
   	Let \Cref{Assumption_1} and \Cref{assump_2} hold. Then, the sequences $\{\x_k\}_{k\geqslant0}$ and  $\{\hat{\x}_k\}_{k\geqslant0}$ generated by the inexact variants of Algorithms \ref{Adaptive_version} with $M > \frac{2L_2}{3} + 2$ satisfy, for all $k \geqslant 0$ and any $\epsilon_1 >0$,
   	\begin{align} 
   	f(\hat{\x}_{k+1}) - f(\x_k) &\leqslant - \frac{3M-2L_2}{12} \norml{\hat{ \x}_{k+1} - \x_k}^3 + \frac{1}{2} \norml{\hat{ \x}_{k+1} - \x_k}^2 \epsilon_1, \label{im0}\\
   	\norml{\nabla f(\hat{\x}_{k+1})} &\leqslant \frac{L_2 + M}{2}\norml{\hat{ \x}_{k+1} - \x_k}^2 + \epsilon_1 \norml{\hat{ \x}_{k+1} - \x_k}, \label{i_momentum_7}\\
   	\lambda_{\min} ( \nabla^2 f(\hat{\x}_{k+1}) ) &\geqslant  -\frac{M + 2L_2}{2} \norml{\hat{ \x}_{k+1} - \x_k} - \epsilon_1. \label{i_momentum_14}
   	\end{align}
   	Furthermore, if the total number of iterations $k >  \left( \frac{12}{3M-2L_2 -6} \right) \frac{  f(\x_0) - f^{\star}} { \epsilon_1^3 }$ for any $\epsilon_1 >0$, then there exists a $k_0 \in \{0, \cdots k\}$ such that
   	\begin{align}
   	\norml{\hat{\x}_{k_0+1}- \x_{k_0}}  \leqslant   \epsilon_1 . \label{i_momentum_8}
	\end{align} 
\end{lemma}

Note that in  CRm, $\hat{\x}_{k+1}$ is generated by $\x_{k}$ through a cubic regularization step, and  the output sequence $\{ f(\x_k)\}_{k \geqslant 0}$ is monotone through a monotone step. Therefore, \Cref{ada_lemma_1,after_cubic,im_important_cubic_lemma} hold for CRm.

\section{Global Convergence: Proof of  \Cref{Adaptive_algorithm_convergence}} \label{proof_s_1}  
 We first prove a useful inequality. Note that  
 \begin{align*}
 \norml{\x_{k+1} - \hat{\x}_{k+1}} & \numleqslant{i}   \max \left(\norml{\hat{\x}_{k+1} - \hat{\x}_{k+1}} ,\norml{\tilde{\x}_{k+1} - \hat{\x}_{k+1}} \right) \\
 &=  \norml{\tilde{\x}_{k+1} - \hat{\x}_{k+1}}  \numleqslant{ii}   \beta_{k+1}  \norml{ \hat{\x}_{k+1} -  \hat{\x}_{k}},   \numberthis  \label{li_transfer}  
 \end{align*}
 where (i) follows from the definition of $\x_{k+1}$ (see \cref{li_0}) and (ii) follows from \cref{li_1}.
  
 We then present the following lemma that bounds $\norml{ \hat{\x}_{k+1} - \hat{\x}_k}$, where $\{\hat{\x_k}\}_{k \geqslant 0}$ is generated by CRm. 
 \begin{lemma} \label{momentum_6}
 	Let \Cref{Assumption_1} hold. Then, the sequence $\{ \hat{\x}_k \}$ generated by CRm  satisfies
 	\begin{align}
 	\norml{ \hat{\x}_{k+1} - \hat{\x}_k}   \leqslant  c_5,
 	\end{align} 
 	where $c_5 \triangleq  \frac{1}{1 - \rho}\left( \frac{ f(\x_0) - f^\star }{\gamma}\right)^{1/3}$.
 \end{lemma} 
 \begin{proof}
 	See \Cref{pf_lemmas_4}.
 \end{proof}
 Next, we prove the main theorem for CRm. Note that
 	\begin{align*}
 	\norml{\nabla f(\x_{k+1})} &\leqslant \norml{\nabla f(\hat{\x}_{k+1})} +  \norml{\nabla f(\hat{\x}_{k+1}) - \nabla f(\x_{k+1})} \\
 	&\numleqslant{i}  \norml{\nabla f(\hat{\x}_{k+1})} + L_1 \norml{\x_{k+1} - \hat{\x}_{k+1}} \\
 	&\numleqslant{ii}  \norml{\nabla f(\hat{\x}_{k+1})} + L_1 \beta_{k+1} \norml{ \hat{\x}_{k+1} - \hat{\x}_{k}} \\
 	&\numleqslant{iii} \norml{\nabla f(\hat{\x}_{k+1})}\left(1  + L_1    \norml{ \hat{\x}_{k+1} - \hat{\x}_{k}} \right) \numberthis \label{li_8_0} \\
 	&\numleqslant{iv}  \frac{L_2 + M}{2} \norml{\hat{\x}_{k+1} - \x_k}^2 \left(1 + \frac{ L_1 }{1 - \rho}\left( \frac{ f(\x_0) - f^\star }{\gamma}\right)^{1/3}  \right) , \numberthis \label{li_8}
 	\end{align*}
 	where (i) follows from  the Lipschitz gradient assumption, (ii) follows from \cref{li_transfer}, (iii) follows from \cref{li_2}, which implies that $\beta_{k+1} \leqslant  \norml{\nabla f(\hat{\x}_{k+1})}$ and (iv) follows from \Cref{momentum_6} and \cref{momentum_7}.
 	
 	Next, we bound the minimum eigenvalue of the Hessian. Observe that
 	\begin{align*}
 	\lambda_{\min}\left({\nabla^2 f(\x_{k+1}) }\right) &\numgeq{i}  \lambda_{\min}\left({\nabla^2 f(\hat{\x}_{k+1}) }\right) - \norml{\nabla^2 f(\x_{k+1}) - \nabla^2 f(\hat{\x}_{k+1})} \\
 	&\numgeq{ii}	\lambda_{\min}\left({\nabla^2 f(\hat{\x}_{k+1}) }\right) - L_2 \norml{\x_{k+1} - \hat{\x}_{k+1}} \\
 	&\numgeq{iii} \lambda_{\min}\left({\nabla^2 f(\hat{\x}_{k+1}) }\right) - L_2 \beta_{k+1} \norml{\hat{\x}_{k+1} - \hat{\x}_{k}}  \\
 	&\numgeq{iv} \lambda_{\min}\left({\nabla^2 f(\hat{\x}_{k+1}) }\right) - L_2 \norml{\hat{\x}_{k+1} - \x_k} \norml{\hat{\x}_{k+1} - \hat{\x}_{k}} \numberthis \label{li_90} \\
 	&\numgeq{v}   - \frac{M+2L_2}{2}  \norml{\hat{\x}_{k+1} - \x_k} - \norml{\hat{\x}_{k+1} - \x_k}  \frac{ L_2 }{1 - \rho}\left( \frac{ f(\x_0) - f^\star }{\gamma}\right)^{1/3}  \\
 	&= -\norml{\hat{\x}_{k+1} - \x_k} \left( \frac{M+2L_2}{2} +   \frac{ L_2 }{1 - \rho}\left( \frac{ f(\x_0) - f^\star }{\gamma}\right)^{1/3} \right) , \numberthis \label{li_9}
 	\end{align*}  
 	where (i) follows form Weyl's inequality, (ii) follows from the fact that $\nabla^2 f(\cdot)$ is $L_2$ Lipschitz, (iii) follows from  \cref{li_transfer}, (iv)  follows from \cref{li_2}, which implies that $\beta_{k+1} \leqslant \norml{\hat{\x}_{k+1} - \x_k}$  and (v) follows from \Cref{momentum_6} and \cref{momentum_14}.
 	
 	Then, by \cref{momentum_8} , there exists a point $k_0 \in \{0, \cdots, k-1 \}$ such that
 	\begin{align}
 	\norml{\hat{\x}_{k_0+1}- \x_{k_0}}  \leqslant  \frac{1}{k^{1/3}} \left(  \frac{ f(\x_0) - f^\star }{  \gamma}\right)^{1/3} \label{momentum_15}.
 	\end{align}
 	Plugging \cref{momentum_15} into \cref{li_8,li_9} with $k = k_0$, we further obtain that
 	\begin{align*}
 	\norml{\nabla f(\x_{k_0+1})}  &\leqslant   \frac{1}{k^{2/3}} \frac{L_2 + M}{2} \left(  \frac{ f(\x_0) - f^\star }{  \gamma}\right)^{2/3} \left(1    + \frac{ L_1 }{1 - \rho}\left( \frac{ f(\x_0) - f^\star }{\gamma}\right)^{1/3}   \right)\numberthis \label{li_10} \\
 	\lambda_{\min}\left({\nabla^2 f(\x_{k_0+1}) }\right)  
 	&\geqslant -  \frac{1}{k^{1/3}}  \left(  \frac{ f(\x_0) - f^\star }{  \gamma}\right)^{1/3}\left( \frac{M+2L_2}{2} +   \frac{ L_2 }{1 - \rho}\left( \frac{ f(\x_0) - f^\star }{\gamma}\right)^{1/3} \right). \numberthis \label{li_11}
 	\end{align*}
 	
 	Thus, in order to  guarantee $\norml{\nabla f(\x_{k_0+1})} \leqslant \epsilon$ in \cref{li_10} and $\lambda_{\min}\left({\nabla^2 f(\x_{k_0+1}) }\right) \geqslant - \sqrt{\epsilon}$ in \cref{li_11}, we require 
 	\begin{align} 
 	k \geqslant     \frac{1}{\epsilon^{3/2}} \bigg(\frac{L_2 + M}{2}\bigg)^{3/2}\left(\frac{ f(\x_0) - f^\star }{\gamma} \right)\left(1    + \frac{ L_1 }{1 - \rho}\left( \frac{ f(\x_0) - f^\star }{\gamma}\right)^{1/3} \right)^{3/2} . \label{li_20}
 	\end{align} 
 	\begin{align}
 	k \geqslant \frac{1}{\epsilon^{3/2}}\left(  \frac{ f(\x_0) - f^\star }{  \gamma}\right)\left( \frac{M+2L_2}{2} +   \frac{ L_2 }{1 - \rho}\left( \frac{ f(\x_0) - f^\star }{\gamma}\right)^{1/3} \right)^3. \label{li_21}
 	\end{align}
 	
 	Combining \cref{li_20,li_21}, we obtain that CRm  must pass an $\epsilon$-approximate second-order stationary point if 
 	\begin{align*} 
 	k &\geqslant \frac{1}{\epsilon^{3/2}} \max  \{c_3,  c_4  \},
 	\end{align*}  
 	where \begin{align*}
 	c_3 &\triangleq \bigg(\frac{L_2 + M}{2}\bigg)^{3/2}\left(\frac{ f(\x_0) - f^\star }{\gamma} \right)\left(1    + \frac{ L_1 }{1 - \rho}\left( \frac{ f(\x_0) - f^\star }{\gamma}\right)^{1/3} \right)^{3/2}   \\
 	c_4 &\triangleq \left(  \frac{ f(\x_0) - f^\star }{  \gamma}\right)\left( \frac{M+2L_2}{2} +   \frac{ L_2 }{1 - \rho}\left( \frac{ f(\x_0) - f^\star }{\gamma}\right)^{1/3} \right)^3 .
 	\end{align*} 

\section{Local Quadratic Convergence: Proof of \Cref{Local_adaptive_momentum_Thm}} \label{proof_s_3}
We first present the following lemma that characterizes the properties of the sequence $\{ \x_{k}\}_{k \geqslant 0}$ generated by  CRm.
\begin{lemma} \label{li_lemma_local}
	Let \Cref{Assumption_1} hold and assume that $\mathcal{L}(f(\x_k))$ is bounded for some $k \geqslant 0$. Then, the sequence $\{ \x_{k}\}_{k \geqslant 0}$ generated by CRm  and its set of accumulation points $\mathbf{\bar{\mathcal{X}}}$ satisfy
	\begin{enumerate}
		\item[(i)] $v \triangleq \lim\limits_{k \rightarrow \infty} f(\x_k)$ exists.
		\item[(ii)]  $\lim\limits_{k \rightarrow \infty} \norml{\x_{k+1} - \x_{k}} = 0$.
		\item[(iii)] The sequence $\{ \x_{k}\}_{k \geqslant 0}$ is bounded.
		\item[(iv)] The set $\mathbf{\bar{\mathcal{X}}}$ is bounded and non-empty. Moreover,  every$\bar{\x} \in \mathbf{\bar{\mathcal{X}}}$, satisfies 
		$$f(\bar{\x})=v,\quad\nabla f(\bar{\x})=0,\quad\nabla^2 f(\bar{\x})\succcurlyeq 0.
		$$
	\end{enumerate}
\end{lemma}  
\begin{proof}
	See \Cref{pf_lemmas_6}.
\end{proof}	

We next prove the main theorem for CRm.
	Denote $\mathbf{\bar{\x}}_k \in \argmin_{\mathbf{z \in \mathcal{X}}} \norml{\x_k - \mathbf{z} }^2$ as the projection of $\x_k$ onto $\mathcal{X}$. Since $\{ \x_{k}\}_{k \geqslant 0}$ is bounded (\Cref{li_lemma_local}, (iii)) and $\mathbf{\bar{\mathcal{X}}}$ is non-empty and bounded (\Cref{li_lemma_local}, (iv)),  we conclude that $\lim_{k \rightarrow \infty} \text{dist}(\x_k, \mathbf{\bar{\mathcal{X}}}) = 0$. By (iv) of  \Cref{li_lemma_local} and the definition of $\mathcal{X}$, we have $\mathbf{\bar{\mathcal{X}}} \subseteq \mathcal{X}$. Thus, we obtain that
	\begin{align} 
	\lim_{k \rightarrow \infty} \norml{\x_k - \mathbf{\bar{\x}}_k} =\lim_{k \rightarrow \infty} \text{dist}(\x_k, \mathcal{X}) \leqslant \lim_{k \rightarrow \infty} \text{dist}(\x_k, \mathbf{\bar{\mathcal{X}}})  =  0,
	\end{align}
	which implies that
	\begin{align}
	\lim_{k \rightarrow \infty} \text{dist}(\x_k, \mathcal{X}) = 0.  \label{li_local_4}
	\end{align}
	Therefore, for any $r>0$, there exists a $k_1 \geqslant 0$ such that  $\text{dist}(\x_k, \mathcal{X}) \leqslant r$ for all $ k \geqslant k_1$. Combining this with \Cref{ass_eb}, we obtain that
	\begin{align}
	\text{dist}(\x_k, \mathcal{X}) \leqslant \kappa\|\nabla f(\x_k)\|, \quad\forall k \geqslant k_1.
	\end{align} 
	Hence, for all $k \geqslant k_1$, we obtain that
	\begin{align*}
	\text{dist}(\x_{k+1}, \mathcal{X}) &\leqslant \kappa\|\nabla f(\x_{k+1})\| \\
	&\leqslant \kappa\|\nabla f(\x_{k+1}) -   \nabla f(\hat{\x}_{k+1}) \| + \kappa\| \nabla f(\hat{\x}_{k+1}) \| \\
	&\numleqslant{i} \kappa L_1 \norml{\x_{k+1} - \hat{\x}_{k+1}} + \kappa\| \nabla f(\hat{\x}_{k+1}) \| \\
	&\numleqslant{ii} \kappa L_1   \beta_{k+1} \norml{\hat{\x}_{k+1} - \hat{\x}_k} + \kappa\| \nabla f(\hat{\x}_{k+1}) \| \\
	&\numleqslant{iii}\kappa\| \nabla f(\hat{\x}_{k+1}) \|\left(L_1   \norml{\hat{\x}_{k+1} - \hat{\x}_k} + 1\right) \\
	&\numleqslant{iv} \kappa\| \nabla f(\hat{\x}_{k+1}) \| \left(    L_1 c_5+ 1\right) \\
	&\numleqslant{v}   \kappa \left(    \frac{L_2 + M}{2} \right)  \left( L_1 c_5+ 1\right)\norml{ \hat{\x}_{k+1} - \x_k}^2  \numberthis \label{li_local_1}
	\end{align*}
	where (i) follows from the Lipschitz gradient property, (ii) follows from  \cref{li_transfer}, (iii) follows from \cref{li_1}, which implies that $\beta_{k+1} \leqslant \| \nabla f(\hat{\x}_{k+1}) \|$,  (iv) follows from \Cref{momentum_6} and (v) follows from \cref{momentum_7}.	Combining \cref{li_local_1} with \Cref{ada_lemma_1}, we obtain that, for all $k \geqslant k_1$, 
	\begin{align*}
	\text{dist}(\x_{k+1}, \mathcal{X}) &\leqslant   \kappa \left(    \frac{L_2 + M}{2} \right)  \left(  L_1c_5+ 1\right) c_1^2   \cdot \text{dist}(\x_k, \mathcal{X})^2 \\
	&= c_6 \cdot \text{dist}(\x_k, \mathcal{X})^2 ,  \numberthis \label{li_local_7}
	\end{align*}
	where    $c_6 \triangleq \kappa \left(    \frac{L_2 + M}{2} \right)  \left(  L_1c_5+ 1\right) c_1^2    $.
	
	Next, we prove that $\{ \x_{k}\}_{k \geqslant 0}$ is Cauchy, and hence is a convergent sequence. For any $\epsilon >0$, by \cref{li_local_4}, there exists $k_2 \geqslant 0$ such that 
	\begin{align} \label{li_local_6}
	\text{dist}(\x_k, \mathcal{X}) \leqslant \min \left(  \frac{1}{2c_6},  \frac{\epsilon }{2c_1(c_5 + 1)}   \right), \quad \forall k \geqslant k_2.
	\end{align} 	Therefore, for any $k \geqslant \max(k_1,k_2)$ and  any $j \geqslant 0$, we have
	\begin{align*}
	\norml{\x_{k+j} - \x_{k}} &\leqslant \sum_{i = k}^{k+j-1}\norml{\x_{i+1} - \x_{i}} \leqslant \sum_{i = k}^{k+j-1} \left(\norml{\x_{i+1} - \hat{\x}_{i+1} } + \norml{\hat{\x}_{i+1} - \x_{i}} 
	\right) \\
	&\numleqslant{i} \sum_{i = k}^{k+j-1} \left( \beta_{i+1} \norml{\hat{\x}_{i+1} - \hat{\x}_i}  + \norml{\hat{\x}_{i+1} - \x_{i}} \right)  \\
	&\numleqslant{ii}\sum_{i = k}^{k+j-1} \norml{\hat{\x}_{i+1} - \x_{i}} \left(   \norml{\hat{\x}_{i+1} - \hat{\x}_i}  + 1   \right) \\
	&\numleqslant{iii} \sum_{i = k}^{k+j-1} \norml{\hat{\x}_{i+1} - \x_{i}}  (c_5 + 1 )  \numleqslant{iv}  \sum_{i = k}^{k+j-1}    c_1 \cdot \text{dist}(\x_i, \mathcal{X}) (c_5 + 1 )    \\
	&= c_1  (c_5 + 1 )  \sum_{i = k}^{k+j-1}       \text{dist}(\x_i, \mathcal{X}) \\
	&\numleqslant{v}  c_1  (c_5 + 1 )   \text{dist}(\x_k, \mathcal{X})  \sum_{i = 0}^{\infty} \frac{1}{2^i} \\
	&= 2c_1 \left( c_5 + 1   \right)  \cdot \text{dist}(\x_k, \mathcal{X}) \numberthis \label{li_local_8}\\
	&\numleqslant{iv} \epsilon,
	\end{align*}
	where (i) follows from \cref{li_transfer}, (ii) follows from \cref{li_1}, which implies that $\beta_{i+1} \leqslant \norml{\hat{\x}_{i+1} -  \x_i} $, (iii) follows from \Cref{momentum_6}, (iv) follows from \Cref{ada_lemma_1}, (v) follows from \cref{li_local_7} and   \cref{li_local_6}, which implies that $	\text{dist}(\x_{k+1}, \mathcal{X}) 
	\leqslant    \text{dist}(\x_k, \mathcal{X}) /2 $ and (iv) follows from \cref{li_local_6}. Then, we conclude  that  $\{ \x_{k}\}_{k \geqslant 0}$ is a Cauchy sequence, and thus converges.
	
	Next, we study the convergence rate of $\{ \x_{k}\}_{k \geqslant 0}$. Let $\x^\star \triangleq \lim_{k \rightarrow \infty} \x_k$. By (iv) of \Cref{li_lemma_local}, we have $\x^\star \in \mathcal{X}$. Then, for all $k \geqslant \max\{k_1,k_2\}$, we obtain that
	\begin{align*}
	\norml{\x^\star - \x_{k+1}} &= \lim_{j \rightarrow \infty} \norml{\x_{k+1+j} - \x_{k+1}} \numleqslant{i}  2c_1 \left( c_5 + 1   \right)\cdot \text{dist}(\x_{k+1}, \mathcal{X}) \\
	&\numleqslant{ii} 2c_1 \left( c_5 + 1   \right)c_6 \cdot \text{dist}(\x_{k}, \mathcal{X})^2   \numleqslant{iii} 2c_1 \left( c_5 + 1   \right)c_6  \norml{\x^\star - \x_{k}}^2  \numberthis \label{li_local_9}
	\end{align*}
	where (i) follows from \cref{li_local_8}, (ii) follows from \cref{li_local_7}, and (iii) follows from the fact that $\text{dist}(\x_{k}, \mathcal{X}) \leqslant \norml{\x^\star - \x_{k}}$.
	
	Note that \cref{li_local_9} implies that 
	\begin{align} \label{li_Q_quadratic_convergence}
	\frac{\norml{\x^\star - \x_{k+1}}}{\norml{\x^\star - \x_{k}}^2} \leqslant 2c_1 \left( c_5 + 1   \right)c_6 , \quad \forall k \geqslant \max\{k_1,k_2 \}.
	\end{align}
	Hence, $\{ \x_{k}\}_{k \geqslant 0}$ converges at least Q-quadratically to $\x^\star$.
	In particular, the Q-quadratic convergence region of CRm  is given by
	\begin{align}
	\norml{\x_{k} - \x^\star} \leqslant \frac{1}{2c_1c_6 \left( c_5 + 1   \right)}.
	\end{align}

\section{Inexact CRm Convergence: Proof of \Cref{tyep1_im_thm}} \label{proof_s_4}
 We first present the following lemma, which bounds the term $\norml{ \hat{\x}_{k+1} - \hat{\x}_k}$, where sequence $\{\x_k\}_{k \geqslant 0}$ is generated by the inexact variant of CRm.
 \begin{lemma}  \label{II_i_summable_lemma}
 	Let \Cref{Assumption_1} and \Cref{assump_2} hold. Set $M > 2L_2/3 + 2, \beta_{k} \leqslant \rho, \epsilon_1 \leqslant 1$. Then the sequence $\{ \hat{\x}_k \}_{k \geqslant 0}$ generated by the inexact variant of CRm  satisfies
 	\begin{align}
 	\norml{ \hat{\x}_{k+1} - \hat{\x}_k}   \leqslant c_8
 	\end{align}
 	where $c_8 \triangleq  \frac{1}{1 - \rho}   \left(\left( \frac{3M-2L_2 - 6}{12}  \right)^{-1/3}(f(\x_0) - f^{\star})^{1/3} + 1 \right)$.
 \end{lemma}
 \begin{proof}
 	See \Cref{pf_lemmas_5}.
 \end{proof}
 
We next prove the main theorem for CRm.
 	The proof of this theorem is similar to that of  the exact case, and hence we only highlight the main difference for simplicity. Define $\epsilon_1 = \theta  \sqrt{\epsilon}$ where 
 	$$\theta  \triangleq  \min \left\{ \sqrt{\frac{2  }{\left(1  +  L_1 c_8 \right) (L_2 + M +2)}}, \frac{2}{M + 2L_2 + 2 + L_2c_8}\right\}.$$
 	
 	Since $\frac{2}{M + 2L_2 + 2 + L_2c_8} \leqslant 1$ and $\epsilon \leqslant 1$, we obtain that $\epsilon_1 \leqslant 1$. 
 	
 	Note that \Cref{im_important_cubic_lemma} implies that if the total number of iterations $k >  \left( \frac{12}{3M-2L_2 -6} \right) \frac{  f(\x_0) - f^{\star}} { \theta^3 \epsilon^{3/2} }$, then there exists a $k_0 \in \{0, \cdots, k  \}$ such that
 	\begin{align}
 	\norml{\hat{\x}_{k_0+1}- \x_{k_0}}  \leqslant  \epsilon_1 \label{II_im_4}.
 	\end{align}

 	Following the reasoning similar to that for proving \cref{li_8_0}, we obtain that
 	\begin{align*}
 	\norml{\nabla f(\x_{k_0+1})} &\leqslant  \norml{\nabla f(\hat{\x}_{k_0+1})}\left(1  + L_1    \norml{ \hat{\x}_{k_0+1} - \hat{\x}_{k_0}} \right) \\
 	&\numleqslant{i}  \norml{\nabla f(\hat{\x}_{k_0+1})}\left(1  +  L_1 c_8 \right) \\ &\numleqslant{ii} \left(1  +  L_1 c_8 \right)  \left( \frac{L_2 + M}{2}\norml{\hat{ \x}_{k_0+1} - \x_{k_0}}^2 + \epsilon_1 \norml{\hat{ \x}_{k_0+1} - \x_{k_0}}\right) \\
 	&\numleqslant{iii} \left(1  +  L_1 c_8 \right)  \left( \frac{L_2 + M +2}{2}  \right) \epsilon_1^2  \numleqslant{iv} \epsilon,	\numberthis \label{II_I_8} 
 	\end{align*} 
 	where (i) follows from \Cref{II_i_summable_lemma}, (ii) follows from \cref{i_momentum_7}, (iii) follows from \cref{II_im_4}, and (iv) follows from the definition of $\epsilon_1$.
 	
 	Then, following the reasoning similar to that for proving \cref{li_90}, we further obtain that
 	\begin{align*}
 	\lambda_{\min}\left({\nabla^2 f(\x_{k_0+1}) }\right) &\geqslant   \lambda_{\min}\left({\nabla^2 f(\hat{\x}_{k_0+1}) }\right) - L_2 \norml{\hat{\x}_{k_0+1} - \x_{k_0}} \norml{\hat{\x}_{k_0+1} - \hat{\x}_{k_0}}  \\  
 	&\numgeq{i} \lambda_{\min}\left({\nabla^2 f(\hat{\x}_{k_0+1}) }\right) - L_2c_8 \norml{\hat{\x}_{k_0+1} - \x_{k_0}}  \\  
 	&\numgeq{ii}  -\frac{M + 2L_2}{2} \norml{\hat{ \x}_{k_0+1} - \x_{k_0}} - \epsilon_1 - L_2c_8 \norml{\hat{\x}_{k_0+1} - \x_{k_0}} \\
 	&\numgeq{iii} -\frac{M + 2L_2 + 2 + L_2c_8}{2} \epsilon_1  \numgeq{iv} -\sqrt{\epsilon}, \numberthis \label{II_I_9}
 	\end{align*}
 	where (i) follows from \Cref{II_i_summable_lemma}, (ii) follows from \cref{momentum_14}, (iii) follows from \cref{II_im_4}, and (iv) follows from the definition of $\epsilon_1$. Combining \cref{II_I_8,II_I_9}, we obtain the main statement of \Cref{tyep1_im_thm}.
 	
\section{Subsampling Technique}
To prove   \Cref{total_compelxity}, we first establish a useful Proposition (\Cref{subsample_per_iteration}) to characterize the per iteration complexity of Hessian in \Cref{inexact_cov_1}, and then establish the overall convergence guarantee in \Cref{inexact_cov_2}.
\subsection{Per iteration Complexity} \label{inexact_cov_1}
In order to satisfy the inexact criterion $\epsilon_1$ in  \Cref{assump_2}, the mini-batch size should be large enough to guarantee statistical concentration with high probability \citep{Xu2017,JinChi2017cubic,kohler2017,Wang2018,Zhou2018}. Such an approach is referred to as the subsampling technique, and has been used in \cite{kohler2017} to implement the inexact CR. Here, we apply such an approach to CRm, and the following theorem characterizes the sample complexity to guarantee the inexact criterion for each iteration.
\begin{Proposition}[Per iteration Hessian sample complexity] \label{subsample_per_iteration} Assuming that \Cref{Assumption_1} holds for each $f_i(\cdot)$, then sub-sampled mini-batch of Hessians $\mathbf{H}_k, k = 0, 1, \ldots$   satisfies \Cref{assump_2} with probability at least $1 - \zeta$ provided that  
	\begin{align}\label{hessian_sample}
	|S_1| \geqslant \left( \frac{8L_1^2}{\epsilon_1^2} + \frac{4L_1 }{3 \epsilon_1 }\right) \log \left(\frac{4d}{\zeta}\right).
	\end{align} 
\end{Proposition}

The idea of the proof is to apply the following matrix Bernstein inequality \citep{Tropp2012}  to characterize the sample complexity in order to satisfy the inexactness condition in \Cref{assump_2} with the probability at least $1 - \zeta$. 
\newcommand{\Ebb}{\mathbb{E}}
\newcommand{\X}{\mathbf{X}}
\begin{lemma}[\cite{Tropp2012}, Theorem 1.6.2] \label{bernstern_0}
	Consider a finite sequence $\{\X_k\}$ of independent, random matrices with dimensions $d_1 \times d_2$. Assume that each random matrix satisfies 
	\begin{align*}
	\Ebb \X_k = \mathbf{0} \quad \text{and} \quad \norml{\X_k} \leqslant R  \quad \text{almost surely}.
	\end{align*}
	Define  
	\begin{align}
	\sigma^2 \triangleq \max \left(	\normlarge{\sum\nolimits_{k} \Ebb(\X_k\X_k^*) }, \normlarge{\sum\nolimits_{k} \Ebb(\X_k^*\X_k) } \right).
	\end{align}
	Then, for all $\epsilon \geqslant 0$,
	\begin{align*}
	P   \bigg(\normlarge{   \sum\nolimits_k\X_{k}  }  \geqslant   \epsilon \bigg)  \leqslant     2(d_1 + d_2) \exp \bigg( -  \frac{  \epsilon^2/2 }{  \sigma^2  + R\epsilon/3}\bigg).
	\end{align*}  
\end{lemma}
With this lemma in hand, we are ready to prove our main result.
\begin{proof}[Proof of \Cref{subsample_per_iteration}]
	In order to apply \Cref{bernstern_0}, we first define 
	\begin{align*}
		\X_i = \frac{1}{|S_1|} \left( \nabla^2 f_i(\x_k) - \nabla^2 f(\x_k) \right).
	\end{align*}
	Then, we obtain that
	\begin{align}
		\Ebb \X_i= 0  \label{Bersterin_0}
	\end{align}
	and
	\begin{align}
		\norml{\X_i} \leqslant \frac{\norml{\nabla^2 f_i(\x_k)} + \norml{\nabla^2 f(\x_k)}}{|S_1|}  = \frac{2L_1}{|S_1|} \triangleq R. \label{Bersterin_1}
	\end{align}
	where (i) follows from item 3 of  \Cref{Assumption_1} that $\nabla f_i(\cdot)$ is $L1$-Lipschitz which implies that $\norml{\nabla^2 f_i(\cdot)} \leqslant L_1$ and $ \norml{\nabla^2 f(\cdot)} \leqslant L_1$.
	
	Moreover, we have that
	\begin{align*}
			\sigma^2 &= \max \left(	 \normlarge{\sum_{i \in S_1} \Ebb(\X_i\X_i^*) }  , \normlarge{\sum_{i \in S_1} \Ebb(\X_i^*\X_i) } \right) \\
			&\numequ{i} \normlarge{\sum_{i \in S_1} \Ebb(\X_i^2) }  \leqslant \sum_{i \in S_1}\normlarge{ \Ebb(\X_i^2) }  \numleqslant{ii} \sum_{i \in S_1} \Ebb\normlarge{  \X_i^2 }   \leqslant \sum_{i \in S_1} \Ebb\normlarge{  \X_i }^2   \numleqslant{iii} \frac{4L_1^2}{|S_1|}, \numberthis \label{Bersterin_2}
	\end{align*}
	where (i) follows from the fact that $\X_i$ is real and symmetric, (ii) follows from Jasen's inequality, and (iii) follows from \cref{Bersterin_1}. 
	
	Plugging \cref{Bersterin_1,Bersterin_2,Bersterin_0}  into \Cref{bernstern_0}, we obtain 
	\begin{align}
	  P   \bigg(\normlarge{   \sum_{i \in S_1} \X_{k}  }  \geqslant   \epsilon_1 \bigg)  \leqslant     4d \exp \left( -  \frac{  \epsilon_1^2/2 }{  \frac{4L_1^2}{|S_1|}  + \frac{2L_1\epsilon_1}{3|S_1|}}  \right).
	\end{align}
	Thus, in order to satisfies $\normlarge{   \sum_{i \in S_1} \X_{k}  }  \geqslant   \epsilon_1 $ with probability at least $1 - \zeta$, it is sufficient to require  
	\begin{align}
		 4d \exp \left( -  \frac{  \epsilon_1^2/2 }{  \frac{4L_1^2}{|S_1|}  + \frac{2L_1\epsilon_1}{3|S_1|}}  \right) \leqslant \zeta,
	\end{align} 
	which gives that
	\begin{align}
		 |S_1| \geqslant \left( \frac{8L_1^2}{\epsilon_1^2} + \frac{4L_1 }{3 \epsilon_1 }\right) \log \left(\frac{4d}{\zeta}\right).
	\end{align}
\end{proof}
 
\subsection{Overall Complexity: Proof of \Cref{total_compelxity}}  \label{inexact_cov_2}
\begin{proof}
	  We first note that \Cref{tyep1_im_thm} shows that let $\epsilon_1 = \theta \sqrt{\epsilon}$, then the sequence $\{\x_{k} \}_{k\geqslant 0}$ generated by the inexact CRm  contains an $\epsilon$-second-order stationary point if the total number $k$ of iterations  satisfies
	 \begin{align}
	 k =  \frac{C}{\epsilon^{3/2}}. \label{total_iteration}
	 \end{align}
	  
	  Next, according to \Cref{subsample_per_iteration}, \Cref{assump_2} is satisfies with probability at least  $1 - \zeta$ for Hessian . Thus, according to the union bound,  for $k$ iterations, the probability of failure satisfaction of \Cref{assump_2} is at most $  k \zeta$. To obtain  \Cref{assump_2} holds for the total $k$ iteration with probability least $1 - \delta$, we require
	 \begin{align*}
	 1 -  k\zeta \geqslant 1 - \delta,
	 \end{align*}
	 which yields
	 \begin{align*}
	 \zeta \leqslant  \frac{\delta}{ k}.
	 \end{align*} 
	 Thus, with probability $1 - \delta$, the algorithms successfully outputs an $\epsilon$ approximated second-order stationary point if we set $\zeta =  {\delta}/ { k}$.
	 Therefore, according to \Cref{subsample_per_iteration}, \Cref{assump_2} with $\epsilon_1 = \theta \sqrt{\epsilon}$  holds with probability at least $1 - \zeta$ given that
	 \begin{align}
	 |S_1| = \left( \frac{8L_1^2}{\theta^2 \epsilon} + \frac{4L_1 }{3 \theta \sqrt{\epsilon} }\right) \log \left(\frac{4dk}{\delta}\right),
	 \end{align}
	  and the total Hessian sample complexity is bounded by
	 \begin{align}
	 S &= k \times |S_1| = k \times \left( \frac{8L_1^2}{\theta^2 \epsilon} + \frac{4L_1 }{3 \theta \sqrt{\epsilon} }\right) \log \left(\frac{4dk}{\delta}\right) \numleqslant{i} C \left( \frac{8L_1^2}{\theta^2 \epsilon^{5/2}} + \frac{4L_1 }{3 \theta {\epsilon}^2 }\right) \log \left(\frac{4d}{\epsilon\delta}\right).
	 \end{align} 
	  where (i) follows from \cref{total_iteration}.

\end{proof}

\section{Proof of Technical Lemmas}  \label{proof_of_lemmas}
In this section, we provide the proofs of the technical lemmas.
\subsection{Useful Inequality}
\begin{lemma} \label{inequality}
	For $x, \Lambda \in \mathbb{R}$, $0 < x \leqslant 1$, and $0 < \Lambda < 1$, the following inequality holds
	\begin{align}
		 x (1-\Lambda) \log \left(\frac{1}{1 - \Lambda}\right) \leqslant 1 - (1 - \Lambda)^x.
	\end{align}
\end{lemma}
\begin{proof}
	Let 
	\begin{align*}
		 f(x) =  1 - (1 - \Lambda)^x -  x (1-\Lambda) \log \left(\frac{1}{1 - \Lambda}\right),
	\end{align*}
	it is sufficient to prove  $f(x) \geqslant 0 $ for $0 < x \leqslant 1$, and $0 < \Lambda < 1$. We first note that
	\begin{align}
		 \nabla f(x) = - (1 - \Lambda)^x \log (1 - \Lambda) - (1-\Lambda) \log \left(\frac{1}{1 - \Lambda} \right),
	\end{align}
	which is decreasing with respect to $x$. Thus, we have, for $ 0 < x \leqslant 1$
	\begin{align*}
		 \nabla f(x) \geqslant \nabla f(1) = 0,
	\end{align*}
	which implies $f(x)$ increasing within $0 < x \leqslant 1$. Thus,
	\begin{align}
		 f(x) \geqslant f(0) = 0.
	\end{align}
	Therefore, we complete our proof.
\end{proof}
\subsection{Proof of \Cref{after_cubic}}
    \begin{proof} 
    	We first prove \cref{momentum_12}. Define $\s_{k} \triangleq \hat{ \x}_{k+1} - \x_k $. Then, we obtain that
    	\begin{align*}
    	  f(\hat{\x}_{k+1}) &\numleqslant{i}   f(\x_k) + \nabla f(\x)^T\s_k + \frac{1}{2} \s_k^T  \nabla^2 f(\x )\s_k   + \frac{L_2}{6}  \norml{\s_k}^3 \\
    	  &\numleqslant{ii} f(\x_k)  -  \frac{M}{12}  \norml{\s_k}^3 + \frac{L_2 - M}{6}  \norml{\s_k}^3 \\
    	  &=f(\x_k) - \frac{3M- 2L_2}{12} \norml{ \hat{ \x}_{k+1} - \x_k}^3,
    	\end{align*}
    	where (i) follows from \cref{Hessian_square_bound}, and (ii) follows from Lemma 4 in \cite{Nesterov2006} and the definition of $\hat{\x}_{k+1}$.
    	Then, we further obtain that
    	\begin{align}
    	f(\hat{\x}_{k+1}) - f(\x_k) \leqslant - \gamma \norml{\hat{\x}_{k+1}- \x_k}^3, \label{momentum_3}
    	\end{align}
    	which gives \cref{momentum_12}. 
    	
    	Next, we prove \cref{momentum_4}. Note that \cref{momentum_12} implies that, for all $i \geqslant 0$,
    	\begin{align*}
    	\norml{\hat{\x}_{i+1}- \x_i}^3 &\numleqslant{i} \frac{ f(\x_i) - f(\hat{\x}_{i+1}) }{\gamma} \numleqslant{ii} \frac{ f(\x_i) - f(\x_{i+1}) }{\gamma}, \numberthis \label{momentum_2}
    	\end{align*}	 
    	where (i) follows from \cref{momentum_3}, and (ii) follows from the definition of $\x_{i+1}$. Summing \cref{momentum_2} over $i$ from $0$ to $k$, we obtain that
    	\begin{align*}
    	\sum_{i=0}^{k} \norml{\hat{\x}_{i+1}- \x_i}^3 \leqslant \frac{ f(\x_0) - f(\x_{k+1}) }{\gamma} \leqslant \frac{ f(\x_0) - f^{\star}}{\gamma}, 
    	\end{align*}
    	which gives \cref{momentum_4}.
    	
    	To prove \cref{momentum_7,momentum_14}, note that 
    	\begin{align*}
    	\hat{ \x}_{k+1} = \argmin_{\s \triangleq \x - \x_k} \nabla f(\x_k)^T\s     + \frac{1}{2} \s^T \nabla^2 f(\x_{k}) \s + \frac{M}{6} \norml{\s}^3.
    	\end{align*}
    	Then, Lemma 5 in \cite{Nesterov2006} directly implies \cref{momentum_7,momentum_14}.
    	
    	Next, we prove \cref{momentum_8}. Note that
    	\begin{align}
    		  \min_{0 \leqslant i \leqslant k}  \norml{\hat{\x}_{i+1}- \x_i}^3  \leqslant \frac{1}{k} \sum_{i=0}^{k} \norml{\hat{\x}_{i+1}- \x_i}^3 \numleqslant{i}  \frac{1}{k} \frac{ f(\x_0) - f^{\star}}{\gamma}, \label{m0}
    	\end{align}
    	where (i) follows from \cref{momentum_4}. Then, \cref{momentum_8} follows by taking the cubic root on both sides of \cref{m0}.
    \end{proof}
\subsection{Proof of \Cref{im_important_cubic_lemma}}
	\begin{proof}
		We first prove \cref{im0}. Note that
		\begin{align*}
			f&(\hat{\x}_{k+1}) \\
			&\numleqslant{i}   f(\x_k) + \nabla f(\x_k)^T\s_k + \frac{1}{2} \s_k^T  \nabla^2 f(\x_k )\s_k   + \frac{L_2}{6}  \norml{\s_k}^3 \\
			&= f(\x_k) + \nabla f(\x_k)^T\s_k + \frac{1}{2} \s_k^T  \mathbf{H}_k \s_k  + \frac{M}{6}  \norml{\s_k}^3 +  \frac{L_2 - M}{6}  \norml{\s_k}^3 + \frac{1}{2} \s_k^T  (\nabla^2  f(\x_k ) - \mathbf{H}_k)\s_k \\
			&\numleqslant{ii}  f(\x_k)  - \frac{M}{12} \norml{\s_k}^3  + \frac{L_2 - M}{6}  \norml{\s_k}^3 + \frac{1}{2} \s_k^T (\nabla^2  f(\x_k ) - \mathbf{H}_k) \s_k \\
			&\numleqslant{iii}  f(\x_k) - \frac{3M-2L_2}{12} \norml{\s_k}^3 + \frac{1}{2} \norml{\s_k}^2 \epsilon_1,  
		\end{align*}
		where  (i) follows from \Cref{Hessian_square_bound}, (ii) follows from Lemma 4 in \cite{Nesterov2006} and the fact that $\hat{ \x}_{k+1} = \argmin_{\s \triangleq \x - \x_k} \nabla f(\x_k)^T\s     + \frac{1}{2} \s^T  \mathbf{H}_k \s + \frac{M}{6} \norml{\s}^3$, and (iii) follows from \Cref{assump_2}.
		
		Next, we prove \cref{i_momentum_7}. Note that \begin{align}
	     	 \hat{ \x}_{k+1} = \argmin_{\s \triangleq \x - \x_k} \nabla f(\x_k)^T\s     + \frac{1}{2} \s^T \mathbf{H}_k \s + \frac{M}{6} \norml{\s}^3. \label{i_m_2}
	     \end{align}
	     By the first-order optimality condition, we obtain that
	     \begin{align}
	     	 \nabla f(\x_k) + \mathbf{H}_k \s_k  + \frac{M}{2} \s_k \norml{\s_k} = 0. \label{i_m_1}
	     \end{align}
	     Then, we further obtain that
	     \begin{align*}
	     	 \normlarge{\nabla f(\hat{ \x}_{k+1})} &\numequ{i} \normlarge{\nabla f(\hat{ \x}_{k+1}) -\nabla f(\x_k) - \mathbf{H}_k \s_k  - \frac{M}{2} \s_k \norml{\s_k}  } \\
	     	 &\leqslant \normlarge{\nabla f(\hat{ \x}_{k+1}) -\nabla f(\x_k) - \mathbf{H}_k \s_k} + \frac{M}{2} \norml{\s_k}^2 \\
	     	 &\leqslant   \normlarge{\nabla f(\hat{ \x}_{k+1}) -\nabla f(\x_k) - \nabla^2 f(\x_{k}) \s_k} + \norml{(\nabla^2 f(\x_{k}) - \mathbf{H}_k) \s_k}+ \frac{M}{2} \norml{\s_k}^2 \\
	     	 &\numleqslant{ii} \frac{L_2 + M}{2}\norml{\hat{ \x}_{k+1} - \x_k}^2 + \epsilon_1 \norml{\hat{ \x}_{k+1} - \x_k},
	     \end{align*}
	     where (i) follows from  \cref{i_m_1}, and (ii) follows from \Cref{Hessian_square_bound}, \Cref{assump_2}, and the fact that  $\s_k \triangleq \hat{ \x}_{k+1} - \x_k$. 
	     
	    Next, we prove \cref{i_momentum_14}.  By \cref{i_m_2} and Proposition 1 in \cite{Nesterov2006}, we obtain that
	    \begin{align} \label{i_m_3}
	       \mathbf{H}_k  \succcurlyeq -\frac{M}{2} \norml{\hat{ \x}_{k+1} - \x_k} \mathbf{I}.
	    \end{align} 
	    Then, we further obtain that
	    \begin{align*}
	    	\lambda_{\min} (\nabla^2 f(\x_{k+1})) &\numgeq{i}   \lambda_{\min} ( \mathbf{H}_k )  - \norml{ \nabla^2 f(\x_{k+1})- \mathbf{H}_k } \\
	    	&\numgeq{ii} -\frac{M}{2} \norml{\hat{ \x}_{k+1} - \x_k} - \norml{ \nabla^2 f(\x_{k+1}) -\nabla^2 f(\x_{k})  } - \norml{\nabla^2 f(\x_{k})   - \mathbf{H}_k} \\
	    	&\numgeq{iii}  -\frac{M}{2} \norml{\hat{ \x}_{k+1} - \x_k} - L_2\norml{\hat{ \x}_{k+1} - \x_k} - \epsilon_1\\
	    	&= -\frac{M + 2L_2}{2} \norml{\hat{ \x}_{k+1} - \x_k} - \epsilon_1,
	    \end{align*}
	    where (i) follows from Wely's inequality, (ii) follows from \cref{i_m_3} and (iii) follows from \Cref{assump_2} and the fact that $\nabla^2 f$ is $L_2$-Lipschitz.
	    
	   Next, we prove \cref{i_momentum_8} by contradiction. Suppose for every $i \in \{0, \cdots, k\} $ it holds that 
	   \begin{align}
	   	\norml{\hat{\x}_{i+1}- \x_{i}}  >  \epsilon_1. \label{im_1}
	   \end{align}
	  Then, \cref{im0} further implies that
	   \begin{align*}
	   f(\hat{\x}_{i+1}) - f(\x_i) &\leqslant - \frac{3M-2L_2}{12} \norml{\hat{ \x}_{i+1} - \x_i}^3 + \frac{1}{2} \norml{\hat{ \x}_{i+1} - \x_i}^2 \epsilon_1 \\
	   	&\numleqslant{i} - \left(\frac{3M-2L_2 -6 }{12}   \right) \norml{\hat{ \x}_{i+1} - \x_i}^3, 
	   \end{align*}
	   where (i) follows from \cref{im_1}. Therefore, we have 
	   \begin{align}
	   	\left(\frac{3M-2L_2 -6}{12}  \right) \norml{\hat{ \x}_{i+1} - \x_i}^3 \leqslant  f(\x_i) -  f(\hat{\x}_{i+1}) \numleqslant{i} f(\x_i) -  f(\x_{i+1}), \label{II_I_7}
	   \end{align}
	   where (i) follows from the definition of $\x_{k+1}$.
	   Summing up \cref{II_I_7} over $i$ from $0$ to $k$, we obtain that
	   \begin{align}
	   	 \sum_{i=0}^{k} \left(\frac{3M-2L_2 -6}{12}  \right) \norml{\hat{ \x}_{i+1} - \x_i}^3 \leqslant f(\x_0) - f^{\star}. \label{im_2}
	   \end{align}
	   Combining \cref{im_2} with the fact that $ \norml{\hat{\x}_{i+1}- \x_{i}}  >  \epsilon_1$ for  $i \in \{0, \cdots, k\} $ and $ M  > 2L_2/3 + 2$, we have
	   \begin{align*}
	   	 k \left(\frac{3M-2L_2 -6}{12}  \right)  \epsilon_1^3 \leqslant \sum_{i=0}^{k} \left(\frac{3M-2L_2-6 }{12}  \right) \norml{\hat{ \x}_{i+1} - \x_i}^3 \leqslant f(\x_0) - f^{\star},
	   \end{align*}
	   which gives 
	   \begin{align*}
	   	k \leqslant \left( \frac{12}{3M-2L_2 -6} \right) \frac{  f(\x_0) - f^{\star}} { \epsilon_1^3 }.
	   \end{align*}
	   This contradicts with our assumption that $k > \left( \frac{12}{3M-2L_2 -6} \right) \frac{  f(\x_0) - f^{\star}} { \epsilon_1^3 }$. Therefore, there must  exist an integer $k_0 \in \{0, \cdots k\}$ such that
	   \begin{align}
	   \norml{\hat{\x}_{k_0+1}- \x_{k_0}}  \leqslant  \epsilon_1,  
	   \end{align} 
	   and the proof is complete.
    	\end{proof}

\subsection{Proof of \Cref{momentum_6}}  \label{pf_lemmas_4}

	\begin{proof}
		For $i \geqslant 1$, note that 
		\begin{align*}
		\norml{ \hat{\x}_{i+1} - \hat{\x}_i} &\numleqslant{i} \norml{ \hat{\x}_{i+1} -  \x_i} + \norml{\x_i - \hat{\x}_i} \\ 
		&\numleqslant{ii} \norml{ \hat{\x}_{i+1} -  \x_i} + \beta_{i} \norml{\hat{\x}_{i}-\hat{\x}_{i-1}} \\
		&\numleqslant{iii} \norml{ \hat{\x}_{i+1} -  \x_i} + \rho \norml{\hat{\x}_{i}-\hat{\x}_{i-1}},  \numberthis \label{momentum_recursive}
		\end{align*}
		where (i) follows from the triangle inequality, (ii) follows from \cref{li_transfer}, and (iii) follows from the fact that $\beta_{k+1} \leqslant \rho$ for all $k \geqslant 0$.
		
		Recursively applying \cref{momentum_recursive}, we obtain that
		\begin{align*}
		\norml{ \hat{\x}_{k+1} - \hat{\x}_k}  &\leqslant  \rho^{k}\norml{ \hat{\x}_{1} -  \hat{\x}_{0}} +   \sum_{i = 1}^{k} \rho^{k-i} \norml{ \hat{\x}_{i+1} -  \x_i} \\
		&\numleqslant{i}  \sum_{i = 0}^{k} \rho^{k-i} \norml{ \hat{\x}_{i+1} -  \x_i}    \numberthis \label{momentum_1},
		\end{align*}
		where (i) follows because $\hat{\x}_0 = \x_0$ in \Cref{Adaptive_version}. Note that \cref{momentum_1} is also true for $k=0$. Then, by \cref{momentum_1}, we further obtain that
		\begin{align*}
		\norml{ \hat{\x}_{k+1} - \hat{\x}_k} &\leqslant  \sum_{i = 0}^{k} \rho^{k-i} \norml{ \hat{\x}_{i+1} -  \x_i} \\
		&\leqslant \max_{i \in \{0,\cdots,k\}}\norml{ \hat{\x}_{i+1} -  \x_i} \sum_{i = 0}^{k} \rho^{k-i}\\
		&\numleqslant{i}\max_{i \in \{0,\cdots,k\}}\norml{ \hat{\x}_{i+1} -  \x_i}  \frac{1}{1 - \rho} \numberthis \label{II_2} \\
		&\numleqslant{ii} \frac{1}{1 - \rho}\left( \frac{ f(\x_0) - f^\star }{\gamma}\right)^{1/3},
		\end{align*}
		where (i) follows from the fact that $\rho <1$ and (ii) follows from \cref{momentum_4}. The proof of \Cref{momentum_6} is complete.
	\end{proof}
 
%
%
%
\subsection{Proof of \Cref{li_lemma_local}} \label{pf_lemmas_6}
\begin{proof}  
	To prove item $(i)$, it suffices to show that $\{f(\x_{k}) \}_{k \geqslant 0 }$ is a decreasing sequence with a lower bound. By \Cref{Assumption_1}, $f$ is bounded below. Thus, $\{f(\x_{k}) \}_{k \geqslant 0 }$ is bounded below. Also, note that
	\begin{align}
	f(\x_{k+1}) \numleqslant{i} f(\hat{\x}_{k+1}) \numleqslant{ii}  f( \x_{k})  \label{li_6}
	\end{align}
	where (i) follows from \cref{li_0}, and (ii) follows from \cref{momentum_12}. Thus, $\{f(\x_{k}) \}_{k \geqslant 0 }$ is a decreasing sequence and is bounded below, which further imply that $\{f(\x_{k}) \}_{k \geqslant 0 }$ converges. We denote the corresponding limit as $v$.
	
	To prove item $(ii)$, note that
	\begin{align*}
	\lim\limits_{k \rightarrow \infty}  \norml{\x_{k+1} - \x_{k}} &\leqslant  \lim\limits_{k \rightarrow \infty}  \norml{\x_{k+1} - \hat{\x}_{k+1}} + \norml{\hat{\x}_{k+1} - \x_{k} } \\
	&\numleqslant{i} \lim\limits_{k \rightarrow \infty}  \beta_{k+1} \norml{\hat{\x}_{k+1} -\hat{ \x}_{k} }  + \norml{\hat{\x}_{k+1} - \x_{k} }\\
	&\numleqslant{ii}  \lim\limits_{k \rightarrow \infty}  \norml{\hat{\x}_{k+1} - \x_{k} } \left(\norml{\hat{\x}_{k+1} -\hat{ \x}_{k} }  + 1 \right)\\ 
	&\numleqslant{iii}  \lim\limits_{k \rightarrow \infty}  \norml{\hat{\x}_{k+1} - \x_{k} } \left( \frac{1}{1 - \rho}\left( \frac{ f(\x_0) - f^\star }{\gamma}\right)^{1/3}+ 1 \right)\\ &\numequ{iv} 0 \numberthis \label{li_5}
	\end{align*}
	where (i) follows from \cref{li_transfer},  (ii) follows from \cref{li_1}, which implies that $\beta_{k+1} \leqslant  \norml{\hat{\x}_{k+1} - \x_{k} }$, (iii) follows from \Cref{momentum_6} and (iv) follows from \cref{momentum_4}, which implies that
	\begin{align}
	\lim\limits_{k \rightarrow \infty}  \norml{\hat{\x}_{k+1} - \x_{k} } = 0.\label{li_7}
	\end{align}  
	Then, we conclude that $\lim\limits_{k \rightarrow \infty}  \norml{\x_{k+1} - \x_{k}} = 0$.

	To prove item $(iii)$, note that \cref{li_6} implies that  $\x_{k} \in  \mathcal{L}(f(\x_0)) $ for all $k$. By the assumption that  $\mathcal{L}(f(\x_k))$ is bounded for some $k \geqslant 0$, we conclude that $\{\x_{k}\}_{k \geqslant 0} $ is also bounded. 
	
	To prove item $(iv)$, note that the Bolzano-Weierstarss theorem and item $(iii)$ of \Cref{li_lemma_local} imply that $\{\x_{k}\}_{k \geqslant 0} $ has a convergent subsequence. Also, the set of its accumulation points $\mathbf{\bar{\mathcal{X}}}$ is bounded.  Moreover, for every accumulation point  $\bar{\x}$, by \cref{li_8,li_9,li_7}, we obtain that
	\begin{align*}
	\norml{\nabla f(\bar{\x})} &\leqslant  \limsup\limits_{k \rightarrow \infty}\norml{\nabla f(\x_{k+1})}   \\
	&\leqslant \limsup\limits_{k \rightarrow \infty}  \frac{L_2}{2} \norml{\hat{\x}_{k+1} - \x_k}^2 \left(1 + \frac{ L_1 }{1 - \rho}\left( \frac{ f(\x_0) - f^\star }{\gamma}\right)^{1/3}  \right)  = 0,
	\end{align*}
	and 
	\begin{align*}
	\lambda_{\min}\left({\nabla^2 f(\bar{\x}) }\right) &\geqslant	\liminf\limits_{k \rightarrow \infty}	\lambda_{\min}\left({\nabla^2 f(\x_{k+1}) }\right)\\
	&\geqslant  	\liminf\limits_{k \rightarrow \infty} -\norml{\hat{\x}_{k+1} - \x_k} \left( \frac{M+2L_2}{2} +   \frac{ L_2 }{1 - \rho}\left( \frac{ f(\x_0) - f^\star }{\gamma}\right)^{1/3} \right) = 0.
	\end{align*}
	Thus,  we conclude that $\nabla f(\bar{\x})=0, \nabla^2 f(\bar{\x})\succcurlyeq 0$. Furthermore, item $(i)$ of \Cref{li_lemma_local} implies that ${f(\x_k)}_{k \geqslant 0}$ converges to its limit $v$.
	
\end{proof}
\subsection{Proof of \Cref{II_i_summable_lemma}} \label{pf_lemmas_5}

	\begin{proof}

		
		Following the proof similar to that of \Cref{momentum_6}, one can show that \cref{II_2} also holds for the inexact algorithm, i.e., 
		\begin{align}
				\norml{ \hat{\x}_{k+1} - \hat{\x}_k} &\leqslant  \max_{i \in \{0,\cdots,k\}}\norml{ \hat{\x}_{i+1} -  \x_i}  \frac{1}{1 - \rho}. \label{II_I_2}
		\end{align}
		
		Then, it suffices to bound $\norml{ \hat{\x}_{i+1} -  \x_i}$. Suppose that the inexact variant of CRm2 terminates at iteration $k$. By the termination criterion, we have
		\begin{align}
		\norml{\hat{\x}_{i+1} -  \x_i} &> \epsilon_1 \quad  \text{ for } \quad  0 \leqslant i \leqslant k-1, \label{II_I_3}
		\end{align}
		and
		\begin{align} 
			\norml{\hat{\x}_{k+1} -  \x_k} &\leqslant \epsilon_1 \leqslant 1 \label{II_I_5}.
		\end{align} 
		
		 For $0 \leqslant i \leqslant k-1$, \cref{im0} implies that
		\begin{align*}
				f(\hat{\x}_{i+1}) - f(\x_i) &\leqslant - \frac{3M-2L_2}{12} \norml{\hat{ \x}_{i+1} - \x_i}^3 + \frac{1}{2} \norml{\hat{ \x}_{i+1} - \x_i}^2  \epsilon_1 \\
				&\numleqslant{i}  - \left(\frac{3M-2L_2 - 6}{12} \right)\norml{\hat{ \x}_{i+1} - \x_i}^3 \numberthis \label{II_I_4}
		\end{align*}
		where (i) follows from \cref{II_I_3}. Summing \cref{II_I_4} over $i$ from $0$ to $k-1$, we obtain that
		\begin{align}
			 \sum_{i=0}^{k-1} \norml{\hat{ \x}_{i+1} - \x_i}^3 \leqslant \left( \frac{3M-2L_2 - 6}{12}   \right)^{-1}(f(\x_0) - f^{\star}),
		\end{align}
		which further implies that
		\begin{align}
			\max_{i \in \{0,\cdots,k-1 \}}\norml{ \hat{\x}_{i+1} -  \x_i} \leqslant \left( \frac{3M-2L_2 - 6}{12}  \right)^{-1/3}(f(\x_0) - f^{\star})^{1/3}. \label{II_I_6}
		\end{align}
		 
		 Combining \cref{II_I_5,II_I_6,II_I_2}, we obtain the statement of \Cref{II_i_summable_lemma}.
	\end{proof}

\end{document}